\documentclass{article}
\usepackage[margin=1in,footskip=0.25in]{geometry}
\usepackage{bm}
\usepackage{amsmath,amsthm,thmtools,thm-restate}
\usepackage[numbers]{natbib}
\numberwithin{equation}{section}
\usepackage{amssymb}
\usepackage{graphicx}
\usepackage{xcolor}
\usepackage{subfig}
\usepackage{hyperref}

\usepackage{mathrsfs}

\newcommand{\mr}{\mathrm}
\newcommand{\UH}{\mathsf{H}}   
\newcommand{\UT}{\mathsf{T}}   

\renewcommand{\Re}{\mr{Re}}   
\renewcommand{\Im}{\mr{Im}}   

\newcommand{\Mydef}{\overset{  \scriptscriptstyle \Delta  }{=}}

\declaretheorem{lemma}
\declaretheorem{proposition}
\declaretheorem{corollary}
\declaretheorem{definition}
\declaretheorem{remark}
\declaretheorem{fact}
\declaretheorem[numbered=unless unique]{theorem}

\newcommand{\E}{\mathbb{E}}
\newcommand{\hatE}{\hat{\E}}
\renewcommand{\P}{\mathbb{P}}
\newcommand{\Tr}{\mathsf{Tr}}

\newcommand*\diff{\mathop{}\!\mathrm{d}}
\newcommand{\gauss}[2]{\mathcal{N}\left( #1,#2 \right)}
\newcommand{\cgauss}[2]{\mathcal{CN}\left( #1,#2 \right)}

\newcommand{\explain}[2]{\overset{\text{\tiny{#1}}}{#2}}
\newcommand{\diag}[1]{\text{Diag}\left(#1\right)}

\newcommand{\gpdf}[1][\sigma]{\psi_{#1}}
\newcommand{\MI}[2]{\mathbf{I} \left( {{#1}; {#2}}  \right)}
\newcommand{\Ent}[1]{\mathbf{H} \left( {#1} \right)}
\newcommand{\CEnt}[2]{\mathbf{H} \left( {#1} \left| \right. {#2} \right)}
\newcommand{\ip}[2]{\langle {#1}, {#2} \rangle}
\newcommand{\weakthr}{\delta_{\tiny{\mathsf{weak}}}}
\newcommand{\critthr}{\delta_{c}}
\newcommand{\Nr}{\mathscr{U}}
\newcommand{\Dr}{\mathscr{L}}
\newcommand{\CalF}{\mathcal{F}}
\newcommand{\CalFhat}{\hat{\mathcal{F}}}
\newcommand{\Texp}[2]{\mathsf{TExp} \left(#1,#2\right)}
\newcommand{\ZTexp}[2]{Z_{{\footnotesize\mathsf{TExp}}} \left(#1,#2\right)}
\newcommand{\VTexp}[2]{\sigma^2_{{\footnotesize\mathsf{TExp}}} \left(#1,#2\right)}
\newcommand{\MTexp}[3]{\mu_{{\footnotesize\mathsf{TExp}}}^{(#3)} \left(#1,#2\right)}
\newcommand{\Exp}[1]{\mathsf{Exp} \left( #1 \right)}
\newcommand{\dGamma}[2]{\mathsf{Gamma} \left( #1, #2 \right)}
\newcommand{\Beta}[2]{\mathsf{Beta} \left( #1,#2 \right) }
\newcommand{\Wis}[2]{\mathsf{Wis}\left( #1,#2 \right)}
\newcommand{\TWis}[3]{\mathsf{TWis} \left(#1,#2,#3\right)}
\newcommand{\ZTWis}[3]{Z_{{\footnotesize\mathsf{TWis}}} \left(#1,#2,#3\right)}
\newcommand{\VTWis}[3]{\bm \Sigma_{{\footnotesize\mathsf{TWis}}} \left(#1,#2,#3\right)}
\newcommand{\MTWis}[4]{\mu_{{\footnotesize\mathsf{TExp}}}^{(#4)} \left(#1,#2,#3\right)}
\renewcommand{\vec}[1]{\mathsf{Vec} \left( #1 \right)}
\renewcommand{\i}{{i\mkern1mu}}
\newcommand{\unif}[1]{\text{Unif} \left( #1 \right)}
\newcommand{\GUE}[1]{\mathsf{GUE}(#1)}

\begin{document}
	\title{Information Theoretic Limits \\ for Phase Retrieval  with \\
		Subsampled Haar Sensing Matrices}
	\author{Rishabh Dudeja, Junjie Ma, Arian Maleki \\ Department of Statistics, Columbia University}
	
	\maketitle
	\begin{abstract} We study information theoretic limits of recovering an unknown $n$ dimensional, complex signal vector $\bm x_\star$ with unit norm from $m$ magnitude-only measurements of the form $y_i = |(\bm A \bm x_\star)_i|^2, \; i = 1,2 \dots , m$, where $\bm A$ is the sensing matrix. This is known as the Phase Retrieval problem and models practical imaging systems where measuring the phase of the observations is difficult. Since in a number of applications, the sensing matrix has orthogonal columns, we model the sensing matrix as a subsampled Haar matrix formed by picking $n$ columns of a uniformly random $m \times m$ unitary matrix.  We study this problem in the high dimensional asymptotic regime, where $m,n \rightarrow \infty$, while $m/n \rightarrow \delta$ with $\delta$ being a fixed number, and show that if $m < (2-o_n(1))\cdot n$, then \emph{any estimator} is asymptotically orthogonal to the true signal vector $\bm x_\star$. This lower bound is sharp since when $m > (2+o_n(1)) \cdot n $, estimators that achieve a non trivial asymptotic correlation with the signal vector are known from previous works \citep{dudeja2019rigorous, ma2019spectral}.
	\end{abstract}
	
	\newpage
	\tableofcontents
	\newpage
	\section{Introduction}
	\subsection{Motivation}
	Phase retrieval refers to the problem of recovering an unknown signal $\bm x_\star \in \mathbb C^n$ from $m$ phaseless measurements of the form:
	\begin{align}
	y_i & =  |(\bm A \bm x_\star)_i|^2,
	\end{align} 
	where the matrix $\bm A \in \mathbb C^{m \times n}$ denotes the sensing matrix.
	
	The Phase Retrieval problem is a model for practical imaging systems where it is infeasible to obtain the phase of the measurements. This is true for imaging systems that arise in a wide range of fields, such as electron microscopy, crystallography, astronomy, and optical imaging \citep{shechtman2015phase}. Due to physical considerations in these imaging systems, the sensing matrix $\bm A$ is often a variant of the Discrete Fourier Transform (DFT) matrix \citep{millane1990phase}. Designing recovery algorithms for the phase retrieval problem with such structured sensing matrices is a challenging problem and no algorithm with rigorous recovery guarantees is known.  In light of this disconnect between theory and practice, \citet{candes2015phase} have proposed using a class of random sensing matrices called coded diffraction patterns (CDPs). They leverage the fact that in some applications, it is possible to randomize the sensing mechanism by introducing random masks. CDP sensing matrices can be expressed as follows:
	\begin{align*}
	\bm A_{CDP} & = \begin{bmatrix} \bm F_n \bm P_1 \\ \bm F_n \bm P_2 \\ \vdots \\ \bm F_n \bm P_L \end{bmatrix}.
	\end{align*}
	In the above display, $\bm F_n$ is the $n \times n$ DFT matrix. The matrices $\bm P_l \;, l = 1,2 \dots L$ are diagonal matrices representing random masks: 
	\begin{align*}
	\bm P_l = \diag{e^{\i \theta_{1,l}}, e^{\i \theta_{2,l}}  \dots e^{\i \theta_{n,l}}}.
	\end{align*}
	where the random phases $\theta_{j,l}, \; j = 1, 2 \dots n, \; l = 1,2 \dots L$ are sampled from some distribution. The authors analyzed an SDP based recovery algorithm and showed that it performs a successful recovery with $m = O(n \log(n))$ measurements. This was the first estimator for the phase retrieval problem with rigorous guarantees in a physically realistic setup. 
	
	The result of \citet{candes2015phase} sheds light on the order of measurements required to solve the phase retrieval problem. However, it is desirable to complement these results with a sharp asymptotic analysis of the problem where $n,m \rightarrow \infty$ and $m = n \delta$. Such sharp results allow us to compare different algorithms (which may require the same order of phaseless measurements) based on the value of $\delta$ at which they succeed. Furthermore, asymptotic results tell us the information theoretically minimum value of $\delta$ required and whether existing algorithms achieve it. 
	
	For gaussian sensing matrices this framework has been successfully used for designing optimal spectral estimators \citep{yuelu_spectral,yuelu_optimal,mondelli2019fundamental}, to compute the precise Bayes risk in Bayesian phase retrieval \citep{barbier2019optimal}, to analyze the performance of message passing algorithms for phase retrieval \citep{ma2019optimization}, and to analyze the performance of convex relaxations \citep{dhifallah2018phase,abbasi2019universality}. The results for Gaussian sensing matrices are universal in the sense that the same results hold for any sensing matrix with i.i.d. entries \citep{barbier2019optimal}. 
	
	Despite the recent progress in the asypptotic analysis of the phase retrieval problem with Gaussian matrices, the study of CDP sensing matrices have resisted the asymptotic analysis. Simulation results have also shown that the results derived for Gaussian sensing matrices are not consistent with the empirical results obtained for CDP sensing matrices. For example, in our previous work \citep{ma2019spectral, dudeja2019rigorous} we observed that for noiseless phase retrieval, theory based on Gaussian sensing matrices predicts that optimally designed spectral methods should achieve a non-trivial correlation with the planted vector as soon as $\delta > 1$. In practice, however, for CDP matrices, spectral estimators don't achieve a non-trivial correlation unless $\delta > 2$. 
	
	In order to better model CDP sensing matrices,  our previous work \citep{ma2019spectral} proposed that since CDP sensing matrices are column orthogonal, i.e. $\bm A_{CDP}^\UH \bm A_{CDP} = \bm I_n$, it is natural to study the phase retrieval problem with subsampled Haar sensing matrices that are generated as follows: 
	\begin{eqnarray}\label{subsampled_haar_sensing}
	\bm A = \bm H_{m} \bm S_{m,n}, \; \bm H_m \sim \unif{\mathbb U_m}, \bm S_{m,n}  = \begin{bmatrix} \bm I_{n} \\ \bm 0_{m-n,n} \end{bmatrix}.
	\end{eqnarray}
	In the above display, $\unif{\mathbb U(m)}$ denotes the Haar measure on the space of $m \times m$ unitary matrices. Note that the matrix $\bm A$ is formed by simply selecting the first $n$ columns of a $m \times m$ Haar matrix. For the subsampled Haar model, we \citep{ma2019spectral,dudeja2019rigorous} derived theoretical predictions for the performance of certain message passing algorithms  and a class of spectral methods for subsampled Haar sensing matrices. We found that these predictions have excellent agreement with the empirical performance of these estimators for CDP sensing matrices. In particular, they correctly predict the failure of spectral methods for $\delta < 2$. This observation underscores the potential of the subsampled Haar model to provide theoretical guidance for design and comparision of estimators for the CDP model. A similar universality phenomena has been observed in the context of noiseless compressed sensing by \citet{donoho2009observed,monajemi2013deterministic} and \citet{oymak2014case} regarding the performance of LASSO, and by \citet{abbara2019universality} regarding the performance of AMP algorithms. 
\subsection{Our main contribution: } Given the empirical evidence that practical sensing matrices like the CDP sensing matrices behave similarly to subsampled Haar matrices \citep{ma2019spectral}, we study information theoretic lower bounds for Phase Retrieval problem (see \eqref{PR_problem}) with subsampled Haar sensing matrices generated according to \eqref{subsampled_haar_sensing} in the presence of measurement noise:
	\begin{eqnarray}
	y_i  = m |(\bm A \bm x_\star)_i|^2 + \sigma  \epsilon_i, \; i = 1,2 \dots m.
	\label{PR_problem}
	\end{eqnarray}
	We assume the measurement noise $\epsilon_i \explain{i.i.d.}{\sim} \gauss{0}{1}$.  We  assume that the signal vector is a uniformly random unit vector: $\bm x_\star \sim \unif{\mathbb S_{n-1}}$. This is intended to model situations where we don't have apriori knowledge regarding the structure of the signal (for example it is not known if it is sparse). Moreover, as we will clarify in a moment, this is the least favorable prior for this problem.  The particular choice of scaling in \eqref{PR_problem} has been made so that the rescaled noiseless measurement $\sqrt{m} \cdot (|\bm A \bm x_\star|)_i$ satisfies $m \cdot \E |\bm A \bm x_\star|_i^2 = 1$. We adopt the sharp high dimensional asymptotic framework for our analysis and study a sequence of phase retrieval problems with $m,n \rightarrow \infty $, such that the oversampling ratio $\delta \Mydef m/n$ remains fixed. Our main result is summarized in the following theorem:
	\begin{theorem}\label{main_result} For any $\delta < 2$ and for any noise level $\sigma > 0$, the Bayes risk satisfies:
		\begin{align*}
		  \lim_{\substack{m, n \rightarrow \infty \\ \frac{m}{n} = \delta}} \;  \E \left\| \bm x_\star \bm x_\star^\UH - \E \left[ \bm x_\star \bm x_\star^\UH \big| \bm y, \bm A \right] \right\|^2 & \rightarrow 1,
		\end{align*}
	where $\| \cdot \|$ denotes the Frobenius norm.
	\end{theorem}
	We interpret the above result in two ways. First note that according to this theorem, for $\delta < 2$, the Bayes risk is the same as the risk of the estimator $\hat{\bm x} = \bm 0$. Hence it is information theoretically impossible for any estimator to have a better performance than the trivial estimator $\hat{\bm x} = \bm 0$. Second, we can make the above point more explicit as follows: Let $\hat{\bm x}(\bm A, \bm y)$ be any estimator for $\bm x_\star$ and let $r\geq 0$ be an arbitrary constant.  By the optimality of the Bayes estimator, we have, 
	\begin{align*}
	\min_{r \geq 0} \E \left\| \bm x_\star \bm x_\star^\UH - r\frac{\hat{\bm x}(\bm A, \bm y) \hat{\bm x}(\bm A, \bm y)^\UH}{\|\hat{\bm x}(\bm A, \bm y)\|^2} \right\|^2 & \geq \E \left\| \bm x_\star \bm x_\star^\UH - \E \left[ \bm x_\star \bm x_\star^\UH \big| \bm y, \bm A \right] \right\|^2.
	\end{align*}
	Taking $m,n \rightarrow \infty$ and some simple algebraic manipulations give us  the following conclusion. When $\delta < 2$, then for any estimator $\hat{\bm x}(\bm A,\bm y)$ we have,
	\begin{align*}
	 \lim_{\substack{m, n \rightarrow \infty \\ \frac{m}{n} = \delta}} \E \left[ \frac{|\bm x_\star^\UH \hat{\bm x}(\bm A,\bm y)|^2 }{\|\hat{\bm x}(\bm A,\bm y)\|^2} \right] & = 0.
	\end{align*}
	That is, when $\delta < 2$, Theorem \ref{main_result} provides \emph{an impossibility result}: any estimator is asymptotically orthogonal to the signal vector $\bm x_\star$. This result complements our previous results \citep{ma2019spectral,dudeja2019rigorous} which showed that the optimally designed spectral estimator is orthogonal to the signal vector in this regime. Moreover, these papers also provide the \emph{achievability result} and exhibit estimators which achieve a strictly positive correlation with the signal vector when $\delta > 2$ and $\sigma = 0$. Hence, the sharp threshold for achieving a non-trivial correlation with the signal vector (called the weak recovery threshold in the literature) is $\weakthr = 2$ for phase retrieval with subsampled Haar sensing matrix and vanishing measurement noise. This also shows that the uniform prior on $\bm x_\star$ as the least favorable prior in the following sense: The achievability results of these papers actually hold for an \emph{arbitrary} signal vector (not necessarily drawn from a prior distribution). Consequently, when $\delta > 2$, for \emph{any} prior on the signal vector, the Bayes risk for noiseless phase retrieval is non-trivial ($<1$).  Hence the uniform prior maximizes the $\delta$ threshold below which the Bayes risk is trivial and hence is least favorable. 
	\subsection{Related Work}
	There are a large number of results on the Phase Retrieval problem with varying assumptions on the sensing matrix, studying different classes of estimators under different analysis frameworks. In this section, we summarize a few important and representative results from each category. 
	\paragraph{Finite Sample Analyses} A large number of estimators are known to   achieve the rate optimal sample complexity of $O(n)$ or the near optimal sample complexity of $O(n \cdot  \text{poly}(\log(n)))$.  The earliest such estimator is PhaseLift SDP relaxation proposed by \citet{candes2013phaselift} and the recently proposed PhaseMax LP relaxation proposed by \citet{goldstein2018phasemax} and \citet{bahmani2017phase}. More recently approaches based on non-convex optimization have been analyzed. This includes an alternating minimization approach due to \citet{netrapalli2013phase} and a gradient descent based algorithm due to \citet{candes2015wirtinger}. Moreover, the finite sample analyses is flexible enough to extend to CDP sensing matrices. We refer the reader to \citep{candes2015phase} for the analysis of PhaseLift for CDP matrices and to \citet{candes2015wirtinger} and \citet{qu2017convolutional} for the analysis of non-convex optimization approach for CDP matrices and random circulant sensing matrices respectively.
	\paragraph{Sharp Asymptotic Analyses for Gaussian Sensing Matrices}  Finite sample analysis, though flexible, lacks the resolution to compare the performance of various estimators which achieve the optimal sample complexity of $O(n)$ measurements. Consequently, recent years have seen a number of works which provide an analysis in the high dimensional asymptotic framework where $m,n \rightarrow \infty$ and $m/n = \delta$. \citet{yuelu_spectral} analyzed a class of spectral estimators in this asymptotic framework for Gaussian sensing matrices. Their analyses was leveraged by \citet{mondelli2019fundamental} and \citet{yuelu_optimal} to design spectral estimators with optimal performance. Convex relaxation based approaches, such as PhaseLift and PhaseMax have also been analyzed in this framework for Gaussian sensing matrices \citep{dhifallah2018phase,abbasi2019universality}. 
	\paragraph{Sharp Information Theoretic Lower Bounds for Gaussian Sensing Matrices} \citet{mondelli2019fundamental} showed that the weak recovery threshold for Gaussian sensing matrices was $\weakthr = 1$. Our work builds on their proof technique. Extending their proof to subsampled Haar sensing matrices requires several new ideas to handle the underlying dependence structure in the sensing matrix. \citet{barbier2019optimal} have used interpolation methods to obtain expressions for the asymptotic Bayes risk for estimating generalized linear models which includes real valued phase retrieval with Gaussian sensing matrices as a special case. In particular, their results recover the results of \citet{mondelli2019fundamental} as a special case and also shed light on the minimum mean square error achievable above the weak recovery threshold. This work also shows that the expression of the Bayes risk for any sensing matrix with i.i.d. entries with some mild moment assumptions is the same as the Bayes risk for Gaussian sensing matrices. 
	\paragraph{Universality Results} In the high dimensional asymptotic, results proved for Gaussian sensing matrices are often universal in the sense that they hold for any sensing matrix whose rows are independent and identically distributed under some mild moment assumptions. Examples of such results include universality of the Bayes risk \citep{barbier2019optimal}, performance of AMP algorithms \citep{bayati2015universality} and performance of convex relaxations \citep{abbasi2019universality}. We emphasize that neither the subsampled Haar sensing matrix nor the CDP sensing matrix have independent rows and hence these universality results don't apply to these ensembles. In fact, we will see that the weak  recovery threshold is different for subsampled Haar sensing matrices and i.i.d. Gaussian sensing matrices. Empirically, it has been observed that CDP sensing matrices behave like subsampled Haar sensing matrices \citep{ma2019spectral}.  
	\paragraph{Sharp Asymptotic Analyses for Non-i.i.d. Sensing Matrices} Results for non-i.i.d. sensing matrices in the high dimensional asymptotic framework are very limited. \citet{thrampoulidis2015isotropically} provide an analysis of the generalized Lasso estimator for compressed sensing using uniformly random row orthogonal matrices using the Convex Gaussian Minmax Theorem (CGMT) framework. Some achievability results for Phase Retrieval with subsampled Haar sensing matrices include the analysis of a class of message passing algorithms (designed to solve certain eigenvalue problems) due to \citet{ma2019spectral} and a class of spectral estimators due to \citet{dudeja2019rigorous}. We note that the non-rigorous replica method can be used to derive conjectures for the asymptotic Bayes risk for a large class of unitarily invariant sensing matrices which includes sub-sampled Haar sensing matrices as a special case. The application of the replica method to unitarily invariant ensembles was pioneered in a sequence of papers by \citet{takeda2006analysis}, \citet{takeda2007statistical} and \citet{kabashima2008inference}. We refer the reader to \citet{reeves2017additivity} for a recent derivation of these conjectures.  To the best of our knowledge, these conjectures have not been rigorously proved except in a few special cases, none of which cover the sub-sampled Haar sensing matrix. The only rigorous result about sharp information theoretic lower bounds for non-i.i.d. sensing matrices is due to \citet{barbier2018mutual} who provide the expression for the limiting Bayes risk for a certain class of sensing matrices. The class of sensing matrices they consider are formed by a product of independent matrices each consisting of i.i.d. entries. This is significantly different from the sub-sampled Haar sensing model which we consider here. Moreover, the sensing problem they study is the real linear sensing problem and not the phase retrieval problem that we study here. Lastly we note that the non-rigorous replica method has also been used to analyze convex relaxation methods like LASSO \citep{vehkapera2016analysis,wen2016sparse} for unitarily invariant sensing matrices. 
	\paragraph{Proof Techniques} Our proof builds on the techniques of \citet{mondelli2019fundamental}: namely relating the Bayes risk to the Mutual Information and bounding the Mutual Information by the $\chi^2$ divergence. However, unlike in the case of Gaussian sensing matrices, the evaluation of $\chi^2$ divergence for our model is non trivial due to the dependence in the entries of the subsampled Haar sensing matrix. In our model, understanding the asymptotics of the $\chi^2$ divergence reduces to understanding the asymptotics of a pair of high dimensional integrals defined on $\mathbb{S}^{m-1}$ and $\mathbb{S}^{m-1} \times \mathbb{S}^{m-1}$ (see Lemma \ref{introduce_main_integrals_lemma}) which we accomplish using Large Deviation techniques. These integrals are related to low rank Harish-Chandra-Itsker-Zuber (HCIZ) integrals studied by \citet{guionnet2005fourier} and our analysis is inspired by their approach. More specifically, our analysis of these integrals is based on the classical approach of \citet{chaganty1993strong} for obtaining strong large deviation results (i.e. results characterizing the leading exponential order as well as the second order polynomial factors in large deviation quantities of interest) using change of measure and local central limit theorems. 
	
	\subsection{Notation}
	\paragraph{Notations for common sets} We use $\mathbb R, \mathbb C$ to denote real numbers and complex numbers respectively.  $\mathbb R^n$ and $\mathbb C^n$ denote the $n$ dimensional real and complex vector spaces respectively. $\mathbb S^{n-1} \subset \mathbb C^n$ is the set of complex $n$-dimensional vectors with unit norm. $\mathbb N$ denotes the set of natural numbers and $[n]$ denotes the set $\{1,2 \dots ,n\}$. The set of $m \times n$ real matrices is denoted by $\mathbb R^{m \times n}$ and the set of $m \times n$ complex matrices is denoted by $\mathbb C^{m \times n}$. The set of all $m \times m$ unitary matrices is denoted by $\mathbb U(m)$.
	\paragraph{Notations for complex analytic aspects} For a complex number $z \in \mathbb C$, $\overline{z}, \Re(z), \Im(z), |z|$ denote the complex conjugate, real part of, imaginary part of and modulus of $z$ respectively. We use $\i$ to denote $\sqrt{-1}$. 
	\paragraph{Notation for Asymptotic Analysis} We say a sequence $f(n)$ is $o(n)$ if $f(n)/n \rightarrow 0$ as $n \rightarrow \infty$. We use the generic constant $C$ to refer to a positive finite constant that does not depend on $m,n$. This constant may change from line to line and may depend on the noise level $\sigma$ and the sampling ratio $\delta$ unless stated otherwise. If this constant depends on any other parameters we will make this dependence explicit: For example,  $C(\epsilon)$ denotes a positive, finite constant depending on some parameter $\epsilon$, the noise level $\sigma$ and possibly the sampling ratio $\delta$ but independent of $m,n$.
	\paragraph{Notations for linear algebraic aspects} For vectors and matrices $\|\cdot\|$ denotes the $\ell_2$ and the Frobenius norm respectively. For complex matrices $\|\cdot\|_{op}$ denotes the operator norm. For a matrix $\bm A \in \mathbb C^{m\times n}$, $\bm A^\UH$ denotes the conjugate transpose of $\bm A$. $\Tr(\cdot)$ denotes the trace of a square matrix. For vectors $\bm a, \bm b \in \mathbb C^n$, the inner product $\ip{\bm a}{\bm b}$ is defined as $\bm a^\UH \bm b$. For matrices $\bm A, \bm B \in \mathbb C^{m \times n}$ the inner product $\ip{\bm A}{\bm B}$ is defined as $\Tr(\bm A^\UH \bm B)$. For a Hermitian matrix $\bm A$, we denote the largest and smallest eigenvalue of $\bm A$ by $\lambda_{\max}(\bm A)$ and $\lambda_{\min}(\bm A)$. For a $2 \times 2$ Hermitian matrix $\bm A$ we define the $\vec{\cdot}$ operation by:
	\begin{align*}
	\vec{\bm A} & = \begin{bmatrix} A_{11} \\ A_{22} \\ \Re(A_{12}) \\ \Im (A_{12}) \end{bmatrix}.
	\end{align*}
	Finally, we use $\bm e_1, \bm e_2 \dots , \bm e_n$ to denote the standard basis vectors in $\mathbb R^n$. 
	\paragraph{Notations for special distributions} $\gauss{\mu}{\sigma^2}$ denotes the (real) Gaussian distribution with mean $\mu$ and variance $\sigma^2$. $\gpdf[\sigma]$ denotes the probability density function of $\gauss{0}{\sigma^2}$: 
	\begin{align*}
	\gpdf[\sigma](x) & = \frac{1}{\sqrt{2\pi \sigma^2}} e^{-\frac{x^2}{2 \sigma^2}}.
	\end{align*}
	The (real) multivariate Gaussian distribution with mean $\bm \mu$ and variance $\bm \Sigma$ is denoted by $\gauss{\bm \mu}{\bm \Sigma}$. 
	We say a complex random variable $Z$ is standard complex gaussian distributed, denoted by $Z \sim \cgauss{0}{1}$ if $\Re(Z)$ and $\Im(Z)$ are i.i.d. $\gauss{0}{\frac{1}{2}}$. We say a complex $n$-dimensional random vector $\bm Z \sim \cgauss{\bm 0}{\bm I_n}$ if each entry $Z_i \explain{i.i.d.}{\sim} \cgauss{0}{1}$. A random matrix $\bm W$ is a $\GUE{n}$ random matrix if it is a Hermitian $n \times n$ random matrix whose entries are sampled as follows:
	\begin{align*}
	W_{ii}  \sim \gauss{0}{1}\; \forall \; i \in [n], \; W_{ij} \sim \cgauss{0}{1}\; \forall \; j < i, \; W_{ji} = \overline{W}_{ij}  \;  \forall \; j > i.
	\end{align*}
	$\Exp{\lambda}$ denotes the exponential distribution with parameter $\lambda$ which has the pdf:
	\begin{align*}
	f(x) & = \begin{cases} \lambda e^{-\lambda x} &: x \geq 0 \\ 0 &: x < 0 \end{cases}.
	\end{align*}
	$\dGamma{\alpha}{\beta}$ denotes the Gamma distribution with shape parameter $\alpha$ and rate parameter $\beta$ and has the pdf:
	\begin{align*}
	f(x; \alpha,\beta) & = \begin{cases} \frac{\beta^\alpha}{\Gamma(\alpha)} x^{\alpha -1} e^{-\beta x}&: x \geq 0 \\ 0  &: x < 0 \end{cases}.
	\end{align*}
	$\Beta{\alpha}{\beta}$ denotes the Beta distribution with shape parameters $\alpha,\beta \geq 0$ which has the pdf:
	\begin{align*}
	f(x; \alpha,\beta) & = \begin{cases} \frac{\Gamma(\alpha+\beta)}{\Gamma(\alpha)\Gamma(\beta)} \cdot x^{\alpha - 1} (1-x)^{\beta -1} &: x \in [0,1] \\ 0 &: x \not \in [0,1] \end{cases}. 
	\end{align*}
	Let $\bm g_1, \bm g_2 \dots , \bm g_n \explain{i.i.d.}{\sim} \cgauss{\bm 0}{\bm I_p}$. Then the matrix $\bm S = \sum_{i=1}^n \bm g_i \bm g_i^\UH$ has a complex Wishart distribution with parameters $n,p$ denoted by $\Wis{n}{p}$. The complex Wishart distribution is supported on positive definite Hermitian matrices and has the pdf:
	\begin{align*}
	f(\bm S; n,p) & = \frac{\det(\bm S)^{n-p} \cdot  e^{-\Tr(\bm S)}}{\pi^{\frac{p(p-1)}{2}}. \prod_{j=1}^p (n-j)!}.
	\end{align*}
	The distribution $\unif{\mathbb U_m}$ denotes the uniform (Haar) probability measure on $\mathbb U(m)$.
	\paragraph{Notation for other probabilistic aspects}For an event $\mathcal{E}$, $\mathbf{1}_\mathcal{E}$ denotes the indicator function of $\mathcal{E}$. We will use $p(\bm y)$ to denote the density of the measurements $\bm y$ with respect to the Lebesgue measure. Likewise $p(\bm y | \bm A)$ and $p(\bm y|\bm A, \bm x)$ denote the conditional density of the measurements $\bm y$ given the measurement matrix $\bm A$ and the conditional density of the $\bm y$ given the measurement matrix $\bm A$ and the signal vector $\bm x$ respectively. For two random variables $B_1,B_2$ we will use $B_1 \explain{d}{=} B_2$ to denote the claim that $B_1$ and $B_2$ have the same distribution. Finally we denote convergence in probability using $\explain{P}{\rightarrow}$. 
	\paragraph{Notation for Information Theoretic Aspects} For random variables $A_1, A_2, \dots A_k$, we denote the entropy of $(A_1 \dots ,A_k)$ by $\Ent{A_1,A_2 \dots , A_k}$. If $(A_1,A_2 \dots A_k)$ have a joint density $p(a_1,a_2 \dots , a_k)$ with respect to the Lebesgue measure, this is defined as:
	\begin{align*}
	\Ent{A_{1:k}} & = - \int_{\mathbb R^k} p(a_{1:k}) \ln p(a_{1:k}) \diff a_{1:k}.
	\end{align*}
	Let $B_1, B_2 \dots B_l$ be another collection of random variables. We denote the conditional entropy of $(A_1,\dots A_k)$ given $(B_1,  \dots ,B_l)$  by $\CEnt{A_1,A_2 \dots , A_k}{B_1, B_2 \dots, B_l}$. When the conditional distribution of $(A_1, A_2 \dots, A_k)$ given $(B_1, B_2 \dots, B_l)$ has a density $p(a_1,a_2 \dots , a_k | b_1, b_2 \dots ,b_l)$ (with respect to Lebesgue measure) and $(B_1, B_2 \dots , B_l)$ has a marginal density $p(b_1, b_2 \dots, b_l)$ (with respect to Lebesgue measure), then $\CEnt{A_{1:k}}{B_{1:l}}$ is given by:
	\begin{align*}
	\CEnt{A_1, \dots A_k}{B_1 \dots , B_l}& = \int_{\mathbb R^l} p(b_1 , \dots b_l) \int_{\mathbb R^k} p(a_1,  \dots a_k | b_1 \dots , b_l) \ln p(a_1, \dots a_k | b_1,  \dots b_l) \diff a_{1:k} \diff b_{1:l}.
	\end{align*}
	The mutual information between $A_1,A_2, \dots A_k$ and $B_1, B_2 \dots B_l$ is denoted $\MI{A_1, \dots  A_k}{B_1,\dots  B_l}$ and is defined by the following equivalent formulae:
	\begin{align*}
	\MI{A_1, \dots  A_k}{B_1,\dots  B_l} & \Mydef \Ent{A_1, \dots  A_k} - \CEnt{A_1, \dots  A_k}{B_1,  \dots B_l}\\ &= \Ent{B_1, \dots  B_l} - \CEnt{B_1,  \dots  B_l}{A_1,  \dots  A_k}.
	\end{align*}
	\paragraph{The random variable Y} We reserve the random variable $Y$ to denote the random variable with one of the following two special distributions: 
	\begin{enumerate}
		\item $Y$ can be sampled from the empirical distribution of the phase retrieval measurements:
		\begin{align*}
		Y & \sim \frac{1}{m} \sum_{i=1}^m \delta_{y_i}.
		\end{align*}
		For any $f: \mathbb R \mapsto \mathbb R$, we define $\hatE f(Y)$ to be the expectation of $f(Y)$ with respect to the empirical measure of the measurements:
		\begin{eqnarray}
		\label{hatE_notation}
		\hatE f(Y)  \Mydef \frac{1}{m} \sum_{i=1}^m f(y_i).
		\end{eqnarray}
		\item Alternatively the distribution of $Y$ can be given by $Y = |Z|^2 + \sigma \epsilon$ where $Z \sim \cgauss{0}{1}, \; \epsilon \sim \gauss{0}{1}$. 	For any $f: \mathbb R \mapsto \mathbb R$, we define $\E f(Y)$ denotes the expectation of $f(Y)$ with respect to this measure, that is, $\E f(Y) = \E f(|Z|^2 + \sigma \epsilon)$. This special distribution is important to us because we will see that for a large class of test functions $f$, $\hatE f(Y) \rightarrow \E f(Y)$ as $m \rightarrow \infty$.
	\end{enumerate}
\subsection{Outline}
The remainder of this paper is dedicated to proving Theorem \ref{main_result}. The proof consists of different steps which are split into various sections as follows:
\begin{description} 
	\item [] In Section \ref{MI_and_Bayes_Risk}, we relate the Bayes risk to the Mutual Information for the phase retrieval problem with a small amount of side information and show that if the mutual information is $o(m)$, then the asymptotic Bayes risk is trivial. Hence, our focus shifts to showing that when $\delta < 2$, the Mutual information is $o(m)$. We then bound the mutual information by the $\chi^2$ divergence. Understanding the $\chi^2$ divergence in the Phase retrieval model requires us to understand the asymptotics of two high dimensional integrals denoted by $\Dr$ and $\Nr$ on $\mathbb S^{m-1}$ and $\mathbb S^{m-1} \times \mathbb S^{m-1}$ respectively. 
	\item [] In Section \ref{section_integral_asymptotics}, we study the asymptotics of the integrals $\Nr,\Dr$ by change of measure techniques and local central limit theorems.
	\item [] In Section \ref{section_stochastic_laplace}, we use a stochastic version of the Laplace Principle along with the asymptotics of $\Nr,\Dr$ to understand the asymptotics of the $\chi^2$ divergence. This results in a explicit condition on the sampling ratio $\delta$ and the noise level $\sigma$ which guarantees that the mutual information is $o(m)$ and hence the Bayes risk is trivial. 
	\item [] In Section \ref{low_noise_section}, we simplify the condition on $\delta,\sigma$ obtained previously in the low noise limit $\sigma \rightarrow 0$. 
\end{description}
	\section{Mutual Information and Bayes Risk}
	\label{MI_and_Bayes_Risk}
	We first relate the Bayes risk to the mutual information in the phase retrieval problem where one observes a small amount of side information about the signal vector $\bm x_\star$. The amount of side information we observe will be controlled by a parameter $\Delta >0$ which will be a constant independent of $n,m$. The side information we observe will be linear gaussian measurements of the matrix $\bm x_\star \bm x_\star^\UH$. More precisely, for $i=1,2 \dots \lfloor\Delta \cdot  m \rfloor$ we observe a measurement pair $(\bm w_i, z_i)$ drawn from the following model:
	\begin{eqnarray}
	\bm w_i \explain{i.i.d.}{\sim} \GUE{n}, \; z_i \explain{i.i.d.}{\sim} \gauss{ \ip{\bm w_i}{\bm x_\star \bm x_\star^\UH}}{1} \; \forall \; i=1,2, \dots \lfloor\Delta \cdot  m \rfloor.
	\label{side_info_model}
	\end{eqnarray} 
	We collect all the side information measurements $z_i$'s in a vector $\bm z \in \mathbb R^{\lfloor \Delta m \rfloor}$. We denote the collection of the GUE sensing matrices by $\bm W \Mydef \{\bm w_1, \bm w_2 \dots \bm w_{\lfloor\Delta m \rfloor}\}$. 
	The following proposition establishes the connection between the Bayes Risk and $\MI{\bm y, \bm z}{\bm A, \bm W}$. 
	\begin{proposition}
		\label{bayes_risk_to_MI}
		Suppose that there exists a constant $\Delta >0$ (independent of $m,n$) such that the mutual information $\MI{\bm y, \bm z}{\bm A, \bm W} = o(m)$. Then we have,
		\begin{align*}
		\lim_{\substack{m,n \rightarrow \infty \\ m = n \delta}}\E_{\bm x_\star, \bm y, \bm A} \| \bm x_\star \bm x^\UH_\star - \E[ \bm x_\star \bm x_\star^\UH | \bm y, \bm A] \|^2 & = 1.
		\end{align*}
	\end{proposition}
	In light of Proposition \ref{bayes_risk_to_MI}, in order to show that the Bayes risk is trivial, it is sufficient to show that an upper bound on the mutual information is $o(m)$. We will use the second moment upper bound (or the $\chi^2$-divergence uppper bound) on mutual information. This upper bound was utilized by \citet{mondelli2019fundamental} for determining the weak recovery threshold for Gaussian sensing matrices. In our setup, the result of these authors can be stated as:
	\begin{align*}
	\MI{\bm y, \bm z}{\bm A, \bm W} & \leq \E_{\bm y, \bm z} \left[  \frac{ \E_{\bm A, \bm W}p^2(\bm y, \bm z | \bm A, \bm W)}{p^2(\bm y, \bm z)} \right] - 1.
	\end{align*}
	It is also well known that the second moment upper bound is sensitive to bad but rare events that can cause the upper bound to blow up. In order to exclude these bad events we will use a conditional version of the above bound which is stated below. A similar result was used by \citet{reeves2019all} in the context of a linear regression problem. The proof of this result is given in Appendix \ref{second_moment_conditional_proof}.
	\begin{lemma}
		\label{second_moment_conditional} Let $\mathcal{E}_m$ be any sequence of  events depending only $\bm y$.
		We have, 
		\begin{align*}
		\MI{\bm y, \bm z}{\bm A, \bm W} & \leq \left(  \int_{\mathcal{E}_m} \frac{\E_{\bm A, \bm W}   p^2(\bm y, \bm z|\bm A, \bm W)}{ p(\bm y, \bm z)} \diff \bm y \diff \bm z -1 \right) + C \cdot m \cdot \sqrt{\P(\mathcal{E}_m^c)}.
		\end{align*}
		In the above display, $C \geq 0$ denotes a finite constant depending only on $\delta, \Delta, \sigma^2$.
	\end{lemma}
	The following lemma simplifies the upper bound on $\MI{\bm y, \bm z}{\bm A, \bm W}$. 	For any $\bm y \in \mathbb R^m$ and any positive semidefinite $2 \times 2$ Hermitian matrix $\bm Q$, introduce the functions: 
	\begin{eqnarray}
	\Nr(\bm y, \bm Q) & \Mydef & \E \left[ \prod_{i=1}^m \gpdf(y_i - | G_{1i}|^2 ) \gpdf(y_i -  |G_{2i}|^2 ) \bigg| \bm G ^\UH \bm G = {m \bm Q} \right], \label{Nr_integral_def_eq}\\
	\Dr(\bm y) &\Mydef& \E \left[ \prod_{i=1}^m \gpdf(y_i - | G_{1i}|^2 ) \bigg| \|\bm G_1\|^2 =m \right], \label{Dr_integral_def_eq}
	\end{eqnarray}
	where $\bm G_1, \bm G_2  \explain{i.i.d.}{\sim} \cgauss{\bm 0}{\bm I_m}$ and the matrix $\bm G = [\bm G_1 \; \bm G_2]$. We emphasize that in the definitions of $\Nr(\bm y,\bm Q)$ and $\Dr(\bm y)$, the measurements $\bm y$ are fixed, and the expectation is only with respect to the Gaussian matrix $\bm G$. 
	\begin{lemma}
		\label{introduce_main_integrals_lemma}
		We have, 
		\begin{align*}
		\int_{\mathcal{E}_m} \frac{\E_{\bm A, \bm W}   p^2(\bm y, \bm z|\bm A, \bm W)}{ p(\bm y, \bm z)} \diff \bm y \diff \bm z  & = \frac{2}{n-1} \E_{\bm y} \left[ \frac{ \int_0^1 \Nr \left(\bm y, \begin{bmatrix} 1 & q \\ q & 1 \end{bmatrix}\right) \cdot \frac{q \cdot (1-q^2)^{n-2}}{(1-q^2/2)^{\lfloor\Delta m \rfloor}} \diff q}{\Dr^2(\bm y)} \cdot  \mathbf{1}_{\mathcal{E}_m}\right] \\
		& \leq \frac{2}{n-1} \E_{\bm y} \left[ \frac{ \int_0^1 \Nr \left(\bm y, \begin{bmatrix} 1 & q \\ q & 1 \end{bmatrix}\right) \cdot \frac{q \cdot (1-q^2)^{n-2}}{(1-q^2/2)^{\Delta m}} \diff q}{\Dr^2(\bm y)} \cdot  \mathbf{1}_{\mathcal{E}_m}\right]
		\end{align*}
	\end{lemma}
	\begin{proof}
		We have, 
		\begin{align*}
		&\E_{\bm A, \bm W} p^2(\bm y, \bm z|\bm A, \bm W) = \E_{\bm A, \bm W, \bm x, \bm x^\prime} p(\bm y, \bm z| \bm A, \bm W, \bm x) p(\bm y, \bm z| \bm A , \bm W, \bm x^\prime) \\
		& = \E_{\bm A, \bm W, \bm x, \bm x^\prime} \left[ \prod_{i=1}^m \gpdf(y_i - m |\ip{\bm a_i}{ \bm x}|^2 ) \gpdf(y_i - m |\ip{\bm a_i}{ \bm x^\prime}|^2 ) \prod_{i=1}^{\lfloor\Delta m \rfloor} \gpdf[1](z_i - \ip{\bm w_i}{\bm x \bm x^\UH}) \gpdf[1](z_i - \ip{\bm w_i}{\bm x^\prime \bm x^{\prime\UH}}) \right]
		\end{align*}
		Define the scalar random variable: 
		\begin{align*}
		q = \bm x^\UH\bm x^\prime,
		\end{align*}
		and the associated random matrices: 
		\begin{align*}
		\bm Q = \begin{bmatrix}1 & q \\ \bar{q} & 1 \end{bmatrix}, \; \bm C = \begin{bmatrix} 1 & q \\ 0 & \sqrt{1-|q|^2} \end{bmatrix}
		\end{align*}
		Note that we have $\bm C^\UH \bm C = \bm Q$.
		It is easy to see that,  conditioned on $\bm x, \bm x^\prime$:
		\begin{align*}
		\begin{bmatrix} \ip{\bm w_i}{\bm x \bm x^{\UH}} \\ \ip{\bm w_i}{\bm x^\prime \bm x^{\prime\UH}}  \end{bmatrix} & \explain{i.i.d.}{\sim} \gauss{\bm 0}{\begin{bmatrix} 1 & |q|^2 \\ |q|^2 & 1 \end{bmatrix}}, \\
		\bm A \bm x, \bm A \bm x^\prime & \explain{d}{=}  \bm U_1, q \bm U_1 + \sqrt{ 1-|q|^2} \bm U_2.
		\end{align*}
		In the above display, $\bm U = [\bm U_1 \;  \bm U_2]$ is a uniformly random $m \times 2$ partial unitary matrix. 
		Let $\bm G$ be a $m \times 2$ matrix consisting of $\cgauss{0}{1}$ entries. Then we have,  
		\begin{align*}
		\sqrt{m} \cdot \begin{bmatrix}\bm A \bm x &  \bm A \bm x^\prime \end{bmatrix} & \explain{d}{=} \sqrt{m} \cdot \bm U \bm C \\
		& =  \bm U \bm C {\bm Q}^{-1/2} {(m \cdot \bm Q)}^{1/2} \\
		& \explain{d,(1)}{=} \bm U {(m\bm Q)}^{1/2} \\
		& \explain{d,(2)}{=} \bm G (\bm G^\UH \bm G)^{-1/2}  {(m \cdot \bm Q)}^{1/2} \\
		& \explain{d,(3)}{=}  \bm G (\bm G^\UH \bm G)^{-1/2}  {(m \cdot \bm Q)}^{1/2} | \bm G^\UH \bm G = m{\bm Q} \\
		& = \bm G | \bm G^\UH \bm G = m{\bm Q}.
		\end{align*} 
		In the step marked (1), we used the fact that $\bm C \bm Q^{-1/2}$ is unitary consequently $\bm U \bm C \bm Q^{-1/2} \explain{d}{=} \bm U$. In step (2) we used the well known fact that a uniformly random partial unitary matrix can be realized as $\bm U \explain{d}{=} \bm G (\bm G ^\UH \bm G)^{-1/2}$. In the step marked (3) we used the fact that $\bm G(\bm G^\UH \bm G)^{-1/2}$ is independent of $\bm G^\UH \bm G$, and hence conditioning on the event $\bm G^\UH \bm G$ does not change the distribution of $ \bm G (\bm G^\UH \bm G)^{-1/2}$. Hence we have shown that, conditioned on $\bm x, \bm x^\prime$,  the matrix $\sqrt{m} \cdot [ \bm A \bm x \; \bm A \bm x^\prime]$ has the same distribution as a Gaussian matrix $\bm G$ conditioned on the event $\bm G^\UH \bm G = m \bm Q$:
		\begin{align*}
			\sqrt{m} \cdot \begin{bmatrix}\bm A \bm x &  \bm A \bm x^\prime \end{bmatrix} &\explain{d}{=} \bm G | \bm G^\UH \bm G = m{\bm Q}. 
		\end{align*}
		Hence we have, 
		\begin{align*}
		&\E_{\bm A, \bm W} p^2(\bm y, \bm z|\bm A, \bm W) = \E_{\bm A, \bm W, \bm x, \bm x^\prime} p(\bm y, \bm z| \bm A, \bm W, \bm x) p(\bm y, \bm z| \bm A , \bm W, \bm x^\prime) \\
		&= \E_{q} \left[ \E \left[ \prod_{i=1}^m \gpdf(y_i - | G_{1i}|^2 ) \gpdf(y_i -  |G_{2i}|^2 ) \bigg| \bm G ^\UH \bm G = m {\bm Q} \right] \prod_{i=1}^{\lfloor\Delta m \rfloor} \E_{Z,Z^\prime} \gpdf[1](z_i - Z) \gpdf[1](z_i - |q|^2 Z - \sqrt{1-|q|^4} Z^\prime) \right],
		\end{align*}
		where $Z,Z^\prime \explain{i.i.d.}{\sim} \gauss{0}{1}$. 
		We observe that the conditional expectation:
		\begin{align*}
		\E \left[ \prod_{i=1}^m \gpdf(y_i - | G_{1i}|^2 ) \gpdf(y_i - |G_{2i}|^2 ) \bigg| \bm G ^\UH \bm G = {m\bm Q} \right],
		\end{align*}
		depends on $q$ only via $|q|$. Consequently, we redefine the matrix $\bm Q$ as: 
		\begin{align*}
		\bm Q & = \begin{bmatrix} 1 & |q| \\ |q| & 1 \end{bmatrix}.
		\end{align*}
		The following integral has been evaluated in Lemma \ref{bivariate_gauss_integral} in Appendix \ref{misc_appendix}.
		\begin{align*}
		\E_{Z,Z^\prime} \gpdf[1](z-Z) \gpdf[1](z- |q|^2 Z - \sqrt{1-|q|^4} Z^\prime) & = \frac{1}{4 \pi\sqrt{1-|q|^4/4}} \exp \left( - \frac{z^2}{2(1+|q|^2/2)} \right).
		\end{align*}
		Hence,
		\begin{align*}
		    &\E_{\bm A, \bm W} p^2(\bm y, \bm z|\bm A, \bm W)   \\&\hspace{2cm}= \E_q \left[\E \left[ \prod_{i=1}^m \gpdf(y_i - | G_{1i}|^2 ) \gpdf(y_i -  |G_{2i}|^2 ) \bigg| \bm G ^\UH \bm G = m {\bm Q} \right] \cdot \frac{\exp \left( - \frac{1}{2(1+|q|^2/2)} \cdot \sum_{i=1}^{\lfloor\Delta m \rfloor} z_i^2 \right)}{(4 \pi\sqrt{1-|q|^4/4})^{\lfloor\Delta m\rfloor}} \right].
		\end{align*}
		Next we compute $p(\bm y, \bm z)$. Since $\bm y $ and $\bm z$ are independent, $p(\bm y, \bm z) = p(\bm y) p(\bm z)$. $p(\bm y)$ can be computed by following similar steps as before:
		\begin{align*}
		    p(\bm y) & = \E \left[ \prod_{i=1}^m \gpdf(y_i -  | G_{1i}|^2 ) \bigg| \|\bm G_1\|^2 =m \right].
		\end{align*}
		It is also easy to check that $z_i \explain{i.i.d.}{\sim} \gauss{0}{2}$. Hence:
		\begin{align*}
		p(\bm y, \bm z) & =   \E \left[ \prod_{i=1}^m \gpdf(y_i -  | G_{1i}|^2 ) \bigg| \|\bm G_1\|^2 =m \right] \cdot \prod_{i=1}^{\lfloor\Delta m \rfloor} \gpdf[\sqrt{2}](z_i).
		\end{align*}
		Consequently, introducing the functions:
			\begin{align*}
		\Nr(\bm y, \bm R) & \Mydef \E \left[ \prod_{i=1}^m \gpdf(y_i - | G_{1i}|^2 ) \gpdf(y_i -  |G_{2i}|^2 ) \bigg| \bm G ^\UH \bm G = {m \bm R} \right], \\
		\Dr(\bm y) &\Mydef \E \left[ \prod_{i=1}^m \gpdf(y_i -  | G_{1i}|^2 ) \bigg| \|\bm G_1\|^2 =m  \right].
		\end{align*}
	we obtain,	
	\begin{align*}
		\E_{\bm z} \frac{\E_{\bm A,\bm W} p^2(\bm y, \bm z |\bm A, \bm W)}{p^2(\bm y, \bm z)}  &= \E_{q} \left[ \frac{\Nr \left(\bm y, \begin{bmatrix} 1 & |q| \\ |q| & 1 \end{bmatrix}\right)}{\Dr^2(\bm y)} \cdot \left( \frac{\E_{Z \sim \gauss{0}{2}} \exp\left( \frac{\frac{|q|^2}{2}}{1+\frac{|q|^2}{2}} \cdot \frac{Z^2}{2}\right)}{\sqrt{1-\frac{|q|^4}{4}}} \right)^{\lfloor\Delta m \rfloor} \right] \\ & \explain{(a)}{=}  \E_{q} \left[\frac{\Nr \left(\bm y, \begin{bmatrix} 1 & |q| \\ |q| & 1 \end{bmatrix}\right)}{\Dr^2(\bm y)} \cdot \frac{1}{(1-|q|^2/2)^{\Delta m}} \right].
		\end{align*}
		In the step marked (a), we used the MGF of $\chi^2$ distribution to compute:
		\begin{align*}
		    \E_{Z \sim \gauss{0}{2}} \exp\left( \frac{\frac{|q|^2}{2}}{1+\frac{|q|^2}{2}} \cdot \frac{Z^2}{2}\right) & = \sqrt{\frac{1+|q|^2/2}{1-|q|^2/2}}
		\end{align*}
		Hence we have,
		\begin{align*}
		\int_{\mathcal{E}_m} \frac{\E_{\bm A, \bm W}   p^2(\bm y, \bm z|\bm A, \bm W)}{ p(\bm y, \bm z)} \diff \bm y \diff \bm z  & = \E_{\bm y,|q|} \left[ \frac{\Nr \left(\bm y, \begin{bmatrix} 1 & |q| \\ |q| & 1 \end{bmatrix}\right)}{\Dr^2(\bm y)} \cdot \frac{1}{(1-|q|^2/2)^{\lfloor\Delta m\rfloor}} \cdot  \mathbf{1}_{\mathcal{E}_m}\right]
		\end{align*}
		Next we observe that,
		\begin{align*}
		|q|^2 & \sim \Beta{1}{n-1}.
		\end{align*}
		Utilizing the formula for the pdf of Beta random variables we have, 
		\begin{align*}
		\int_{\mathcal{E}_m} \frac{\E_{\bm A, \bm W}   p^2(\bm y, \bm z|\bm A, \bm W)}{ p(\bm y, \bm z)} \diff \bm y \diff \bm z  & = \frac{1}{n-1} \E_{\bm y} \left[ \frac{ \int_0^1 \Nr \left(\bm y, \begin{bmatrix} 1 & \sqrt{b} \\ \sqrt{b} & 1 \end{bmatrix}\right) \cdot \frac{(1-b)^{n-2}}{(1-b/2)^{\lfloor\Delta m \rfloor}} \diff b}{\Dr^2(\bm y)} \cdot  \mathbf{1}_{\mathcal{E}_m}\right].
		\end{align*}
		Finally making the change of variable $b = q^2$ gives us:
		\begin{align*}
		\int_{\mathcal{E}_m} \frac{\E_{\bm A, \bm W}   p^2(\bm y, \bm z|\bm A, \bm W)}{ p(\bm y, \bm z)} \diff \bm y \diff \bm z & = \frac{2}{n-1} \E_{\bm y} \left[ \frac{ \int_0^1 \Nr \left(\bm y, \begin{bmatrix} 1 & q \\ q & 1 \end{bmatrix}\right) \cdot \frac{q \cdot (1-q^2)^{n-2}}{(1-q^2/2)^{\lfloor\Delta m\rfloor}} \diff q}{\Dr^2(\bm y)} \cdot  \mathbf{1}_{\mathcal{E}_m}\right]. \\
		& \leq \frac{2}{n-1} \E_{\bm y} \left[ \frac{ \int_0^1 \Nr \left(\bm y, \begin{bmatrix} 1 & q \\ q & 1 \end{bmatrix}\right) \cdot \frac{q \cdot (1-q^2)^{n-2}}{(1-q^2/2)^{\Delta m}} \diff q}{\Dr^2(\bm y)} \cdot  \mathbf{1}_{\mathcal{E}_m}\right].
		\end{align*}
	\end{proof}
	\begin{remark} At this point, it is instructive to compare the claim of Lemma \ref{introduce_main_integrals_lemma} to its counterpart from \citep{mondelli2019fundamental}. If $\bm A$ were Gaussian, then, \citet{mondelli2019fundamental} have shown that, 
		\begin{align}
		\int \frac{\E_{\bm A, \bm W}   p^2(\bm y, \bm z|\bm A, \bm W)}{ p(\bm y, \bm z)} \diff \bm y \diff \bm z  & = \frac{2}{n-1} \E_{\bm y} \left[ \frac{ \int_0^1 \Nr_{\mathsf{Gauss}} \left(\bm y, \begin{bmatrix} 1 & q \\ q & 1 \end{bmatrix}\right) \cdot \frac{q \cdot (1-q^2)^{n-2}}{(1-q^2/2)^{\lfloor\Delta m\rfloor}} \diff q}{\Dr^2_{\mathsf{Gauss}}(\bm y)} \cdot  \right],
		\label{MM_second_moment}
		\end{align}
		where the functions $\Nr_{\mathsf{Gauss}}$ and $\Dr_{\mathsf{Gauss}}$ are defined as follows:
		\begin{align*}
		\Nr_{\mathsf{Gauss}}\left(\bm y, \begin{bmatrix} 1 & q \\ q & 1 \end{bmatrix}\right) & \Mydef \E \left[ \prod_{i=1}^m \gpdf(y_i - | G_{1i}|^2 ) \gpdf(y_i -  |q G_{1i} + \sqrt{1-q^2} G_{2i}|^2 )\right], \\
		\Dr_{\mathsf{Gauss}}(\bm y) &\Mydef \E \left[ \prod_{i=1}^m \gpdf(y_i - | G_{1i}|^2 ) \right]
		\end{align*}
		Because the conditioning is absent in the definitions of $\Nr_{\mathsf{Gauss}}$ and $\Dr_{\mathsf{Gauss}}$, one can leverage the independence in $\bm G_1, \bm G_2$ and obtain straightforwardly: 
		\begin{align*}
		\frac{\Nr_{\mathsf{Gauss}}\left(\bm y, \begin{bmatrix} 1 & q \\ q & 1 \end{bmatrix}\right)}{\Dr_{\mathsf{Gauss}}^2(\bm y)} & = \prod_{i=1}^m\frac{\E_{G_1,G_2} \left[ \gpdf(y_i - | G_{1}|^2 ) \gpdf(y_i -  |q G_{1} + \sqrt{1-q^2} G_{2}|^2 )\right]}{\E^2_G \left[ \gpdf(y_i - | G|^2 )\right]}.
		\end{align*}
		Furthermore when the sensing matrix is Gaussian, the observations $y_1,y_2 \dots y_m$ are i.i.d. Let $Y$ be a random variable with the same distribution as $y_i$. The expression in \eqref{MM_second_moment} simplifies significantly:
		\begin{align}
		\int \frac{\E_{\bm A, \bm W}   p^2(\bm y, \bm z|\bm A, \bm W)}{ p(\bm y, \bm z)} \diff \bm y \diff \bm z  & = \frac{2}{n-1}\int_0^1 F_{\mathsf{Gauss}}(q)^m \cdot  \frac{q \cdot (1-q^2)^{n-2}}{(1-q^2/2)^{\lfloor\Delta m\rfloor}} \diff q,
		\label{MM_result}
		\end{align}
		where,
		\begin{align*}
		F_{\mathsf{Gauss}}(q) \Mydef \E_Y \left[ \frac{\E_{G_1,G_2} \left[ \gpdf(Y - | G_{1}|^2 ) \gpdf(Y -  |q G_{1} + \sqrt{1-q^2} G_{2}|^2 )\right]}{\E^2_G \left[ \gpdf(Y - | G|^2 )\right]}\right] . 
		\end{align*}
		\citet{mondelli2019fundamental} analyze the integral in \ref{MM_result} by a straightforward application of the Laplace Principle. Note that this whole approach breaks down in our case because the conditioning in the definition of $\Nr,\Dr$ introduces dependence between the Gaussian random vectors $\bm G_1, \bm G_2$ and their entries. This dependence is a manifestation of the dependence present in a subsampled Haar unitary matrix. 
	\end{remark}
	
	\section{Asymptotic Analysis of $\Dr$ and $\Nr$} \label{section_integral_asymptotics}
	In order to evaluate the upper bound on the mutual information that is given in Lemma \ref{introduce_main_integrals_lemma}, one needs to understand the asymptotic behaviour of the functions $\Dr$ and $\Nr$ introduced in Lemma \ref{introduce_main_integrals_lemma}.
	\subsection{Analysis of $\Dr$}
	Recall that $\Dr(\bm y)$ was defined as:
\begin{align*}
\Dr(\bm y) &\Mydef \E \left[ \prod_{i=1}^m \gpdf(y_i - | G_{1i}|^2 ) \bigg| \|\bm G_1\|^2 =m \right].
\end{align*}
	We can rewrite $\Dr(\bm y)$ as follows:
	\begin{align*}
	\Dr(\bm y) &\Mydef \E \left[ \exp \left(\sum_{i=1}^m \ln \gpdf(y_i - | G_{1i}|^2 ) \right) \bigg| \frac{1}{m} \|\bm G_1\|^2 =1 \right].
	\end{align*}
	The above equation suggests that the asymptotics of $\Dr$ are determined by the large deviation properties of the random variables:
	\begin{align}
	\label{two_rvs_dr}
	\left( \frac{1}{m} \sum_{i=1}^m \ln \gpdf(y_i - | G_{1i}|^2 ),  \frac{1}{m} \|\bm G_1\|^2  \right).
	\end{align}
	Note that the random variables in the display above are a sum of independent random variables.  In our analysis we treat $\bm y$ as a fixed vector in  $\mathbb R^m$ and only leverage the randomness in $\bm G_1$. Consequently, the two random variables in  \eqref{two_rvs_dr} are sums of independent, but not identically distributed random variables. This makes our analysis a bit delicate.  Large deviation theory tells us that the Cramer Transform plays a crucial role in understanding the large deviations of sums of independent random variables. Hence, we define the Tilted Exponential distribution which is the Cramer Transform (or the exponential tilting) of the pair of random variables $(\ln \gpdf(y -|G|^2), |G|^2)$ where $G \sim \cgauss{0}{1}$ and $y \in \mathbb R$ is a fixed scalar below.
	\begin{definition}[The Tilted Exponential Distribution]\label{Texp_definition} The Tilted Exponential distribution with parameters $(\lambda,y)$ denoted by $\Texp{\lambda}{y}$ is the distribution on $[0,\infty)$ with the pdf: 
		\begin{align*}
		f(u) & = \frac{e^{-(1-\lambda)u}\gpdf(u-y)}{\ZTexp{\lambda}{y}},
		\end{align*}
		where, $\ZTexp{\lambda}{y}$ denotes the normalizing constant: 
		\begin{align*}
		\ZTexp{\lambda}{y} & \Mydef \int_0^\infty e^{-(1-\lambda)u}\gpdf(u-y)  \diff u = \E_{E \sim \Exp{1}} e^{\lambda E} \gpdf(E-y).
		\end{align*}
		We also denote the variance of $\Texp{\lambda}{y}$ by $\VTexp{\lambda}{y}$. 
	\end{definition}
	In Appendix \ref{texp_properties_appendix} we prove some essential properties of the Tilted Exponential distribution which will be useful in our analysis.
	
	The analysis of $\Dr(\bm y)$ uses two standard techniques from large deviation theory: performing an exponential change of measure and then applying a central limit theorem under the tilted measure.
	
	The following lemma is a change of measure result that we use in our analysis. In order to state it we first introduce some notation. Fix any $\bm y \in \mathbb R^m$ and any $ \lambda \in \mathbb R$. Let $u_1,u_2 \dots u_m$  be independent non-negative random variables with $u_i \sim \Texp{\lambda}{y_i}$. Let $F_{\lambda, \bm y}$ be the density of the random variable $\sum_{i=1}^m u_i$.
	\begin{lemma} \label{Dr_representation_formula}
		For any $\lambda \in \mathbb{R}, \bm y \in \mathbb{R}^m$  we have, 
		\begin{align*}
		\Dr(\bm y) & =  \frac{(m-1)! \cdot  e^{m(1-\lambda)} \cdot F_{\lambda, \bm y}(m)}{m^{m-1}} \cdot \prod_{i=1}^m \ZTexp{\lambda}{y_i}.
		\end{align*} 
		In the above display, $F_{\lambda, \bm y}$ is the density of the random variable $\sum_{i=1}^m u_i$ where the random variables $u_i$ are sampled independently with marginal distribution $u_i \sim \Texp{\lambda}{y_i}$.
	\end{lemma}
	\begin{proof}
		Define the random variables:
		\begin{align*}
		U  = \sum_{i=1}^m u_i, \; T  = \sum_{i=1}^m \ln \gpdf(y_i-u_i).
		\end{align*}
		Consider two possible probability distributions for $U$ and $T$: 
		\begin{enumerate}
			\item $u_i$ are i.i.d. $\Exp{1}$. Let $G(u,t)$ be the joint pdf of $U$ and $T$ in this setup. 
			\item $u_i$ are sampled independently from $\Texp{\lambda}{y_i}$ defined in the statement of the lemma. Let $F_{\lambda, \bm y}(u,t)$ denote the joint pdf of $U,T$ in this setup.  
		\end{enumerate}
		We can compute $F_{\lambda, \bm y}(u,t)$ in terms of $G(u,t)$ in the following way: 
		\begin{align}
		F_{\lambda, \bm y}(u,t) & = \frac{\exp(t+\lambda u)}{\prod_{i=1}^m \ZTexp{\lambda}{\bm y_i}} \cdot G(u,t).
		\label{F2Geq}
		\end{align}
		Let $G(t|u)$ denote the conditional density of $T$ given $U = u$ and $G(u)$ denote the marginal density of $U$ under Setup 1.  Analogously define $F_{\lambda, \bm y}(t|u)$ and $F_{\lambda, \bm y}(u)$. We can then compute $\Dr(\bm y)$ as follows: 
		\begin{align*}
		\Dr(\bm y) & = \E_{\bm g \sim \cgauss{\bm 0}{\bm I_m}} \left[ \exp \left( \sum_{i=1}^m \ln \phi_{\sigma}(y_i -  |g_{i}|^2) \right) \bigg| \|\bm g\|^2 = m  \right] \\ 
		& \explain{}{=} e^{-m  \lambda} \E \left[ \exp \left( \sum_{i=1}^m \ln \phi_{\sigma}(y_i -  |g_{i}|^2) + \lambda m \right) \bigg| \|\bm g\|^2 = m \right] \\
		& \explain{(a)}{=} e^{-m  \lambda } \int e^{t+ m  \lambda} G(t|m) \diff t \\
		& = \frac{e^{-m  \lambda}}{G(m )} \int e^{t + m  \lambda } G(m,t) \diff t.
		\end{align*}
		In the step marked (a), we used the fact that if $ G \sim \cgauss{0}{1}$, then $|G|^2 \sim \Exp{1}$. Next, appealing to \eqref{F2Geq}, we obtain:
		\begin{align*}
		\Dr(\bm y) & \explain{(b)}{=} \frac{e^{-m\lambda} \prod_{i=1}^m \ZTexp{\lambda}{y_i} }{G(m)} \int F_{\lambda, \bm y}(m,t) \diff t \\
		& = \frac{F_{\lambda, \bm y}(m)e^{-m\lambda}}{G(m)} \cdot \prod_{i=1}^m \ZTexp{\lambda}{y_i} \\
		& \explain{(c)}{=} \frac{(m-1)! \cdot  e^{m(1-\lambda)} \cdot F_{\lambda, \bm y}(m)}{(m)^{m-1}} \cdot \prod_{i=1}^m \ZTexp{\lambda}{y_i}. 
		\end{align*}
		The equality marked (b) follows from \eqref{F2Geq}. 
		In the step (c), we used the fact that under Setup 1, $U$ is a sum of exponential random variables and hence $U \sim \dGamma{m}{1}$. Therefore the density of the Gamma distribution can be used to evaluate $G(m)$. This proves the claim of the lemma. 
	\end{proof}
Our next step will be to develop the asymptotics of $F_{\lambda, \bm y}$ by means of a local CLT. Note that in Lemma \ref{Dr_representation_formula}, $\lambda \in \mathbb R$ was arbitrary. We will set $\lambda = \hat{\lambda}_1(\sigma)$, where
\begin{align}
\hat{\lambda}_1( \sigma) & \Mydef \arg\max_{\lambda \in \mathbb{R}} \left( \lambda  - \hat\E_Y  
\ln \E_{E \sim \Exp{1}} e^{\lambda E} \gpdf(E-Y) \right).
\label{xi_1_problem_def_argmax}
\end{align}
We also define,
\begin{align}
\hat\Xi_1(\sigma) &\Mydef  \max_{\lambda \in \mathbb{R}} \left( \lambda  - \hatE_Y  
\ln \E_{E \sim \Exp{1}} e^{\lambda E} \gpdf(E-Y) \right).
\label{xi_1_problem_def_max}
\end{align}
The notation $\hatE$ in the above display, has been introduced in \eqref{hatE_notation}. Note that the above quantities depend on the vector $\bm y$, but we have not made the dependence explicit in the notation. The intuition for setting $\lambda$ in this way is that the first order stationarity condition applied to the concave variational problem in \eqref{xi_1_problem_def_argmax} and \eqref{xi_1_problem_def_max} give us:
\begin{align*}
\frac{1}{m} \sum_{i=1}^m \frac{\E E e^{\hat\lambda_1(\sigma) E} \gpdf(E-y_i)}{\ZTexp{\hat{\lambda}_1(\sigma)}{y_i}} = 1 \implies    \E \left[  \sum_{i=1}^m u_i \right] = m.
\end{align*}
Consequently, by the central limit theorem, we expect that, $m^{-\frac{1}{2}} \cdot ((\sum_i u_i)-m)$ is close to a Gaussian distribution with variance:
\begin{align}
\label{hatv_def}
\hat{v}(\sigma) & \Mydef \frac{1}{m} \sum_{i=1}^m \VTexp{\hat{\lambda}_1(\sigma)}{y_i} =    \hatE_Y \VTexp{\hat{\lambda}_1(\sigma)}{Y}.
\end{align}
Hence, $F_{\hat{\lambda}_1(\sigma),\bm y}$, which is the density of $\sum_{i=1}^m u_i$ can be approximated by the density of $\gauss{m}{m \hat{v}(\sigma)}$.
\begin{align*}
F_{\hat{\lambda}_1(\sigma),\bm y}(m) & \approx \gpdf[ m \cdot \hat{v}(\sigma)](0)  = \frac{1}{\sqrt{2\pi \hat{v}(\sigma)\cdot m}}.
\end{align*}
This intuition is made rigorous in the following proposition.
\begin{restatable}[A Local Central Limit Theorem]{proposition}{localcltdenom} \label{local_clt_denom} Suppose that there exists a constant $0<K< \infty$, such that, 
	\begin{align*}
	|\hat{\lambda}_1(\sigma)|  \leq K, \; \hatE_Y (|Y|+|Y|^2 + |Y|^3) \leq K, \; \frac{1}{K} \leq \hat{v}(\sigma) \leq K.
	\end{align*}
	Then, there exists a constant $C(K)$, depending only on $K$ such that we have the following asymptotic expansion for $F_{\hat{\lambda}_1(\sigma), \bm y}(m)$: 
	\begin{align*}
	\left|{F}_{\hat{\lambda}_1(\sigma),\bm y}(m) - \frac{1}{\sqrt{2\pi \hat{v}(\sigma)\cdot m}}\right|  \leq \frac{C(K)\ln(m)}{m},
	\end{align*}
where $\hat{\lambda}_1(\sigma)$ and $\hat{v}(\sigma)$ have been defined in \eqref{xi_1_problem_def_argmax} and \eqref{hatv_def}.
\end{restatable}
There is a large literature on local central limit theorems. We refer the reader to \citet{bhattacharya1986normal} for a textbook treatment of these results. We are unable to use the statements of local central limit theorems already available in the literature because we require a local central limit theorem for sums of independent but not identically distributed random variables and we further require some control on the error of normal approximation. The proof of Proposition \ref{local_clt_denom} can be found in Appendix \ref{appendix_localcltdenom}. It closely follows the classical proofs of local central limit theorems based on characteristic functions (see for example \citet[Chapter 16]{feller2008introduction}).

We conclude our analysis of $\Dr$ with the following result which is a straightforward corollary of the change of measure result given in Lemma \ref{Dr_representation_formula} and the local central limit theorem in Proposition \ref{local_clt_denom}. 
\begin{corollary}[Lower Bound on $\Dr$] 
	\label{local_clt_denom_corollary} Under the assumptions of Proposition \ref{local_clt_denom}, there exists $M(K) \in \mathbb{N}$ depending only on $K$ such that,
	\begin{align*}
	\Dr(\bm y) & \geq \frac{1}{2\sqrt{K}} \exp \left(  -m \cdot  \hat\Xi_1(\sigma) \right), \; \forall m \geq M(K),
	\end{align*}
	where the function $\hat\Xi_1(\sigma)$ has been defined in \eqref{xi_1_problem_def_max}.
\end{corollary}	
\begin{proof}
	Applying Lemma \ref{Dr_representation_formula} with $\hat{\lambda} = \hat{\lambda}_1(\sigma)$, we have, 
	\begin{align*}
	\Dr(\bm y) & =  \frac{(m-1)! \cdot  e^{m(1-\hat{\lambda})} \cdot F_{\hat{\lambda}, \bm y}(m)}{(m)^{m-1}} \cdot \prod_{i=1}^m \ZTexp{\hat{\lambda}}{y_i}. 
	\end{align*}
	Note by Stirling's Approximation, we have: 
	\begin{align*}
	\frac{(m-1)!}{m^{m-1}} & \geq \sqrt{2\pi (m-1)} \cdot e^{-(m-1)} \cdot \left( 1 - \frac{1}{m} \right)^{m-1} \\& \explain{(a)}{\geq} 
	\sqrt{2\pi (m-1)} \cdot e^{-m}
	\end{align*}
	In the step marked (a), we used the bound $1-x \geq e^{-\frac{x}{1-x}}, \; x \in (0,1)$.
From Proposition \ref{local_clt_denom}, we conclude that there exists a constant $M(K)$, depending only on $K$, such that,
\begin{align*}
F_{\hat{\lambda}, \bm y} (m) & \geq \frac{1}{\sqrt{2\pi \hat{v}(\sigma) m}} - \frac{C(K)\ln(m)}{m}
.\end{align*} 
In particular, this means that there exists $M(K)$ depending only on $K$ such that,
\begin{align*}
F_{\hat{\lambda}, \bm y} (m) & \geq \frac{1}{2\sqrt{2\pi K m}} \; \forall m \geq M(K).
\end{align*}
This gives us the lower bound: 
\begin{align*}
\Dr(\bm y) & \geq \frac{1}{2 \sqrt{2K}} \cdot e^{- m\hat{\lambda} } \cdot \prod_{i=1}^m \ZTexp{\hat{\lambda}}{y_i}, \; \forall m \geq M(K) \\
& = \frac{1}{2\sqrt{K}} \exp \left(  - m \max_{\lambda \in \mathbb{R}} \left( \lambda r - \frac{1}{m} \sum_{i=1}^m  
\ln \E_{E \sim \Exp{1}} e^{\lambda E} \gpdf(E-y_i) \right) \right) \\
& = \frac{1}{2\sqrt{K}} \exp \left( -m \cdot  \hat\Xi_1(\sigma) \right).
\end{align*}
In the last step, we used \eqref{xi_1_problem_def_argmax} and \eqref{xi_1_problem_def_max}.
\end{proof}
\subsection{Analysis of $\Nr$}
We recall the function $\Nr$ was defined as follows:
\begin{align*}
\Nr(\bm y, \bm Q) & \Mydef \E \left[ \prod_{i=1}^m \gpdf(y_i - | G_{1i}|^2 ) \gpdf(y_i -  |G_{2i}|^2 ) \bigg| \bm G ^\UH \bm G = {m \bm Q} \right],
\end{align*}
where the matrix $\bm Q$ is of the form:
\begin{align}
\label{Q_form}
\bm Q &  = \begin{bmatrix} 1 & q \\ q & 1 \end{bmatrix}, \; q \in (0,1).
\end{align}
We observe that $\Nr$ can be rewritten as:
\begin{align*}
\Nr(\bm y, \bm Q) & = \E \left[ \exp\left(\sum_{i=1}^m \ln \gpdf(y_i - | G_{1i}|^2 ) + \ln\gpdf(y_i -  |G_{2i}|^2 ) \right) \bigg| \frac{1}{m} \bm G ^\UH \bm G = {\bm Q} \right].
\end{align*}
The asymptotics of $\Nr$ are determined by the large deviation properties of the pair of random variables:
\begin{align*}
\left(\frac{1}{m}\sum_{i=1}^m \ln \gpdf(y_i - | G_{1i}|^2 ) + \ln\gpdf(y_i -  |G_{2i}|^2 ), \frac{1}{m} \bm G ^\UH \bm G \right).
\end{align*}
Both of these random variables are a sum of independent random variables. The Tilted Wishart distribution which is defined below will play a key role in our analysis. This distribution is the Cramer transform (or the exponential tilting) of the random variables defined above. 
\begin{definition}[The Tilted Wishart Distribution with Parameters $(\lambda, \phi, y)$]\label{titled_wishart_definition} A $2 \times 2$ Hermitian matrix $\bm S$ is said to be $\TWis{\lambda}{\phi}{y}$ if 
	\begin{align*}
	\bm S &   = \begin{bmatrix} s & \sqrt{s s^\prime} e^{\i \theta} \\ \sqrt{s s^\prime} e^{-i \theta} & s^\prime \end{bmatrix},
	\end{align*}
	and the random variables $s \in [0,\infty), s^\prime \in [0,\infty), \theta\in (-\pi, \pi]$ are sampled from the pdf:
	\begin{align*}
	h(s, s^\prime, \theta) & \Mydef \frac{1}{2 \cdot \pi \cdot \ZTWis{\lambda}{\phi}{y}} \cdot  \exp(-(1-\lambda)(s + s^\prime) + \phi \sqrt{s s^\prime} \cos(\theta)) \cdot \gpdf(s - y) \cdot  \gpdf(s^\prime - y).
	\end{align*}
	In the above display, the normalizing constant $\ZTWis{\lambda}{\phi}{y}$ is defined as:
	\begin{align*}
	\ZTWis{{\lambda}}{\phi}{y} & \Mydef  \frac{1}{2\pi }\int_0^\infty \int_0^\infty \int_{-\pi}^\pi \exp(-(1-\lambda)(s + s^\prime) + \phi \sqrt{s s^\prime} \cos(\theta)) \cdot \gpdf(s - y) \cdot  \gpdf(s^\prime - y) \diff \theta \diff s \diff s^\prime. 
	\end{align*} 
	We denote the covariance matrix of the tilted Wishart distribution by $\VTWis{\lambda}{\phi}{y}$, that is: 
	\begin{align*}
	\VTWis{\lambda}{\phi}{y} & = \E \left[ \vec{\bm S - \E \bm S} \vec{\bm S - \E \bm S } ^\UH \right].
	\end{align*}
\end{definition}
Similar to the analysis of $\Dr$, the analysis of $\Nr$ consists of two steps: First, a change of measure step which is given in Lemma \ref{Nr_representation_formula} and second, an application of the local central limit theorem which is given in Proposition \ref{local_clt_num}. 

We begin with the change of measure result. Let $\lambda, \phi \in \mathbb R$ be arbitrary. Let $\bm S_{1},\bm S_{2} \dots \bm S_{m}$  be independent Hermitian random matrices with $$\bm S_i \sim \TWis{\lambda}{\phi}{y_i},\; \forall \; i \; \in \; [m].$$
Define the random variable $\bm S$ as: 
\begin{align*}
\bm S & = \sum_{i=1}^m \bm S_i.
\end{align*}
Let $H_{\lambda,\phi, \bm y}$ be the density of the random matrix $\bm S$. 
\begin{lemma} \label{Nr_representation_formula}
	For any $\bm y \in \mathbb R^m$ and any $2 \times 2$ positive definite Hermitian matrix $\bm Q$,  we have, 
	\begin{align*}
	\Nr(\bm y, \bm Q) & = \frac{\pi (m-1)! (m-2)!}{m^{2m-2} \cdot \det(\bm Q)^{m-2}} \cdot e^{m (1-\lambda)\Tr(\bm Q) - m \phi \Re(Q_{12})} \cdot \left( \prod_{i=1}^m \ZTWis{\lambda}{\phi}{y_i} \right) \cdot  H_{\lambda,\phi, \bm y}(m\bm Q).
	\end{align*}
\end{lemma}
\begin{proof}
	Let us index the entries of $\bm S_k, \; k \; \in [m]$ as follows: 
	\begin{align*}
	\bm S_{k} & = \begin{bmatrix} s_k & \sqrt{s_k s_{k}^\prime} e^{\i \theta_k} \\ \sqrt{s_k s_{k}^\prime} e^{-\i \theta_k} & r_k^\prime \end{bmatrix}
	\end{align*}
	Define the random variables:
	\begin{align*}
	\bm S = \sum_{k=1}^m \bm S_k, \; T  = \sum_{k=1}^m \ln \gpdf(y_k-r_k) + \ln \gpdf(y_k - r_k^\prime).
	\end{align*}
	Consider two possible probability distributions for $\bm S,T$: 
	\begin{description}
		\item [Setup 1:] $\bm S_k = \bm g_k \bm g_k^\UH$ where $\bm g_k \sim \cgauss{\bm 0_}{\bm I_2}$. Equivalently,  $s_i$ and $s_i^\prime$ are i.i.d. $\Exp{1}$ and $\theta_i$ are i.i.d. $\text{Unif}(-\pi,\pi]$. Let $H(\cdot,\cdot)$ be the joint pdf of $\bm S, T$ in this setup.
		\item [Setup 2: ] $\bm S_k$ are independent and distributed as $\bm S_k \sim \TWis{\lambda}{\phi}{y_k}$. Let $H_{\lambda, \phi, \bm y}(\cdot,\cdot)$ denote the joint pdf of $\bm S,T$ in this setup.  
	\end{description}
	We can compute $H_{\lambda, \phi,\bm y}$ in terms of $G$ as follows: 
	\begin{align*}
	H_{\lambda,\phi, \bm y}(\bm S,T) & = \frac{\exp(T+\lambda \cdot \Tr(\bm S) + \phi \cdot  \Re(S_{12}))}{\prod_{i=1}^m \ZTWis{\lambda}{\phi}{y_i}} \cdot H(\bm S,T).
	\end{align*}
	Let $H(\cdot |\bm S)$ denote the conditional density of $T$ given $\bm S$ and $H(\bm S)$ denote the marginal density of $\bm S$ under Setup 1. Analogously define $H_{\lambda, \phi, \bm y}(\cdot |\bm S)$ and $H_{\lambda, \phi,\bm y}(\bm S)$ under Setup 2. We can then compute $\Nr(\bm y, \bm Q)$ as follows: 
	\begin{align*}
	\Nr(\bm y, \bm Q) & \Mydef \E \left[ \prod_{i=1}^m \phi_{\sigma}(y_i - | G_{1i}|^2 ) \phi_{\sigma}(y_i -  |G_{2i}|^2 ) \bigg| \bm G ^\UH \bm G = m {\bm Q} \right] \\ 
	& \explain{}{=} \E \left[ \exp \left( \sum_{i=1}^m \ln \phi_{\sigma}(y_i -  |g_{1i}|^2) + \ln \phi_{\sigma}(y_i -  |g_{2i}|^2)  \right) \bigg| \bm G^\UH \bm G = m \bm Q \right] \\
	& \explain{(a)}{=}  \int e^{t} H(t|m\bm Q) \diff t \\
	& = \frac{e^{-m \lambda \Tr(\bm Q) - m\phi \Re(Q_{12})}}{H(m \bm Q)} \int e^{t +m \lambda \Tr(\bm Q) + m\phi \Re(Q_{12})} H(m \bm Q,t) \diff t \\
	& = \frac{e^{-m \lambda \Tr(\bm Q) - m\phi \Re(Q_{12})}}{H(m \bm Q)} \cdot \prod_{i=1}^m \ZTWis{\lambda}{\phi}{y_i} \cdot  \int H_{\lambda,\phi, \bm y}(m\bm Q,t) \diff t \\
	& =\frac{e^{-m \lambda \Tr(\bm Q) - m\phi \Re(Q_{12})}}{H(m \bm Q)} \cdot \prod_{i=1}^m \ZTWis{\lambda}{\phi}{y_i} \cdot  H_{\lambda,\phi, \bm y}(m\bm Q) \\
	& \explain{(b)}{=} \frac{\pi (m-1)! (m-2)!}{m^{2m-2}\cdot \det(\bm Q)^{m-2}}  \cdot e^{m (1-\lambda)\Tr(\bm Q) - m \phi \Re(Q_{12})} \cdot \left( \prod_{i=1}^m \ZTWis{\lambda}{\phi}{y_i} \right) \cdot  H_{\lambda,\phi, \bm y}(m\bm Q).
	\end{align*}
	In the step marked (a), we used the fact that under Setup 1, we have
	\begin{align*}
	(\bm S, T) & \explain{d}{=} \left(\bm G^\UH \bm G, \;  \sum_{i=1}^m \ln \phi_{\sigma}(y_i -  |g_{1i}|^2) + \ln \phi_{\sigma}(y_i -  |g_{2i}|^2) \right).
	\end{align*}
	In the step marked (b), we used the fact that under Setup 1, $\bm S$ is distributed as a complex Wishart random matrix and hence,
	\begin{align*}
	H(m \bm Q) & = \frac{1}{\pi} \cdot \frac{m^{2m-2}}{(m-1)! (m-2)!} \cdot \exp(-m \Tr(\bm Q)) \cdot \det(\bm Q)^{m-2}.
	\end{align*} 
	This concludes the proof of the lemma.
\end{proof}
Next, we will use a local central limit theorem to characterize the asymptotics of $H_{\lambda,\phi,\bm y}(m \bm Q)$. Note that Lemma \ref{Nr_representation_formula} holds for any $\lambda,\phi \in \mathbb R$. We will set $\lambda = \hat{\lambda}_2(q; \sigma), \phi = \hat{\phi}(q; \sigma)$, where
\begin{align}
\label{Xi2_argmax}
	(\hat{\lambda}_2(q; \sigma), \hat{\phi}(q; \sigma)) & \Mydef \arg \max_{(\lambda,\phi) \in \mathbb{R}} \left( 2  \lambda + q \phi -\hatE_Y  
\ln \ZTWis{\lambda}{\phi}{Y} \right).
\end{align} 
We also define
\begin{align}
\label{Xi2_max}
\hat{\Xi}_2(q;\sigma) & \Mydef \max_{(\lambda,\phi) \in \mathbb{R}} \left( 2  \lambda + q \phi -\hatE_Y  
\ln \ZTWis{\lambda}{\phi}{Y} \right).
\end{align}
The rational behind this choice of $\lambda,\phi$ is that the first order optimality conditions for the above concave variational problem give us:
\begin{align*}
2  & = \frac{1}{m} \sum_{i=1}^m \frac{\partial_{\lambda}\ZTWis{\hat{\lambda}_2(q; \sigma)} {\hat{\phi}(q; \sigma)}{y_i}}{\ZTWis{\hat{\lambda}_2(q; \sigma)} {\hat{\phi}(q; \sigma)}{y_i}}  \explain{(a)}{=} \frac{1}{m} \sum_{i=1}^m \E (s_i + s_i^\prime) \\
q & = \frac{1}{m} \sum_{i=1}^m \frac{\partial_{\phi}\ZTWis{\hat{\lambda}_2(q; \sigma)}{\hat{\phi}(q; \sigma)}{y_i}}{\ZTWis{\hat{\lambda}_2(q; \sigma)} {\hat{\phi}(q; \sigma)}{y_i}}  \explain{(a)}{=} \frac{1}{m} \sum_{i=1}^m \E \sqrt{s_i s_i^\prime} \cos(\theta_i). \\
\end{align*}
In the steps marked (a), we used the formula for the normalizing constant $\ZTWis{\lambda}{\phi}{y}$, given in Definition \ref{titled_wishart_definition}, to compute the partial derivatives. It is also clear by the symmetry of Definition \ref{titled_wishart_definition} that: 
\begin{align*}
\E s_i = \E s_i^\prime, \; \E \sqrt{s_i s_i^\prime} \sin(\theta) = 0.
\end{align*}
Hence, the first order optimality conditions imply: 
\begin{align*}
\E \bm S &  =  \sum_{i=1}^m \E \bm S_i = m \bm Q.
\end{align*}
By the Multivariate Central Limit Theorem, we expect that $m^{-\frac{1}{2}} \cdot (\bm S - m\bm Q)$ to be asymptotically Gaussian. We also define the covariance matrix of $m^{-\frac{1}{2}} \cdot (\bm S - m\bm Q)$ as $\hat{\bm V}(q;\sigma)$:
\begin{align}
\label{numer_localclt_variance}
\hat{\bm V}(q; \sigma) & \Mydef \frac{\E \vec{\bm S - \E \bm S} \vec{\bm S - \E \bm S}^\UH}{m} = \hatE \VTWis{\hat{\lambda}_2(\bm Q; \sigma)}{ \hat{\phi}(\bm Q; \sigma)}{Y}.
\end{align}
By the CLT, we expect
\begin{align*}
m^{-\frac{1}{2}} \cdot \vec{ \bm S - m\bm Q} & \approx \gauss{\bm 0}{\hat{\bm V}(q; \sigma)}.
\end{align*}
Hence,
\begin{align*}
 H_{\hat{\lambda}_2(q;\sigma),\hat{\phi}(q; \sigma), \bm y}(m \bm Q) & \approx \frac{1}{\sqrt{(2\pi m)^4 \det(\hat{\bm V}(q; \sigma))}}.
\end{align*}
The following proposition makes this argument rigorous. 
\begin{restatable}[A Local Central Limit Theorem]{proposition}{localcltnum} \label{local_clt_num} Suppose that there exists a constant $0<K< \infty$ such that:
	\begin{align*}
	|\hat{\lambda}_2(q;\sigma)| + |\hat{\phi}(q; \sigma)| \leq K, \; \hatE_Y |Y|^{40} \leq K, \; \frac{1}{K}  \leq \lambda_{\min}\left(\hat{\bm V}(q;\sigma) \right)  \leq \lambda_{\max}\left(\hat{\bm V}(q;\sigma)\right) \leq K.
	\end{align*}
	Then, there exists a constant $C(K)$, depending only on $K$ such that we have the following asymptotic expansion for  $H_{\hat{\lambda}_2(q;\sigma),\hat{\phi}(q; \sigma), \bm y}$: 
	\begin{align*}
	\left| H_{\hat{\lambda}_2(q;\sigma),\hat{\phi}(q; \sigma), \bm y}- \frac{1}{\sqrt{(2\pi m)^4 \det(\hat{\bm V}(q; \sigma))}}\right|  \leq \frac{C(K)\ln^5(m)}{m^2\sqrt{m}}.
	\end{align*}
\end{restatable}
The proof of this proposition appears in Appendix \ref{local_clt_num_appendix} and closely follows classical proofs of local central limit theorems based on characteristic functions (see for example, \citet[Chapter 16]{feller2008introduction}).
We conclude our analysis of $\Nr$ with the following upper bound on $\Nr$ which is a straightforward corollary of Lemma \ref{Nr_representation_formula} and Proposition \ref{local_clt_num}.
\begin{corollary} \label{local_clt_num_corollary} Under the assumptions of Proposition 2, there exists $M(K) \in \mathbb N$ depending only on $K$ such that
	\begin{align*}
	\Nr(\bm y, \bm Q) & \leq \frac{C(K)}{m^2\cdot (1-q^2)^{m-2}}  \cdot  \exp \left(- m \cdot \hat{\Xi}_2(q;\sigma)  \right),
	\end{align*}
	for all $m \geq M(K)$.
\end{corollary}
\begin{proof}
	From Lemma 5, we know that
	\begin{align*}
	\Nr(\bm y, \bm Q) & = \frac{\pi (m-1)! (m-2)!}{m^{2m-2}} \cdot \det(\bm Q)^{m-2} \cdot e^{m (1-\lambda)\Tr(\bm Q) - m \phi \Re(Q_{12})} \cdot \left( \prod_{i=1}^m \ZTWis{\lambda}{\phi}{y_i} \right) \cdot  H_{\lambda,\phi, \bm y}(m\bm Q). 
	\end{align*}
	In Proposition \ref{local_clt_num}, we obtained the bound
	\begin{align*}
	\left|	H_{\hat{\lambda}, \hat{\phi}, \bm y}(m \bm Q) - \frac{1}{\sqrt{(2\pi m)^4 \det(\hat{\bm V}(q; \sigma))}}\right|  \leq \frac{C(K)\ln^5(m)}{m^2\sqrt{m}}.
	\end{align*}
	Note that under the assumptions of Proposition \ref{local_clt_num}, we have
	\begin{align*}
	\det(\hat{\bm V}(q; \sigma)) & \geq \lambda^4_{\min}(\hat{\bm V}(q; \sigma)) \geq \frac{1}{K^4}.
	\end{align*}
	This tells us, that there is a $M(K) \in \mathbb N$ depending only on $K$, such that, 
	\begin{align*}
	H_{\hat{\lambda}, \hat{\phi}, \bm y}(m \bm Q) & \leq \frac{C(K)}{m^2}, \; \forall m \geq M(K).
	\end{align*}
	By Stirling's approximation, we have 
	\begin{align*}
	\frac{\pi (m-1)! (m-2)!}{m^{2m-2}} & \leq \frac{\pi e^5}{e^{2m}}.
	\end{align*}
	These estimates give us the upper bound: 
	\begin{align*}
	\Nr(\bm y, \bm Q) & \leq \frac{C(K)e^{m(\Tr(\bm Q)-2)}}{m^2 \cdot \det(\bm Q)^{m-2}}  \cdot  \exp \left(- m \max_{(\lambda,\phi) \in \mathbb{R}} \left( \lambda \Tr(\bm Q) + \phi \Re(Q_{12}) - \frac{1}{m} \sum_{i=1}^m  
	\ln \ZTWis{\lambda}{\phi}{y_i} \right)  \right),
	\end{align*}
	for all $m \geq M(K)$.	
	Recalling the definition of $\hat{\Xi}_2(q;\sigma)$ (See \eqref{Xi2_argmax}) and the form of the matrix $\bm Q$ (see  \eqref{Q_form}) gives us the claim of the corollary.
\end{proof}
	\section{The Stochastic Laplace Method} \label{section_stochastic_laplace}
	Recall that in Lemmas \ref{second_moment_conditional} and \ref{introduce_main_integrals_lemma} we have shown the following upper bound on $\MI{\bm y, \bm z}{\bm A, \bm W}$:
	\begin{align*}
	\MI{\bm y, \bm z}{\bm A, \bm W} & \leq  \frac{2}{n-1} \E_{\bm y} \left[ \frac{ \int_0^1 \Nr \left(\bm y, \begin{bmatrix} 1 & q \\ q & 1 \end{bmatrix}\right) \cdot \frac{q \cdot (1-q^2)^{n-2}}{(1-q^2/2)^{\Delta m}} \diff q}{\Dr^2(\bm y)} \cdot  \mathbf{1}_{\mathcal{E}_m}\right] -1  + C \cdot m \cdot \sqrt{\P(\mathcal{E}_m^c)},
	\end{align*}
	where $\mathcal{E}_m$ is an arbitrary event depending on $\bm y$ and the functions $\Nr,\Dr$ were defined in \eqref{Nr_integral_def_eq} and \eqref{Dr_integral_def_eq}. Let us for the moment, also assume that the conditions required for Corollary \ref{local_clt_denom_corollary} and \ref{local_clt_num_corollary} are met. Then, tracking only the exponential order terms, we obtain,
	\begin{align}
	\label{laplace_integral}
	\MI{\bm y, \bm z}{\bm A, \bm W} & \lessapprox \E_{\bm y} \left[ \int_0^1 e^{-m \cdot \CalFhat(q; \delta, \Delta,\sigma) } \diff q \cdot \mathbf{1}_{\mathcal{E}_m}\right],
	\end{align}
	where,
	\begin{align}
	\label{calF_empirical_definition}
	\CalFhat(q; \delta, \Delta,\sigma) & = \hat{\Xi}_2(q;\sigma) - 2\hat\Xi_1(\sigma)  + \left( 1 - \frac{1}{\delta} \right) \ln(1-q^2) + \Delta \ln \left( 1- \frac{q^2}{2} \right).
	\end{align}
	Our goal will be to evaluate the integral in \eqref{laplace_integral} via the Laplace Method. However, we observe that the function $\CalFhat(q; \delta, \Delta,\sigma)$ is stochastic since it depends on the empirical distribution of the phase retrieval observations $\bm y$. It turns out that $\hat{\Xi}_2(q;\sigma)$, defined in \eqref{Xi2_argmax}, and $\hat\Xi_1(\sigma)$, defined in \eqref{xi_1_problem_def_max}, and hence $\CalFhat(q; \delta, \Delta,\sigma)$ concentrate around deterministic functions ${\Xi}_2(q;\sigma), \Xi_1(\sigma), \CalF(q; \delta, \Delta,\sigma)$ defined below:
	\begin{align}
	{\Xi}_1( \sigma) & \Mydef \max_{\lambda \in \mathbb{R}} \left( \lambda  - \E_Y  
	\ln \E_{E \sim \Exp{1}} e^{\lambda E} \gpdf(E-Y) \right),\label{xi_1_problem_def_max_pop} \\
	{\Xi}_2(q;\sigma) & \Mydef \max_{(\lambda,\phi) \in \mathbb{R}} \left( 2  \lambda + q \phi -\E_Y  
	\ln \ZTWis{\lambda}{\phi}{Y} \right),\label{Xi2_max_pop} \\
	\CalF(q; \delta, \Delta,\sigma) & = {\Xi}_2(q;\sigma) - 2\Xi_1(\sigma)  + \left( 1 - \frac{1}{\delta} \right) \ln(1-q^2) + \Delta \ln \left( 1- \frac{q^2}{2} \right) \label{calF_pop_def}.
	\end{align} 
	In the above display, the random variable $Y = |Z|^2 + \sigma \epsilon$ where $Z \sim \cgauss{0}{1}, \; \epsilon \sim \gauss{0}{1}$. We also define the deterministic counterparts to $\hat{\lambda}_1(\sigma)$, defined in \eqref{xi_1_problem_def_argmax} and $\hat{\lambda}_2(q;\sigma), \hat{\phi}(q;\sigma)$, defined in \eqref{Xi2_argmax}:
	\begin{align}
	\lambda_1(\sigma) & \Mydef \arg\max_{\lambda \in \mathbb{R}} \left( \lambda  - \E_Y  
	\ln \E_{E \sim \Exp{1}} e^{\lambda E} \gpdf(E-Y) \right),\label{xi_1_problem_def_argmax_pop} \\
({\lambda}_2(q;\sigma), \phi(q;\sigma) ) & \Mydef \max_{(\lambda,\phi) \in \mathbb{R}} \left( 2  \lambda + q \phi -\E_Y  
	\ln \ZTWis{\lambda}{\phi}{Y} \right).\label{Xi2_argmax_pop}
	\end{align}
	The convergence to these deterministic functions allows to design a high probability event $\mathcal{E}_m$ on which applying Laplace method to the stochastic function $\CalFhat(q; \delta, \Delta,\sigma)$ is essentially the same as applying it to the deterministic function $\CalF(q; \delta, \Delta,\sigma)$.  We state our concentration result in the proposition below. 
	\begin{proposition} \label{concentration_proposition} For any fixed $\sigma > 0$, we have the following convergence results:
		\begin{enumerate}
			\item Convergence of Moments: $\hatE Y^{k} \explain{P}{\rightarrow} \E Y^{k}$ for any $k \in \mathbb N$, where $Y = |Z|^2 + \sigma \epsilon, \; Z \sim \cgauss{0}{1}$ and $\epsilon \sim \gauss{0}{1}$. 
			\item For any $R \in (0,\infty)$, we have the uniform convergence of the functions: \begin{align*}\sup_{|\lambda| \leq R} |\hatE \VTexp{\lambda}{Y} - \E \VTexp{\lambda}{Y} | &\explain{P}{\rightarrow} 0, \\ \sup_{|\lambda| + |\phi| \leq R} \| \hatE \VTWis{\lambda}{\phi}{Y} - \E \VTWis{\lambda}{\phi}{Y}  \|  &\explain{P}{\rightarrow} 0.
			\end{align*}
			\item $\hat{\lambda}_1(\sigma)$ is tight in the sense that, there exists a constant $R \in (0,\infty)$, depending only on $\sigma$ such that,
			\begin{align*}
			\P \left( |\hat{\lambda}_1(\sigma)| > R\right) & \rightarrow 0.
			\end{align*}
			\item $\hat \Xi_1(\sigma) \explain{P}{\rightarrow} \Xi_1(\sigma)$
			\item For any $\eta \in (0,1)$, there exists  $R_\eta \in (0,\infty)$ (depending only on $\eta,\sigma$) such that:
			\begin{align*} 
			\P \left( \max_{0 \leq q \leq 1-\eta} |\hat\lambda_2(q;\sigma)| + |\hat \phi(q;\sigma)| > R_\eta \right)  \rightarrow 0.
			\end{align*}
			\item For any $\eta \in (0,1)$, we have,
			\begin{align*}
			\sup_{q \in [0,1-\eta]} |\hat \Xi_2(q;\sigma) - \Xi_2(q;\sigma)| & \explain{P}{\rightarrow} 0.
			\end{align*}
			\item For any $\eta \in (0,1)$, we have, 
			\begin{align*}
			\sup_{q \in [0,1-\eta]}|\hat \lambda_2 (q;\sigma) - \lambda_2(q;\sigma)| \explain{P}{\rightarrow} 0, \; 
			\sup_{q \in [0,1-\eta]}|\hat \phi (q;\sigma) - \phi(q;\sigma)| \explain{P}{\rightarrow} 0.
			\end{align*}
			\item For any $\eta \in (0,1)$, we have, 
			\begin{align*}
			\sup_{q \in [0,1-\eta]}\left| \frac{\diff^2  }{\diff q^2} \hat{\Xi}_2(q;\sigma) - \frac{\diff^2  }{\diff q^2} {\Xi}_2(q;\sigma)  \right| &\explain{P}{\rightarrow} 0.
			\end{align*}
		\end{enumerate}
	\end{proposition}
	The proof of this Proposition appears in Appendix \ref{concentration_appendix}. It uses standard empirical process theory results from \citet{vanweak} with some modification to account for the fact that the observations $y_1,y_2 \dots ,y_m$ are not independent. With the above concentration result, we suitably design an event $\mathcal{E}_m$ with $\P(\mathcal{E}_m) \rightarrow 1$ such that on the event $\mathcal{E}_m$, we are able to adapt the usual proof of Laplace Method to obtain the following conclusion.
	\begin{restatable}{proposition}{deltaDeltaCriteria} \label{delta_Delta_criteria} Suppose that $\delta,\Delta,\sigma$ are such that $ \CalF(q;\delta,\Delta,\sigma) >  \CalF(0;\delta,\Delta,\sigma) = 0 \; \forall \; q \in (0,1)$ and $ \frac{\diff^2 \CalF }{\diff q^2}(0;\delta,\Delta,\sigma) > 0$. Then, $\MI{\bm y, \bm z}{\bm A, \bm W} = o(m)$.
	\end{restatable}
	The proof of this proposition can be found in Appendix \ref{delta_Delta_criteria_appendix}. The claim of this Proposition is very intuitive: It says that due to the concentration of $\CalFhat(q;\delta,\Delta,\sigma)$ to $\CalF(q;\delta,\Delta,\sigma)$, the stochastic and the deterministic integrals:
	\begin{align*}
	\int_0^1 e^{-m \CalFhat(q;\delta,\Delta,\sigma) } \diff q & \approx \int_0^1 e^{-m \CalF(q;\delta,\Delta,\sigma) } \diff q,
	\end{align*}
	behave very similarly. According to the standard Laplace method, the condition $ \CalF(q;\delta,\Delta,\sigma) >  \CalF(0;\delta,\Delta,\sigma) = 0$ ensures that,
	\begin{align*}
	\frac{1}{m} \ln \left(\int_0^1 e^{-m \CalF(q;\delta,\Delta,\sigma) } \diff q \right) & \rightarrow 0,
	\end{align*}
	whereas the positivity requirement on the second derivative ensures that the second order, subexponential factors in the Laplace integral are sufficiently well controlled to obtain $\MI{\bm y, \bm z}{\bm A, \bm W} = o(m)$.
	\section{Low Noise Asymptotics} \label{low_noise_section}
	Proposition \ref{delta_Delta_criteria} and Proposition \ref{bayes_risk_to_MI} tell us that if for some $\delta,\sigma$, we can find $\Delta > 0$ such that:
	\begin{align}
	\CalF(q;\delta,\Delta,\sigma) >  \CalF(0;\delta,\Delta,\sigma) \; \forall \; q \in (0,1), \; \frac{\diff^2 \CalF }{\diff q^2}(0;\delta,\Delta,\sigma) > 0, \label{analysis_question}
	\end{align}
	then,
	\begin{align*}
		\lim_{\substack{m,n \rightarrow \infty \\ m = n \delta}}\E_{\bm x_\star, \bm y, \bm A} \| \bm x_\star \bm x^\UH_\star - \E[ \bm x_\star \bm x_\star^\UH | \bm y, \bm A] \|^2 & = 1.
	\end{align*}
	Note that the Bayes risk increases monotonically with the noise level $\sigma$ (that is, the phase retrieval problem is harder for larger noise levels). Furthermore, the Bayes risk is atmost the risk of the trivial estimator $\hat{\bm x} = \bm 0$:
	\begin{align*}
	    \limsup_{\substack{m,n \rightarrow \infty \\ m = n \delta}}\E_{\bm x_\star, \bm y, \bm A} \| \bm x_\star \bm x^\UH_\star - \E[ \bm x_\star \bm x_\star^\UH | \bm y, \bm A] \|^2 &   \leq \E_{\bm x_\star, \bm y, \bm A} \| \bm x_\star \bm x^\UH_\star -\bm 0\|^2 = 1.
	\end{align*}
	Hence if show that the asymptotic Bayes risk is trivial (that is, equal to $1$) for an arbitrarily small $\sigma > 0$, it automatically implies the Bayes risk is trivial for larger values of noise. Consequently we will focus on verifying condition \eqref{analysis_question} for small values of noise, where the analysis of the variational problems involved simplifies considerably.  We show the following result:
	\begin{restatable}{proposition}{lownoiseprop} \label{lownoiseprop}
		Recall that $\CalF(q;\delta,\Delta,\sigma)$ was defined as:
		\begin{align*}
		\CalF(q; \delta, \Delta,\sigma) & = {\Xi}_2(q;\sigma) - 2\Xi_1(\sigma)  + \left( 1 - \frac{1}{\delta} \right) \ln(1-q^2) + \Delta \ln \left( 1- \frac{q^2}{2} \right).
		\end{align*}
		For any $\delta$ and $\Delta$ that satisfy
		\begin{align*}
		1 \leq  \delta < 2, \; 0< \Delta< \frac{2-\delta}{\delta},
		\end{align*}
		there exists a critical value of the noise level $\sigma_c(\delta,\Delta)>0$ such that, for any $0 < \sigma < \sigma_c(\delta,\Delta)$, we have
		\begin{enumerate}
			\item The function $\CalF(q;\delta,\Delta,\sigma)$ has a unique minimum at $q=0$ and $\CalF(q;\delta,\Delta,\sigma) > \CalF(0;\delta,\Delta,\sigma), \; \forall \; q \; \in (0,1)$. 
			\item $\frac{\diff^2 \CalF}{\diff q^2}(q; \delta,\Delta,\sigma) \bigg|_{q=0} > 0$. 
		\end{enumerate}
	\end{restatable}
	Combined with 	Proposition \ref{delta_Delta_criteria} and Proposition \ref{bayes_risk_to_MI} it immediately gives us Theorem \ref{main_result} as a corollary. 
	\begin{corollary} Theorem \ref{main_result} holds.
	\end{corollary}
	\begin{proof}
		When $\delta < 2$, we can set:
		\begin{align*}
		\Delta = \frac{2-\delta}{2 \delta} > 0.
		\end{align*}
		Proposition \ref{lownoiseprop} guarantees that \eqref{analysis_question} holds for all values of $0<\sigma \leq \sigma_c(\delta,\Delta)$. Proposition \ref{delta_Delta_criteria} lets us conclude that for all $0<\sigma \leq \sigma_c(\delta,\Delta)$, $\MI{\bm y,\bm z}{\bm A, \bm W} = o(m)$. Consequently, by Proposition \ref{bayes_risk_to_MI}, for any $0<\sigma \leq \sigma_c(\delta,\Delta)$ we have,
		\begin{align*}
	\lim_{\substack{m,n \rightarrow \infty \\ m = n \delta}}\E_{\bm x_\star, \bm y, \bm A} \| \bm x_\star \bm x^\UH_\star - \E[ \bm x_\star \bm x_\star^\UH | \bm y, \bm A] \|^2 & = 1.
		\end{align*}
		Since the Bayes risk is atmost $1$ and increases monotonically with $\sigma$, this means for any $\sigma > 0$:
		\begin{align*}
	 \lim_{\substack{m,n \rightarrow \infty \\ m = n \delta}}\E_{\bm x_\star, \bm y, \bm A} \| \bm x_\star \bm x^\UH_\star - \E[ \bm x_\star \bm x_\star^\UH | \bm y, \bm A] \|^2 & = 1.
		\end{align*}
	\end{proof}
The proof of Proposition \ref{lownoiseprop} can be found in Appendix \ref{low_noise_section_proofs}. The main idea of the proof is that in the limit $\sigma \rightarrow 0$, the analysis of the function $\CalF(q;\delta,\Delta,\sigma)$ simplifies considerably. 
	\section{Conclusions and Future Work}
	In this work we studied the Phase Retrieval problem with subsampled Haar sensing matrices with non-zero but vanishing measurement noise in the high dimensional asymptotic where the signal dimension ($n$) and the number of measurements ($m$) diverge such that the sampling ratio $\delta = m/n$ remains fixed. We showed that when the sampling ratio $\delta = m/n < 2$, then it is information theoretically impossible for any estimator to obtain an asymptotically non-trivial performance: any estimator is asymptotically uncorrelated with the signal vector. Since previous work \citep{ma2019spectral,dudeja2019rigorous} has designed estimators which achieve a non trivial correlation with the planted vector when $\delta > 2$, this shows that the weak recovery threshold for this model is $\weakthr = 2$. There are a number of interesting directions for future work which we discuss below:
	\paragraph{Other Unitarily Invariant Ensembles} It is interesting to understand the weak recovery threshold when the measurement matrix $\bm A$ is drawn from a general unitarily invariant ensemble: $\bm A = \bm U \bm \Lambda \bm V^\UH$ where $\bm U, \bm V$ are independent uniformly random column orthogonal random matrices and $\bm \Lambda$ is a diagonal matrix whose spectral measure converges to some limiting distribution. Both Gaussian sensing matrices and Subsampled Haar sensing matrices are special cases of this general ensemble. While we don't pursue this direction in this paper, we expect our techniques to generalize to general unitarily invariant ensembles. 
	\paragraph{Mean Square Error above Weak Recovery Threshold} Our proof techniques do not offer any information about the behaviour of Bayes risk above the weak recovery threshold. For Gaussian sensing matrices \citet{barbier2019optimal} have developed interpolation based methods to compute the exact expression of Bayes risk. It would be interesting to see if these can be extended beyond i.i.d. sensing matrices. This would enable us to determine the threshold at which exact recovery of the signal vector becomes possible. Recent work by \citet{barbier2018mutual} takes a step in this direction and studies linear sensing problems where the sensing matrix is a product of matrices with i.i.d. gaussian entries. 
	\paragraph{High Dimensional Analysis for Highly Structured Sensing Matrices} An important and difficult open problem is to understand the information theoretic limits of recovery when the sensing matrix is a structured random matrix like the CDP sensing matrix. While simulations results confirm that there is a good agreement between the CDP matrices and the theoretical results obtained for partial Haar matrices, the theoretical justification of this empirical observation and the limits of this agreement are yet to be discovered. 
	\bibliographystyle{plainnat}
	\bibliography{ref}

\begin{thebibliography}{45}
\providecommand{\natexlab}[1]{#1}
\providecommand{\url}[1]{\texttt{#1}}
\expandafter\ifx\csname urlstyle\endcsname\relax
  \providecommand{\doi}[1]{doi: #1}\else
  \providecommand{\doi}{doi: \begingroup \urlstyle{rm}\Url}\fi

\bibitem[Abbara et~al.(2019)Abbara, Baker, Krzakala, and
  Zdeborov{\'a}]{abbara2019universality}
Alia Abbara, Antoine Baker, Florent Krzakala, and Lenka Zdeborov{\'a}.
\newblock On the universality of noiseless linear estimation with respect to
  the measurement matrix.
\newblock \emph{arXiv preprint arXiv:1906.04735}, 2019.

\bibitem[Abbasi et~al.(2019)Abbasi, Salehi, and
  Hassibi]{abbasi2019universality}
Ehsan Abbasi, Fariborz Salehi, and Babak Hassibi.
\newblock Universality in learning from linear measurements.
\newblock \emph{arXiv preprint arXiv:1906.08396}, 2019.

\bibitem[Abramowitz and Stegun(1965)]{abramowitz1965handbook}
Milton Abramowitz and Irene~A Stegun.
\newblock \emph{Handbook of mathematical functions: with formulas, graphs, and
  mathematical tables}, volume~55.
\newblock Courier Corporation, 1965.

\bibitem[Bahmani and Romberg(2017)]{bahmani2017phase}
Sohail Bahmani and Justin Romberg.
\newblock Phase retrieval meets statistical learning theory: A flexible convex
  relaxation.
\newblock In \emph{Artificial Intelligence and Statistics}, pages 252--260,
  2017.

\bibitem[Barbier et~al.(2018)Barbier, Macris, Maillard, and
  Krzakala]{barbier2018mutual}
Jean Barbier, Nicolas Macris, Antoine Maillard, and Florent Krzakala.
\newblock The mutual information in random linear estimation beyond iid
  matrices.
\newblock In \emph{2018 IEEE International Symposium on Information Theory
  (ISIT)}, pages 1390--1394. IEEE, 2018.

\bibitem[Barbier et~al.(2019)Barbier, Krzakala, Macris, Miolane, and
  Zdeborov{\'a}]{barbier2019optimal}
Jean Barbier, Florent Krzakala, Nicolas Macris, L{\'e}o Miolane, and Lenka
  Zdeborov{\'a}.
\newblock Optimal errors and phase transitions in high-dimensional generalized
  linear models.
\newblock \emph{Proceedings of the National Academy of Sciences}, 116\penalty0
  (12):\penalty0 5451--5460, 2019.

\bibitem[Bayati et~al.(2015)Bayati, Lelarge, Montanari,
  et~al.]{bayati2015universality}
Mohsen Bayati, Marc Lelarge, Andrea Montanari, et~al.
\newblock Universality in polytope phase transitions and message passing
  algorithms.
\newblock \emph{The Annals of Applied Probability}, 25\penalty0 (2):\penalty0
  753--822, 2015.

\bibitem[Bhattacharya and Rao(1986)]{bhattacharya1986normal}
Rabi~N Bhattacharya and R~Ranga Rao.
\newblock \emph{Normal approximation and asymptotic expansions}, volume~64.
\newblock SIAM, 1986.

\bibitem[Bhattacharya(1975)]{bhattacharya1975errors}
Rabindra~N Bhattacharya.
\newblock On errors of normal approximation.
\newblock \emph{The Annals of Probability}, 3\penalty0 (5):\penalty0 815--828,
  1975.

\bibitem[Boucheron et~al.(2013)Boucheron, Lugosi, and
  Massart]{boucheron2013concentration}
St{\'e}phane Boucheron, G{\'a}bor Lugosi, and Pascal Massart.
\newblock \emph{Concentration inequalities: A nonasymptotic theory of
  independence}.
\newblock Oxford university press, 2013.

\bibitem[Candes et~al.(2013)Candes, Strohmer, and
  Voroninski]{candes2013phaselift}
Emmanuel~J Candes, Thomas Strohmer, and Vladislav Voroninski.
\newblock Phaselift: Exact and stable signal recovery from magnitude
  measurements via convex programming.
\newblock \emph{Communications on Pure and Applied Mathematics}, 66\penalty0
  (8):\penalty0 1241--1274, 2013.

\bibitem[Candes et~al.(2015{\natexlab{a}})Candes, Li, and
  Soltanolkotabi]{candes2015phase}
Emmanuel~J Candes, Xiaodong Li, and Mahdi Soltanolkotabi.
\newblock Phase retrieval from coded diffraction patterns.
\newblock \emph{Applied and Computational Harmonic Analysis}, 39\penalty0
  (2):\penalty0 277--299, 2015{\natexlab{a}}.

\bibitem[Candes et~al.(2015{\natexlab{b}})Candes, Li, and
  Soltanolkotabi]{candes2015wirtinger}
Emmanuel~J Candes, Xiaodong Li, and Mahdi Soltanolkotabi.
\newblock Phase retrieval via wirtinger flow: Theory and algorithms.
\newblock \emph{IEEE Transactions on Information Theory}, 61\penalty0
  (4):\penalty0 1985--2007, 2015{\natexlab{b}}.

\bibitem[Chaganty and Sethuraman(1993)]{chaganty1993strong}
Narasinga~Rao Chaganty and Jayaram Sethuraman.
\newblock Strong large deviation and local limit theorems.
\newblock \emph{The Annals of Probability}, pages 1671--1690, 1993.

\bibitem[Dhifallah et~al.(2018)Dhifallah, Thrampoulidis, and
  Lu]{dhifallah2018phase}
Oussama Dhifallah, Christos Thrampoulidis, and Yue~M Lu.
\newblock Phase retrieval via polytope optimization: Geometry, phase
  transitions, and new algorithms.
\newblock \emph{arXiv preprint arXiv:1805.09555}, 2018.

\bibitem[Donoho and Tanner(2009)]{donoho2009observed}
David Donoho and Jared Tanner.
\newblock Observed universality of phase transitions in high-dimensional
  geometry, with implications for modern data analysis and signal processing.
\newblock \emph{Philosophical Transactions of the Royal Society A:
  Mathematical, Physical and Engineering Sciences}, 367\penalty0
  (1906):\penalty0 4273--4293, 2009.

\bibitem[Dudeja et~al.(2019)Dudeja, Bakhshizadeh, Ma, and
  Maleki]{dudeja2019rigorous}
Rishabh Dudeja, Milad Bakhshizadeh, Junjie Ma, and Arian Maleki.
\newblock Rigorous analysis of spectral methods for random orthogonal matrices.
\newblock \emph{arXiv preprint arXiv:1903.02676}, 2019.

\bibitem[Feller(2008)]{feller2008introduction}
Willliam Feller.
\newblock \emph{An introduction to probability theory and its applications},
  volume~2.
\newblock John Wiley \& Sons, 2008.

\bibitem[Goldstein and Studer(2018)]{goldstein2018phasemax}
Tom Goldstein and Christoph Studer.
\newblock Phasemax: Convex phase retrieval via basis pursuit.
\newblock \emph{IEEE Transactions on Information Theory}, 64\penalty0
  (4):\penalty0 2675--2689, 2018.

\bibitem[Guionnet and Maida(2005)]{guionnet2005fourier}
Alice Guionnet and M~Maida.
\newblock A fourier view on the r-transform and related asymptotics of
  spherical integrals.
\newblock \emph{Journal of functional analysis}, 222\penalty0 (2):\penalty0
  435--490, 2005.

\bibitem[Kabashima(2008)]{kabashima2008inference}
Yoshiyuki Kabashima.
\newblock Inference from correlated patterns: a unified theory for perceptron
  learning and linear vector channels.
\newblock In \emph{Journal of Physics: Conference Series}, volume~95, page
  012001. IOP Publishing, 2008.

\bibitem[Lu and Li(2019)]{yuelu_spectral}
Yue~M. Lu and Gen Li.
\newblock Phase transitions of spectral initialization for high-dimensional
  nonconvex estimation.
\newblock \emph{Information and Inference, to appear}, 2019.
\newblock URL \url{https://arxiv.org/abs/1702.06435}.

\bibitem[Luo et~al.(2019)Luo, Alghamdi, and Lu]{yuelu_optimal}
Wangyu Luo, Wael Alghamdi, and Yue~M. Lu.
\newblock Optimal spectral initialization for signal recovery with applications
  to phase retrieval.
\newblock \emph{IEEE Transactions on Signal Processing}, 67\penalty0
  (9):\penalty0 2347--2356, 2019.
\newblock URL \url{https://arxiv.org/abs/1811.04420}.

\bibitem[Ma et~al.(2019{\natexlab{a}})Ma, Dudeja, Xu, Maleki, and
  Wang]{ma2019spectral}
Junjie Ma, Rishabh Dudeja, Ji~Xu, Arian Maleki, and Xiaodong Wang.
\newblock Spectral method for phase retrieval: an expectation propagation
  perspective.
\newblock \emph{arXiv preprint arXiv:1903.02505}, 2019{\natexlab{a}}.

\bibitem[Ma et~al.(2019{\natexlab{b}})Ma, Xu, and Maleki]{ma2019optimization}
Junjie Ma, Ji~Xu, and Arian Maleki.
\newblock Optimization-based amp for phase retrieval: The impact of
  initialization and regularization.
\newblock \emph{IEEE Transactions on Information Theory}, 65\penalty0
  (6):\penalty0 3600--3629, 2019{\natexlab{b}}.

\bibitem[Millane(1990)]{millane1990phase}
Rick~P Millane.
\newblock Phase retrieval in crystallography and optics.
\newblock \emph{JOSA A}, 7\penalty0 (3):\penalty0 394--411, 1990.

\bibitem[Monajemi et~al.(2013)Monajemi, Jafarpour, Gavish, Donoho,
  Collaboration, et~al.]{monajemi2013deterministic}
Hatef Monajemi, Sina Jafarpour, Matan Gavish, David~L Donoho, Stat 330/CME~362
  Collaboration, et~al.
\newblock Deterministic matrices matching the compressed sensing phase
  transitions of gaussian random matrices.
\newblock \emph{Proceedings of the National Academy of Sciences}, 110\penalty0
  (4):\penalty0 1181--1186, 2013.

\bibitem[Mondelli and Montanari(2019)]{mondelli2019fundamental}
Marco Mondelli and Andrea Montanari.
\newblock Fundamental limits of weak recovery with applications to phase
  retrieval.
\newblock \emph{Foundations of Computational Mathematics}, 19\penalty0
  (3):\penalty0 703--773, 2019.

\bibitem[Netrapalli et~al.(2013)Netrapalli, Jain, and
  Sanghavi]{netrapalli2013phase}
Praneeth Netrapalli, Prateek Jain, and Sujay Sanghavi.
\newblock Phase retrieval using alternating minimization.
\newblock In \emph{Advances in Neural Information Processing Systems}, pages
  2796--2804, 2013.

\bibitem[Oymak and Hassibi(2014)]{oymak2014case}
Samet Oymak and Babak Hassibi.
\newblock A case for orthogonal measurements in linear inverse problems.
\newblock In \emph{2014 IEEE International Symposium on Information Theory},
  pages 3175--3179. IEEE, 2014.

\bibitem[Petrov(2012)]{petrov2012sums}
Valentin~Vladimirovich Petrov.
\newblock \emph{Sums of independent random variables}, volume~82.
\newblock Springer Science \& Business Media, 2012.

\bibitem[Qu et~al.(2017)Qu, Zhang, Eldar, and Wright]{qu2017convolutional}
Qing Qu, Yuqian Zhang, Yonina Eldar, and John Wright.
\newblock Convolutional phase retrieval.
\newblock In \emph{Advances in Neural Information Processing Systems}, pages
  6086--6096, 2017.

\bibitem[Reeves(2017)]{reeves2017additivity}
Galen Reeves.
\newblock Additivity of information in multilayer networks via additive
  gaussian noise transforms.
\newblock In \emph{2017 55th Annual Allerton Conference on Communication,
  Control, and Computing (Allerton)}, pages 1064--1070. IEEE, 2017.

\bibitem[Reeves et~al.(2019)Reeves, Xu, and Zadik]{reeves2019all}
Galen Reeves, Jiaming Xu, and Ilias Zadik.
\newblock The all-or-nothing phenomenon in sparse linear regression.
\newblock \emph{arXiv preprint arXiv:1903.05046}, 2019.

\bibitem[Shechtman et~al.(2015)Shechtman, Eldar, Cohen, Chapman, Miao, and
  Segev]{shechtman2015phase}
Yoav Shechtman, Yonina~C Eldar, Oren Cohen, Henry~Nicholas Chapman, Jianwei
  Miao, and Mordechai Segev.
\newblock Phase retrieval with application to optical imaging: a contemporary
  overview.
\newblock \emph{IEEE signal processing magazine}, 32\penalty0 (3):\penalty0
  87--109, 2015.

\bibitem[Takeda et~al.(2006)Takeda, Uda, and Kabashima]{takeda2006analysis}
Koujin Takeda, Shinsuke Uda, and Yoshiyuki Kabashima.
\newblock Analysis of cdma systems that are characterized by eigenvalue
  spectrum.
\newblock \emph{EPL (Europhysics Letters)}, 76\penalty0 (6):\penalty0 1193,
  2006.

\bibitem[Takeda et~al.(2007)Takeda, Hatabu, and
  Kabashima]{takeda2007statistical}
Koujin Takeda, Atsushi Hatabu, and Yoshiyuki Kabashima.
\newblock Statistical mechanical analysis of the linear vector channel in
  digital communication.
\newblock \emph{Journal of Physics A: Mathematical and Theoretical},
  40\penalty0 (47):\penalty0 14085, 2007.

\bibitem[Thrampoulidis and Hassibi(2015)]{thrampoulidis2015isotropically}
Christos Thrampoulidis and Babak Hassibi.
\newblock Isotropically random orthogonal matrices: Performance of lasso and
  minimum conic singular values.
\newblock In \emph{2015 IEEE International Symposium on Information Theory
  (ISIT)}, pages 556--560. IEEE, 2015.

\bibitem[Ushakov(2011)]{ushakov2011selected}
Nikolai~G Ushakov.
\newblock \emph{Selected topics in characteristic functions}.
\newblock Walter de Gruyter, 2011.

\bibitem[Van~der Vaart and Wellner(1996)]{vanweak}
AW~Van~der Vaart and JA~Wellner.
\newblock Weak convergence and empirical processes, 1996.

\bibitem[Vehkaper{\"a} et~al.(2016)Vehkaper{\"a}, Kabashima, and
  Chatterjee]{vehkapera2016analysis}
Mikko Vehkaper{\"a}, Yoshiyuki Kabashima, and Saikat Chatterjee.
\newblock Analysis of regularized ls reconstruction and random matrix ensembles
  in compressed sensing.
\newblock \emph{IEEE Transactions on Information Theory}, 62\penalty0
  (4):\penalty0 2100--2124, 2016.

\bibitem[Watson(1983)]{watson1983statistics}
Geoffrey~S Watson.
\newblock Statistics on spheres.
\newblock 1983.

\bibitem[Watson(1995)]{watson1995treatise}
George~Neville Watson.
\newblock \emph{A treatise on the theory of Bessel functions}.
\newblock Cambridge university press, 1995.

\bibitem[Wedin(1973)]{wedin1973perturbation}
Per-{\AA}ke Wedin.
\newblock Perturbation theory for pseudo-inverses.
\newblock \emph{BIT Numerical Mathematics}, 13\penalty0 (2):\penalty0 217--232,
  1973.

\bibitem[Wen et~al.(2016)Wen, Zhang, Wong, Chen, and Yuen]{wen2016sparse}
Chao-Kai Wen, Jun Zhang, Kai-Kit Wong, Jung-Chieh Chen, and Chau Yuen.
\newblock On sparse vector recovery performance in structurally orthogonal
  matrices via lasso.
\newblock \emph{IEEE Transactions on Signal Processing}, 64\penalty0
  (17):\penalty0 4519--4533, 2016.

\end{thebibliography}
	\appendix
	\section{Proofs from Section \ref{MI_and_Bayes_Risk}}
	In this section, we collect the missing proofs from Section \ref{MI_and_Bayes_Risk}. We begin with a lemma describing the joint distribution of the phase retrieval measurements $\bm y$ and the side information $\bm z$.
	\begin{lemma} \label{side_info_distribution_lemma} Let $\bm x_\star, \bm y$ and $\bm z$ denote the signal vector, the measurements and side information sampled from the phase retrieval with side information model (see \eqref{subsampled_haar_sensing}, \eqref{PR_problem} and \eqref{side_info_model}). Then conditioned on $\bm x_\star$, $\bm y$ and $\bm z$ are independent with marginal distributions:
		\begin{align*}
		\bm y &\explain{d}{=} m |\bm U|^2 + \sigma \bm \epsilon, \; \bm U \sim \unif{\mathbb S^{m-1}}, \; \bm \epsilon \sim \gauss{\bm 0}{\bm I_m}, \\
		\bm z & \explain{}{\sim} \gauss{\bm 0}{2 \cdot \bm I_{\lfloor \Delta \cdot m \rfloor}}.
		\end{align*}
		Furthermore, since the above distributions do not depend on $\bm x_\star$, this result holds even without conditioning on $\bm x_\star$. 
	\end{lemma}
	\begin{proof}
		From \eqref{subsampled_haar_sensing}, \eqref{PR_problem} and \eqref{side_info_model}, we know that,
		\begin{align*}
		\bm y  = m |\bm A \bm x_\star|^2 + \sigma \bm \epsilon, \; z_i \explain{i.i.d.}{\sim} \gauss{\ip{\bm w_i}{\bm x_\star \bm x_\star^\UH}}{1}, \; i \in \{1,2 \dots , \lfloor \Delta m \rfloor\},
		\end{align*}
		where $\bm A$ is a uniformly random $m \times n$ partial unitary matrix and the matrices $\bm w_i \explain{i.i.d.}{\sim} \GUE{n}$. Since $\bm A$ is independent of $\bm w_1, \bm w_2 \dots, \bm w_{\lfloor \Delta m \rfloor}$, we have $\bm y, \bm z$ are conditionally independent given $\bm x_\star$.  
		Let $\bm B$ be the $n \times n$ unitary matrix whose first column $\bm B_1 = \bm x_\star$ (and the remaining columns can be arbitrary). Then note that, conditioned on $\bm x_\star$,
		\begin{align*}
		\bm A \bm x_\star & \explain{(1)}{=} \bm A \bm B \bm B^\UH \bm x_\star \\
		& \explain{}{=} \bm A \bm B  \bm e_1 \\
		& \explain{d,(2)}{=} \bm A \bm e_1.
		\end{align*}
		In the above display, the step marked (1) used the fact that $\bm B \bm B^\UH  = \bm I_n$, the distribution inequality (2) used the fact that since $\bm A$ is a uniformly random partial unitary matrix, its distribution is invariant to left multiplication by a unitary matrix. Finally note that the first column of a partial unitary matrix $\bm A \bm e_1 \sim \unif{\mathbb S^{m-1}}$. This gives us:
		\begin{align}
		\bm y &\explain{d}{=} m |\bm U|^2 + \sigma \bm \epsilon, \; \bm U \sim \unif{\mathbb S^{m-1}}, \; \bm \epsilon \sim \gauss{\bm 0}{\bm I_m}. \label{y_dist_eq}
		\end{align}
		Next observe that since $\bm w_i \sim \GUE{n}$, conditioned on $\bm x_\star$, 
		\begin{align*}
		\ip{\bm w_i}{\bm x_\star \bm x_\star^\UH} & \explain{i.i.d.}{\sim} \gauss{0}{1}, \; i \in \{1,2 \dots, \lfloor \Delta m \rfloor\}.
		\end{align*}
		Hence, conditioned on $\bm x_\star$,
		\begin{align*}
		z_i & \explain{i.i.d.}{\sim} \gauss{0}{2}.
		\end{align*}
		This proves the claim of the lemma. Note that since the conditional distributions do not depend on $\bm x_\star$, this result holds even without conditioning on $\bm x_\star$
	\end{proof}
	The remainder of this section is organized as follows:
	\begin{enumerate}
		\item Section \ref{bayes_risk_to_MI_proof} is devoted to the proof of Proposition \ref{bayes_risk_to_MI}.
		\item Section \ref{second_moment_conditional_proof} is devoted to the proof of Lemma \ref{second_moment_conditional}.
	\end{enumerate}
	\subsection{Proof of Proposition \ref{bayes_risk_to_MI}}
	\label{bayes_risk_to_MI_proof}
	\begin{proof}
		Let $\Delta>0$ be fixed to any value that guarantees: 
		\begin{align*}
		\MI{\bm y, \bm z}{\bm A, \bm W} & = o(m).
		\end{align*}
		By the chain rule for mutual information,
		\begin{align*}
		\MI{\bm y, \bm z}{\bm A, \bm W } & = \MI{\bm y}{\bm A, \bm W} + \MI{\bm z}{\bm A, \bm W | \bm y} \\
		& = \MI{\bm y}{\bm A} + \MI{\bm z}{\bm A, \bm W | \bm y} \\
		& \geq \MI{\bm y}{\bm A}.
		\end{align*}
		Consequently $\MI{\bm y}{\bm A} = o(m)$. This means,
		\begin{align}
		\CEnt{\bm y}{\bm A} & = \Ent{\bm y} - o(m) \label{conditional_to_unconditional_ent_1} \\
		\CEnt{\bm y, \bm z}{\bm A,\bm W} & = \Ent{\bm y, \bm z} - o(m). \label{conditional_to_unconditional_ent_2}
		\end{align}
		In order to prove the claim of the proposition, we will costruct an upper bound and a lower bound on the quantity $\CEnt{\bm z}{\bm y, \bm A, \bm W}$. Comparing the upper and lower bound will give us the claim of the proposition. 
		\begin{align}
		\CEnt{\bm z}{\bm y, \bm A, \bm W} & \explain{}{=} \CEnt{\bm y, \bm z}{\bm A,\bm W} - \CEnt{\bm y}{\bm A, \bm W} \nonumber\\& \explain{}{=} \CEnt{\bm y, \bm z}{\bm A,\bm W} - \CEnt{\bm y}{\bm A} \nonumber \\
		& \explain{(a)}{=} \Ent{\bm y, \bm z} - \Ent{\bm y} - o(m) \nonumber\\
		& \explain{(b)}{=} \Ent{\bm z} - o(m) \nonumber\\
		& \explain{(c)}{=}  \lfloor \Delta m \rfloor \cdot  h( 2 ) \cdot (1 - o(1)) \\
		& \explain{}{=} \Delta m \cdot h(2) \cdot (1-o(1))
		 \label{conditional_entropy_lower_bound}.
		\end{align}
		In the equality marked (a), we used the conclusions derived in \eqref{conditional_to_unconditional_ent_1} and \eqref{conditional_to_unconditional_ent_2}. In the step marked (b), we used the fact that $\bm y, \bm z$ are independent (see Lemma \ref{side_info_distribution_lemma}). In step (c)  we defined $h(v) \Mydef \frac{1}{2} \ln(2\pi v)$, which is the entropy of $\gauss{0}{v}$ and recalled the claim of Lemma \ref{side_info_distribution_lemma}: $z_i \explain{i.i.d.}{\sim} \gauss{0}{1}$. 
		On the other hand we can upper bound $\CEnt{\bm z}{\bm y, \bm A, \bm W}$ as follows: 
		\begin{align}
		\CEnt{\bm z}{\bm y, \bm A, \bm W} & \leq \sum_{i=1}^{\lfloor \Delta m \rfloor} \CEnt{z_i}{\bm y, \bm A, \bm W}\nonumber \\& \explain{(a)}{\leq} \sum_{i=1}^{\lfloor\Delta m\rfloor} \E h \left( \text{Var}(z_i | \bm y, \bm A, \bm W_{}) \right) \nonumber \\
		& \explain{(b)}{\leq}  \sum_{i=1}^{\lfloor\Delta m\rfloor} h( \E \text{Var}(z_i | \bm y, \bm A, \bm W) ) \label{condition_entry_upper_bound}
		\end{align}
		In the step marked (a) we used the fact that the Gaussian Distribution has the maximal entropy for a fixed variance and in step (b) we used the concavity of $h$. 
		Next we compute $ \E \text{Var}(Z_i | \bm Y, \bm A, \bm W_{1:\Delta m})$. We have, 
		\begin{align}
		\E \text{Var}(z_i | \bm y, \bm A, \bm W) & = \E (z_i - \E[z_i | \bm Y ,\bm A, \bm W])^2  \nonumber\\& \explain{(a)}{=} 
		\E \ip{\bm w_i} {\bm x_\star \bm x_\star^\UH - \E[\bm x_\star \bm x_\star^\UH|\bm y, \bm A, \bm W]}^2 + 1 \nonumber\\
		& \explain{(b)}{=} \E \ip{\bm w_i} {\bm x_\star \bm x_\star^\UH - \E[\bm x_\star \bm x_\star^\UH|\bm y, \bm A]}^2 + 1 \nonumber\\
		& \explain{(c)}{=} \E \| \bm x_\star \bm x_\star^\UH - \E[\bm x_\star \bm x_\star^\UH|\bm y, \bm A] \|^2 + 1 \label{conditional_variance}.
		\end{align}
		In the above display, the equality (a) follows from the fact that $z_i \sim \gauss{\ip{\bm w_i}{\bm x_\star \bm x_\star^\UH}}{1}$ and  equality (b) used the fact that $\bm W$ is independent of $\bm x_\star, \bm y, \bm A$. In the step (c), we used the following property of a GUE matrix: for a deterministic Hermitian matrix $\bm M$, $\ip{\bm w_i}{\bm M} \sim \gauss{0}{\|\bm M\|^2}$. 
		 \eqref{conditional_entropy_lower_bound}, \eqref{condition_entry_upper_bound} and \eqref{conditional_variance} give us the conclusion:
		\begin{align*}
		\Delta m \cdot  h(\E \| \bm x_\star \bm x_\star^\UH - \E[\bm x_\star \bm x_\star^\UH|\bm Y, \bm A] \|^2 + 1) & \geq \Delta m \cdot h(2) (1- o(1)).
		\end{align*}
		Since $h$ is an increasing function this gives us: 
		\begin{align*}
		\liminf_{\substack{m,n \rightarrow \infty \\ m= \delta n}} \E_{\bm x_\star, \bm y, \bm A} \| \bm x_\star \bm x^\UH_\star - \E[ \bm x_\star \bm x_\star^\UH | \bm y, \bm A] \|^2 & \geq 1.
		\end{align*}
		On the other hand, by the optimality of the Bayes estimator, we have: $\E_{\bm x_\star, \bm y, \bm A} \| \bm x_\star \bm x^\UH_\star - \E[ \bm x_\star \bm x_\star^\UH | \bm y, \bm A] \|^2  \leq \E_{\bm x_\star, \bm y, \bm A} \| \bm x_\star \bm x^\UH_\star - \bm 0 \|^2 = 1$. 
		Hence, 
		\begin{align*}
		\lim_{\substack{m,n \rightarrow \infty \\ m = n \delta}}\E_{\bm x_\star, \bm y, \bm A} \| \bm x_\star \bm x^\UH_\star - \E[ \bm x_\star \bm x_\star^\UH | \bm y, \bm A] \|^2 & = 1.
		\end{align*}
		This concludes the proof of the proposition.
	\end{proof}
	\subsection{Proof of Lemma \ref{second_moment_conditional}}
	\label{second_moment_conditional_proof}
	\begin{proof}
		Through out this proof $C$ refers to a finite non-negative constant independent of $m,n$ that can possibly depend on $\delta, \sigma^2, \Delta$. This constant may change from line to line. Recall that, 
		\begin{align*}
		\MI{\bm y, \bm z}{\bm A, \bm W} & = \Ent{\bm y, \bm z} - \Ent{\bm y, \bm z | \bm A, \bm W}.
		\end{align*}
		We can split $ \Ent{\bm y, \bm z | \bm A, \bm W}$ as follows:
		\begin{align*}
		&\Ent{\bm y, \bm z| \bm A, \bm W}  \explain{}{=} - \E_{\bm A, \bm W}  \left(\int  p(\bm y, \bm z | \bm A, \bm W) \ln p(\bm y, \bm z | \bm A, \bm W)     \diff \bm y \diff \bm  z \right)  \\
		& = -\E_{\bm A, \bm W}  \left(\int_{\mathcal{E}_m}  p(\bm y, \bm z | \bm A, \bm W) \ln p(\bm y, \bm z| \bm A, \bm W)     \diff \bm y \diff \bm z \right) - \E_{\bm A, \bm W}  \left(\int_{\mathcal{E}_m^c}  p(\bm y, \bm z | \bm A, \bm W) \ln p(\bm y, \bm z| \bm A,\bm W)     \diff \bm y \diff \bm z \right) \\
		& \explain{(a)}{=} - \int_{\mathcal{E}_m} \E_{\bm A,\bm W} \left[ p(\bm y, \bm z | \bm A, \bm W) \ln p(\bm y, \bm z | \bm A, \bm W) \right]  \diff \bm y\diff \bm z-   \int_{\mathcal{E}_m^c} \E_{\bm A, \bm W}  p(\bm y, \bm z | \bm A, \bm W) \ln p(\bm y, \bm z| \bm A,\bm W)     \diff \bm y \diff \bm z.
		\end{align*}
		In the step marked (a) we used Fubini's Theorem. Likewise we can split $\Ent{\bm y, \bm z}$ as follows:
		\begin{align*}
		\Ent{\bm y, \bm z} & = - \int_{\mathcal{E}_m}  p(\bm y, \bm z) \ln  p(\bm y, \bm z) \diff \bm y \diff \bm z - \int_{\mathcal{E}_m^c}  p(\bm y, \bm z) \ln  p(\bm y, \bm z) \diff \bm y \diff \bm z.
		\end{align*}
		Hence,
		\begin{align*}
		\MI{\bm y, \bm z}{\bm A, \bm W} & =\mathsf{I}+ \mathsf{II} + \mathsf{III},
		\end{align*}
		where the terms $\mathsf{I}, \mathsf{II}, \mathsf{III}$ are defined as:
		\begin{align*}
		\mathsf{I} & \Mydef	- \int_{\mathcal{E}_m}  p(\bm y, \bm z) \ln  p(\bm y, \bm z) \diff \bm y \diff \bm z + \int_{\mathcal{E}_m} \E_{\bm A,\bm W} \left[ p(\bm y, \bm z | \bm A, \bm W) \ln p(\bm y, \bm z | \bm A, \bm W) \right]  \diff \bm y\diff \bm z, \\
		\mathsf{II} & \Mydef  - \int_{\mathcal{E}_m^c}  p(\bm y, \bm z) \ln  p(\bm y, \bm z) \diff \bm y \diff \bm z, \\
		\mathsf{III} & \Mydef \int_{\mathcal{E}_m^c} \E_{\bm A, \bm W}  p(\bm y, \bm z | \bm A, \bm W) \ln p(\bm y, \bm z| \bm A,\bm W)     \diff \bm y \diff \bm z.
		\end{align*}
		\paragraph{Analysis of $\mathsf{I}:$}
		Consider the  following inequality:
		\begin{align*}
		\ln(x) \leq (x-1) \implies  x\ln(x)  \leq x(x-1), \; \forall x \geq 0.
		\end{align*}
		Applying this to $\frac{p(\bm y, \bm z | \bm A, \bm W)}{p(\bm y, \bm z)}$, we obtain,
		\begin{align*}
		p(\bm y, \bm z | \bm A, \bm W) \ln  p(\bm y, \bm z | \bm A, \bm W) & \leq  p(\bm y, \bm z | \bm A, \bm W) \ln(p(\bm y, \bm z)) - p(\bm y, \bm z | \bm A, \bm W) + \frac{p^2(\bm y, \bm z | \bm A, \bm W)}{p(\bm y, \bm z)}.
		\end{align*}
		Substituting this in the expression for $\mathsf{I}$ we obtain,
		\begin{align*}
		\mathsf{I} & \explain{}{\leq } - \int_{\mathcal{E}_m}  p(\bm y, \bm z) \ln  p(\bm y, \bm z) \diff \bm y \diff \bm z + \int_{\mathcal{E}_m} p(\bm y,\bm z) \ln p(\bm y, \bm z )  \diff \bm y \diff \bm z - \Pr(\mathcal{E}_m)  + \int_{\mathcal{E}_m} \frac{\E_{\bm A, \bm W}   p^2(\bm y, \bm z|\bm A, \bm W)}{ p(\bm y, \bm z)} \diff \bm y \diff \bm z  \\
		& =  \left(  \int_{\mathcal{E}_m} \frac{\E_{\bm A, \bm W}   p^2(\bm y, \bm z|\bm A, \bm W)}{ p(\bm y, \bm z)} \diff \bm y \diff \bm z -1 \right) + \P(\mathcal{E}_m^c).
		\end{align*}
		Hence we have,
		\begin{align}
		\mathsf{I} & \leq  \left(  \int_{\mathcal{E}_m} \frac{\E_{\bm A, \bm W}   p^2(\bm y, \bm z|\bm A, \bm W)}{ p(\bm y, \bm z)} \diff \bm y \diff \bm z -1 \right) + \P(\mathcal{E}_m^c). \label{conditional_second_moment_term_I}
		\end{align}
		\paragraph{Analysis of $\mathsf{II}:$ }
		We can handle $\mathsf{II}$ as follows:
		\begin{align*}
		\mathsf{II} &\Mydef -\int_{\mathcal{E}_m^c} p(\bm y, \bm z) \ln p(\bm y, \bm z) \diff \bm y \diff\bm z\\& \explain{(a)}{=} -\int_{\mathcal{E}_m^c} p(\bm y) \ln p(\bm y) \diff \bm y + \Ent{\bm z} \P(\mathcal{E}_m^c) \\
		& \explain{(b)}{\leq} -\int_{\mathcal{E}_m^c} p(\bm y) \ln p(\bm y) \diff \bm y + C \cdot m \cdot  \P(\mathcal{E}_m^c)  \\ & = -\int_{\mathcal{E}_m^c} p(\bm y) \ln \E_{\bm x,\bm A} p(\bm y|\bm x, \bm A) \diff \bm y + C \cdot m \cdot  \P(\mathcal{E}_m^c) \\
		& \explain{(c)}{\leq} -\E_{\bm x, \bm A,\bm y} \mathbf{1}_{\mathcal{E}_m^c} \ln p(\bm y | \bm x, \bm A)+ C \cdot m \cdot  \P(\mathcal{E}_m^c) \\
		&\explain{}{ =}  \frac{1}{2\sigma^2} \E \|\bm y - m|\bm A \bm x|^2 \|^2 \mathbf{1}_{\mathcal{E}_m^c}  + \frac{m \ln(2 \pi \sigma^2)}{2} \P(\mathcal{E}_m^c) + C \cdot m \cdot  \P(\mathcal{E}_m^c) \\
		& \explain{}{\leq } C \cdot \left(m^2\P(\mathcal{E}_m^c) \E \|\bm A \bm x\|^4_4 + \E \|\bm y\|^2 \mathbf{1}_{\mathcal{E}_m^c} \right) + C \cdot m \cdot  \P(\mathcal{E}_m^c).
		\end{align*}
		In the step marked (a) we used the fact that $\bm y, \bm z$ are marginally independent. In the step marked (b) we used the fact that $\Ent{\bm z} \leq Cm$ for a suitable $C$.  In the step marked (c) we applied Jensen's Inequality and note that the random variables $\bm x, \bm A$ and $\bm y$ are independent.  
		Note that by Cauchy Schwartz Inequality, we have, 
		\begin{align*}
		\E \|\bm y\|^2 \mathbf{1}_{\mathcal{E}_m^c}  \leq \sqrt{\E \|\bm y\|^4 \cdot  \P(\mathcal{E}_m^c)}.
		\end{align*}
		It is also straightforward to obtain the following estimates by simple moment computations:
		\begin{align*}
		\E \| \bm A \bm x\|^4_4 = \sum_{i=1}^m \E |\ip{\bm a_i}{\bm x}|^4 = \sum_{i=1}^m \E \| \bm a_i\|^4 |x_1|^4 \leq m \E |x_1|^4  \leq \frac{C}{m}, \; \E \|\bm y\|^4 \leq C m^2.
		\end{align*}
		for some $0 \leq C < \infty$.
		This gives us: 
		\begin{align}
		\mathsf{II} & \leq  Cm \left( \P(\mathcal{E}_m^c) + \sqrt{\P(\mathcal{E}_m^c)} \right).
		\label{conditional_second_moment_term_II}
		\end{align}
		\paragraph{Analysis of $\mathsf{III}:$ }
		Next we analyze the term $\mathsf{III}$: 
		\begin{align*}
		\mathsf{III} &= \int_{\mathcal{E}_m^c} \E_{\bm A, \bm W}p(\bm y, \bm z | \bm A, \bm W ) \ln p(\bm y, \bm z |\bm A, \bm W) \diff \bm y \diff \bm z
		\end{align*}
		Noting that: 
		\begin{align*}
		\ln p(\bm y, \bm z |\bm A, \bm W) & = \ln \E_{\bm x} p(\bm y, \bm z | \bm A, \bm W, \bm x) \\ 
		& = \ln \E_{\bm x} e^{-\| \bm y - m |\bm A \bm x|^2\|^2/2\sigma^2} + \ln \E_{\bm  x} \left[ \prod_{i=1}^{\lfloor \Delta m \rfloor} e^{- (z_i - \ip{\bm x \bm x^\UH}{ \bm w_i})^2/2} \right]  - \frac{m \ln(2 \pi \sigma^2) + \lfloor \Delta m \rfloor \ln(2\pi)}{2} \\
		& \leq - \frac{m \ln(2 \pi \sigma^2) +\lfloor \Delta m \rfloor \ln(2\pi)}{2} \\
		& \leq Cm.
		\end{align*}
		Hence we obtain, 
		\begin{align*}
		\mathsf{III} & \leq  Cm \P(\mathcal{E}_m^c).
		\end{align*}
		Combining the estimates on $\mathsf{I, II, III}$ we obtain, 
		\begin{align*}
		\MI{\bm y, \bm z}{\bm A, \bm W} & \leq \left(  \int_{\mathcal{E}_m} \frac{\E_{\bm A, \bm W}   p^2(\bm y, \bm z|\bm A, \bm W)}{ p(\bm y, \bm z)} \diff \bm y \diff \bm z -1 \right) + C \cdot m \cdot \sqrt{\P(\mathcal{E}_m^c)}.
		\end{align*}
	\end{proof}

\section{Proofs of Local Central Limit Theorems}
The proofs of the Local central limit theorems are based on the classical approach using characteristic functions. Section \ref{appendix_localcltdenom} contains the proof of the local CLT in Proposition \ref{local_clt_denom} and Section \ref{local_clt_num_appendix} contains the proof of the local CLT in Proposition \ref{local_clt_num}. The proofs use some standard properties of characteristic functions which have been collected in Appendix \ref{CF_appendix} for reference. We will also rely on some analytic properties of the Tilted Exponential distribution and Tilted Wishart distribution given in Appedices \ref{texp_properties_appendix} and \ref{twis_properties_appendix}.
\subsection{Proof of Proposition \ref{local_clt_denom}}
\label{appendix_localcltdenom}
\begin{proof}
	Recall the random variable $U$ was defined as:
	\begin{align*}
	U & = \sum_{i=1}^m u_i, \; u_i \sim \Texp{\hat{\lambda}_1(\sigma)}{y_i}, \; i \in [\Delta m],
	\end{align*}
	where,
	\begin{align*}
	\hat{\lambda}_1( \sigma) & \Mydef \arg\max_{\lambda \in \mathbb{R}} \left( \lambda  - \hat\E_Y  
	\ln \E_{E \sim \Exp{1}} e^{\lambda E} \gpdf(E-Y) \right).
	\end{align*}
	Note that $\lambda = \hat{\lambda}_1(\sigma)$ satisfies the first order stationarity condition: 
	\begin{align*}
	1 & = \frac{1}{m} \sum_{i=1}^m \frac{\E_{E \sim \Exp{1}} E \cdot e^{\hat{\lambda}_1(\sigma) E} \gpdf(E-y_i)}{\E_{E \sim \Exp{1}} e^{\hat{\lambda}_1(\sigma) E} \gpdf(E-y_i)} \Leftrightarrow \sum_{i=1}^m \E u_i = m.
	\end{align*}
	From here on, throughout this proof, we will shorthand $\hat{\lambda}_1(\sigma)$ as simply $\hat\lambda_1$.
	Define the centered random variables: $\check{u}_i = u_i - \E u_i$ and centered and normalized random variable: 
	\begin{align*}
	\check{U} & = \frac{U - m }{\sqrt{m}} = \frac{\sum_{i=1}^m \check{u}_i}{\sqrt{m}}.
	\end{align*}
	Let $\hat{v}(\sigma)$ denote the variance of $\check{U}$: 
	\begin{align*}
	\hat{v}(\sigma) & \Mydef \frac{1}{m} \sum_{i=1}^m \E \check{u}_i^2 = \frac{1}{m} \sum_{i=1}^m \VTexp{\hat{\lambda}_1}{y_i}.
	\end{align*}
	
	Again for ease of notation we will short hand $\hat{v}(\sigma)$ as $\hat{v}$. 
	Let $\check{F}$ denote the density of $\check{U}$. Let $\check{\psi}(t) = \E e^{\i t \check{U}}$ denote the characteristic function of $\check{U}$. By the change of variable formula, we have, 
	\begin{align*}
	F_{\hat{\lambda}_1, \bm y} (m) & = \frac{\check{F}(0)}{\sqrt{m}}
	\end{align*}
	Hence we focus on computing $\check{F}(0)$.  By the Fourier Inversion formula (Lemma \ref{fourier_inversion_CF}, Appendix \ref{CF_appendix}) we have, 
	\begin{align*}
	&|\check{F}(u) - \phi_{\sqrt{\hat{v}}}(u)|  = \frac{1}{2\pi} \left| \int_{\mathbb{R}} e^{-\i t u} \left( \check{\psi}(t) - e^{-\frac{\hat{v} t^2}{2}} \right) \diff t  \right| \\
	& \explain{(a)}{\leq} \frac{1}{2\pi} \left( \int_{|t| \leq t_1} \left| \check{\psi}(t) - e^{-\frac{\hat{v} t^2}{2}} \right| \diff t + \int_{t_1 \leq|t| \leq t_2 \sqrt{m}} |\check{\psi}(t)| \diff t + \int_{|t| \geq t_2  \sqrt{m} } |\check{\psi}(t)|  \diff t + \int_{|t| \geq t_1} e^{-\frac{\hat{v} t^2}{2}} \diff t \right) \\
	& \explain{(b)}{\leq} \frac{1}{2\pi} \left( \underbrace{\int_{|t| \leq t_1} \left| \check{\psi}(t) - e^{-\frac{\hat{v} t^2}{2}} \right| \diff t}_{(1)} + \underbrace{\int_{t_1 \leq|t| \leq t_2 \sqrt{m}} |\check{\psi}(t)| \diff t}_{(2)} + \underbrace{\int_{|t| \geq t_2  \sqrt{m} } |\check{\psi}(t)|  \diff t}_{(3)} + \frac{2}{\hat{v}} e^{-\frac{\hat{v} t_1^2}{2} } \right) \\
	\end{align*}
	In the step marked (a), the cutoff parameters $t_1,t_2$ are arbitrary and will be fixed later. In the step marked (b), we used standard bounds on the tail of a gaussian integral (see Lemma \ref{truncated_gauss_integral}, Appendix \ref{misc_appendix}). In the following sequence of steps, we upper bound each of the error terms $(1),(2)$ and $(3)$. We will be able to show, for a suitable selection of $t_1,t_2$, that, 
	\begin{align*}
	(1) + (2) + (3) + \frac{2}{\hat{v}} e^{-\frac{\hat{v} t_1^2}{2} } & \leq \frac{C(K) \cdot \ln(m)}{\sqrt{m}}.
	\end{align*}
	This gives us,
	\begin{align*}
	\left|\check{F}(0) - \frac{1}{\sqrt{2\pi \hat{v}}}\right|  \leq \frac{C(K)\ln(m)}{\sqrt{m}} \implies \left|{F}_{\hat{\lambda}_1,\bm y}(m) - \frac{1}{\sqrt{2\pi \hat{v}\cdot m}}\right|  \leq \frac{C(K)\ln(m)}{m},
	\end{align*}
	which is the claim of this proposition. The remaining proof is devoted to the analysis of (1), (2) and (3).
	\begin{description}
		\item[Analysis of (1): ] Recall $\check{\psi}(t) = \E e^{\i t \check{U}}$ and $f(x)=e^{\i tx}$ is bounded, $t$-lipchitz function of $x$. Applying the Berry-Eseen Inequality (Theorem \ref{berry_eseen_theorem}, Appendix \ref{CF_appendix}), we have, 
		\begin{align*}
		\left|\check{\psi}(t) - e^{-\frac{\hat{v} t^2}{2}} \right| & \leq \frac{C \cdot (1+ \sqrt{\hat{v}}|t|) \cdot \rho_3}{\sqrt{m \cdot \hat{v}^3}}.
		\end{align*}
		In the above display, $C$ is a universal constant and $\rho_3$ is given by: 
		\begin{align*}
		\rho_3		 & = \frac{1}{m} \sum_{i=1}^m \E |u_i - \E u_i|^3 \\
		& \leq \frac{8}{m} \sum_{i=1}^m \E |u_i|^3 \\
		& \explain{(c)}{\leq} C \left(1 + |\hat{\lambda}_1|^3 + \frac{1}{m} \sum_{i=1}^m |y_i|^3 \right).
		\end{align*}
		In the step marked (c) we used the estimate on $\E |u_i|^3$ proved in Lemma \ref{lemma_texp_properties}.
		Integrating the pointwise bound above we obtain: 
		\begin{align*}
		(1) & \leq C \cdot \left(1 + |\hat{\lambda}_1|^3 + \frac{1}{m} \sum_{i=1}^m |y_i|^3 \right) \cdot \frac{t_1(1+\sqrt{\hat{v}}t_1)}{\sqrt{m \cdot \hat{v}^3}}. 
		\end{align*}
		We set: 
		\begin{align*}
		t_1 & = \sqrt{\frac{2\ln(m)}{\hat{v}}}.
		\end{align*}
		This gives us: 
		\begin{align*}
		(1)  \leq \frac{C}{\hat{v}^2}  \cdot \left(1 + |\hat{\lambda}_1|^3 + \frac{1}{m} \sum_{i=1}^m |y_i|^3 \right) \cdot \frac{\ln(m)}{\sqrt{m}} \leq \frac{C(K) \cdot \ln(m)}{\sqrt{m}}.
		\end{align*}
		\item[Analysis of (2):] Let $(u_1^\prime, u_2^\prime \dots u_m^\prime)$ be independent and identically distributed as $(u_1,u_2 \dots u_m)$. Note that, 
		\begin{align*}
		\left| \E e^{\i t \check{u}_i } \right|^2  =  \left| \E e^{\i t {u}_i } \right|^2 
		=  \E e^{\i t ({u}_i-u_i^\prime) }. 
		\end{align*}
		Hence, 
		\begin{align*}
		\left| \check{\psi}(t) \right|^2  = \prod_{i=1}^m \left| \E \exp \left( \frac{\i t\check{u}_i}{\sqrt{m}}  \right) \right|^2 = \prod_{i=1}^m \E \exp \left( \frac{\i t({u}_i-u_i^\prime)}{\sqrt{m}}  \right).
		\end{align*}
		By the Taylor's theorem for CF (Theorem \ref{taylors_cf}, Appendix \ref{CF_appendix}), we have, 
		\begin{align*}
		\E \exp \left( \frac{\i t({u}_i-u_i^\prime)}{\sqrt{m}}  \right) & = 1 - \frac{\E(u_i - u_i^\prime)^2 \cdot t^2}{2m} + E_i, \; |E_i| \leq \frac{\E|u_i - u_i^\prime|^3 \cdot |t|^3}{6 m \sqrt{m}}.
		\end{align*}
		Now consider any $t \leq t_2 \sqrt{m}$:
		\begin{align*}
		|\check{\psi}(t)|^2 & = \prod_{i=1}^m \left( 1 - \frac{\E(u_i - u_i^\prime)^2 \cdot t^2}{2m} + E_i \right) \\& \explain{}{\leq} \prod_{i=1}^m  \left(1 - \frac{\E(u_i - u_i^\prime)^2 \cdot t^2}{2m} + \frac{\E|u_i - u_i^\prime|^3 \cdot |t|^3}{6 m \sqrt{m}} \right) \\
		& \leq \exp \left( - \frac{t^2}{2m} \sum_{i=1}^m \E(u_i - u_i^\prime)^2 + \frac{|t|^3}{6 m \sqrt{m}} \sum_{i=1}^m \E|u_i - u_i^\prime|^3 \right).
		\end{align*}
		Next we observe that, 
		\begin{align*}
		\frac{1}{2m} \sum_{i=1}^m \E (u_i - u_i^\prime)^2  &= \hat{v}.
		\end{align*}
		We set: 
		\begin{align*}
		t_2 & = \frac{\hat{v}}{2} \cdot \left( \frac{1}{m} \sum_{i=1}^m \E |u_i - u_i^\prime|^3 \right)^{-1}.
		\end{align*}
		This ensures, for any $|t| \leq t_2 \sqrt{m}$, we have,
		\begin{align*}
		|\check{\psi}(t)|^2 & \leq \exp \left( - \frac{t^2}{2m} \sum_{i=1}^m \E(u_i - u_i^\prime)^2 + \frac{t^3}{6 m \sqrt{m}} \sum_{i=1}^m \E|u_i - u_i^\prime|^3 \right) \leq  \exp\left(- \frac{\hat{v} t^2}{2}\right).
		\end{align*}
		Consequently,
		\begin{align*}
		(2) = \int_{t_1 \leq|t| \leq t_2 \sqrt{m}} |\check{\psi}(t)| \diff t 
		\leq \int_{t_1 \leq|t| \leq t_2 \sqrt{m}} e^{-\hat{v} t^2/4} \diff t 
		&\leq \int_{t_1 \leq|t|} e^{-\hat{v} t^2/4} \diff t \\
		&\explain{(d)}{\leq} \frac{4}{\hat{v}} \exp\left(-\frac{\hat{v} t^2_1}{4}\right) \\ &\explain{(e)}{=} \frac{4}{\hat{v} \sqrt{m}} \leq \frac{C(K)}{\sqrt{m}}.
		\end{align*}
		In the step marked (d), we used the standard bound on gaussian tail integrals (Lemma \ref{truncated_gauss_integral}) and in the step marked (e) we substituted the value of $t_1$ fixed in the analysis of (1). 
		Finally, to wrap up this step, we note that there exists a finite positive constant $C(K)$ such that, 
		\begin{align*}
		t_2 & \geq \frac{1}{C(K)}.
		\end{align*}
		Indeed, 
		\begin{align*}\frac{1}{m} \sum_{i=1}^m \E |u_i - u_i^\prime|^3 \leq \frac{8}{m} \sum_{i=1}^m \E |u_i |^3 \leq C \left(1 + |\hat{\lambda}_1|^3 + \frac{1}{m} \sum_{i=1}^m |y_i|^3 \right) \leq C(K),
		\end{align*}
		and,
		\begin{align*}
		t_2 & = \frac{\hat{v}}{2} \cdot \left( \frac{1}{m} \sum_{i=1}^m \E |u_i - u_i^\prime|^3 \right)^{-1} \geq \frac{1}{C(K)}.
		\end{align*}
		\item[Analysis of (3): ] Recall the term (3) was given by: 
		\begin{align*}
		(3)  = \int_{|t| \geq t_2  \sqrt{m} } |\check{\psi}(t)|  \diff t = \sqrt{m} \int_{|t| \geq t_2} |\check{\psi}(t \sqrt{m})| \diff t.
		\end{align*}
		
		By AM-GM for non-negative real numbers we have, 
		\begin{align*}
		\left| \check{\psi}(t\sqrt{m}) \right|^2  = \prod_{i=1}^m \left| \E e^{\i t u_i} \right|^2 \leq \left( \frac{1}{m} \sum_{i=1}^m  \left| \E e^{\i t u_i}  \right|^2 \right)^m.
		\end{align*}
		We use two different strategies to further control the above bound:
		\begin{enumerate}
			\item Applying Lemma \ref{lemma_texp_properties}, we obtain, 
			\begin{align*}
			\frac{1}{m} \sum_{i=1}^m  \left| \E e^{\i t u_i}  \right|^2 & \leq \frac{C}{|t|^2} \cdot \frac{1}{m} \sum_{i=1}^m (1+ |\hat{\lambda}_1| + |y|_i)^2 \leq \frac{C(K)}{|t|^2}.
			\end{align*} 
			\item The above bound tells us that for $|t| \geq \sqrt{2 C(K)}$, we have, 
			\begin{align*}
			\frac{1}{m} \sum_{i=1}^m  \left| \E e^{\i t u_i}  \right|^2 & \leq \frac{1}{2}.
			\end{align*}
			Applying Lemma \ref{cf_bound_requirement} in Appendix \ref{CF_appendix}, we can find a constant $0<\eta(K)<1$ depending only on $K$ such that, 
			\begin{align*}
			\frac{1}{m} \sum_{i=1}^m  \left| \E e^{\i t u_i}  \right|^2 & \leq (1-\eta(K)), \; \forall |t| \geq t_2.
			\end{align*}
		\end{enumerate}
		We can combine the above to bounds to control (3) as follows: 
		\begin{align*}
		(3) & = \sqrt{m} \int_{|t| \geq t_2} |\check{\psi}(t \sqrt{m})| \diff t \\
		& \leq \sqrt{m} \int_{|t| \geq t_2} \left( \frac{1}{m} \sum_{i=1}^m  \left| \E e^{\i t u_i}  \right|^2 \right)^{\frac{m}{2}} \diff t \\
		& \leq \sqrt{m} \cdot C(K) \int_{|t| \geq t_2} \left( \frac{1}{m} \sum_{i=1}^m  \left| \E e^{\i t u_i}  \right|^2 \right)^{\frac{m}{2}-1} \cdot \frac{1}{|t|^2} \diff t \\
		& \leq C(K) \cdot \sqrt{m} \cdot (1-\eta(K))^{\frac{m}{2}-1} \cdot \int_{|t| \geq t_2} \frac{1}{|t|^2} \diff t \\
		& \leq \frac{C(K)}{\sqrt{m}}.
		\end{align*}
	\end{description}
This concludes the proof of the proposition.
\end{proof}
\subsection{Proof of Proposition \ref{local_clt_num}}\label{local_clt_num_appendix}
\begin{proof}
Recall that the random variable $\bm S$ was defined as:
\begin{align*}
\bm S  = \sum_{k=1}^m \bm S_k, \; \bm S_{k}  = \begin{bmatrix} s_k & \sqrt{s_k s_{k}^\prime} e^{\i \theta_k} \\ \sqrt{s_k s_{k}^\prime} e^{-\i \theta_k} & r_k^\prime \end{bmatrix} \sim \TWis{\hat\lambda_2(q;\sigma)}{\hat\phi_2(q;\sigma)}{y_i},
\end{align*}
where $(\hat{\lambda}_2(q;\sigma),\hat{\phi}(q;\sigma))$ solved the concave variational problem:
\begin{align*}
(\hat{\lambda}_2(q; \sigma), \hat{\phi}(q; \sigma)) & \Mydef \arg \max_{(\lambda,\phi) \in \mathbb{R}} \left( 2  \lambda + q \phi -\hatE_Y  
\ln \ZTWis{\lambda}{\phi}{Y} \right).
\end{align*}
	Throughout this proof for easy of notation we will omit the dependence of quantities like $\hat{\lambda}_2(q; \sigma), \hat{\phi}(q;\sigma)$ and $\hat{\bm V}(q; \sigma)$ on $q, \sigma$ and denote them by $\hat{\lambda}_2,\hat{\phi},\hat{\bm V}$.
	Since the optimizer of the variational problem lies in a compact set, we know that $\hat{\lambda}_2,\hat{\phi}$ satisfy the first order optimality conditions:
	\begin{align*}
	2  & = \frac{1}{m} \sum_{i=1}^m \frac{\partial_{\lambda}\ZTWis{\hat{\lambda}_2}{\hat\phi}{y_i}}{\ZTWis{\hat{\lambda}_2}{\hat\phi}{y_i}}  \explain{(a)}{=} \frac{1}{m} \sum_{i=1}^m \E (s_i + s_i^\prime) \\
	q & = \frac{1}{m} \sum_{i=1}^m \frac{\partial_{\phi}\ZTWis{\hat{\lambda}_2}{\hat{\phi}}{y_i}}{\ZTWis{\hat{\lambda}_2}{\hat\phi}{y_i}}  \explain{(a)}{=} \frac{1}{m} \sum_{i=1}^m \E \sqrt{s_i s_i^\prime} \cos(\theta_i). \\
	\end{align*}
	In the steps marked (a), we used the formula for the normalizing constant $\ZTWis{\lambda}{\phi}{y}$ given in Definition \ref{titled_wishart_definition} to compute the partial derivatives. It is also clear from Definition \ref{titled_wishart_definition} that: 
	\begin{align*}
	\E s_i = \E s_i^\prime, \; \E \sqrt{s_i s_i^\prime} \sin(\theta) = 0.
	\end{align*}
	Hence the first order optimality conditions imply: 
	\begin{align*}
	\E \bm S & = m\bm Q.
	\end{align*}
	
	Next we define the centered random variables:
	\begin{align*}
	\check{\bm S}_i & = \bm S_i - \E \bm S_i, \; \check{\bm S} = \frac{\bm S - \E \bm S}{\sqrt{m}} = \frac{1}{\sqrt{m}} \sum_{i=1}^m \check{\bm S}_i.
	\end{align*} 
	Note that, 
	\begin{align*}
	\E \vec{\check{\bm S}} \vec{\hat{\bm S}}^\UH & = \frac{1}{m} \sum_{i=1}^m \VTWis{\hat{\lambda}_2}{\hat{\phi}}{y_i} = \hat{\bm V}.
	\end{align*}
	Let $\check{H}$ denote the density of $\check{\bm S}$. We note that it is sufficient to study the asymptotics of $\check{H}(\bm 0)$ since by the change of variable formula we have: 
	\begin{align*}
	H_{\hat{\lambda}_2, \hat{\phi}, \bm y}(m \bm Q) & = \frac{\check{H}(\bm 0)}{m^2}.
	\end{align*}
	In the remainder of the proof we focus on developing asymptotic expansions for $\check{H}$. 
	We define the characteristic function of $\check{\bm S}$: 
	\begin{align*}
	\check{\Psi}(\bm t) & = \E \exp \left( \i \ip{\bm t}{\vec{\check{\bm S}}} \right).
	\end{align*}
	By the Fourier Inversion formula (Lemma \ref{fourier_inversion_CF}) we have,
	\begin{align*}
	\check{H}(\bm U) & = \frac{1}{(2\pi)^4} \int_{\mathbb{R}^4} e^{-\i \ip{\bm t}{ \vec{\bm U}}} \check{\Psi}(\bm t)  \diff \bm t.
	\end{align*}
	Applying the inversion formula to $\gauss{\bm 0}{\hat{\bm V} }$ gives us: 
	\begin{align*}
	\frac{1}{\sqrt{(2\pi)^4 \det(\hat{\bm V})}} e^{-\frac{1}{2} \vec{\bm U}^\UH \hat{\bm V}^{-1} \vec{\bm U}} & = \frac{1}{(2\pi)^4} \int_{\mathbb{R}^4} e^{-\i \ip{\bm t}{ \vec{\bm U}}} e^{-\frac{1}{2} \bm t^\UH \hat{\bm V} \bm t}  \diff \bm t.
	\end{align*}
	Setting $\bm U = 0$ and computing the error between the above two displays we obtain:
	Appendix \ref{CF_appendix}) we have, 
	\begin{align*}
	&(2\pi)^4\left|\check{H}(\bm 0) - \frac{1}{\sqrt{(2\pi)^4 \det(\hat{\bm V})}}\right| =  \left| \int_{\mathbb{R}^4} e^{-\i \ip{\bm t}{ \vec{\bm U}}} \left( \check{\Psi}(\bm t) - e^{-\frac{1}{2} \bm t^\UH \hat{\bm V} \bm t} \right) \diff \bm  t  \right| \\
	& \explain{(a)}{\leq} \left( \underbrace{\int_{\|\bm t\| \leq t_1} \left| \check{\Psi}(\bm t) - e^{-\frac{1}{2} \bm t^\UH \hat{\bm V} \bm t} \right| \diff t}_{(1)} + \underbrace{\int_{t_1 \leq\|\bm t\| \leq t_2 \sqrt{m}} |\check{\Psi}(\bm t)| \diff t}_{(2)} + \underbrace{\int_{\|\bm t\| \geq t_2  \sqrt{m} } |\check{\Psi}(\bm t)|  \diff t}_{(3)} + \underbrace{\int_{\|\bm t\| \geq t_1}  e^{-\frac{1}{2} \bm t^\UH \hat{\bm V} \bm t} \diff t}_{(4)} \right). \\
	\end{align*}
	In the step marked (a), the cutoff parameters $t_1,t_2$ are arbitrary and will be fixed later. We will be able to choose $t_1, t_2$ such that the following bound holds:
	\begin{align*}
	(1) + (2) + (3) + (4) & \leq \frac{C(K) \cdot \ln^5(m)}{\sqrt{m}}.
	\end{align*}
	This gives us,
	\begin{align*}
	\left|\check{H}(\bm 0) - \frac{1}{\sqrt{(2\pi)^4 \det(\hat{\bm V})}}\right|  \leq \frac{C(K)\ln^5(m)}{\sqrt{m}} \implies \left|	H_{\hat{\lambda}_2, \hat{\phi}, \bm y}(m \bm Q) - \frac{1}{\sqrt{(2\pi m)^4 \det(\hat{\bm V})}}\right|  \leq \frac{C(K)\ln^5(m)}{m^2\sqrt{m}}
	\end{align*}
	which is the claim of this proposition. The remaining proof is devoted to the analysis of (1), (2), (3) and (4).
	\begin{description}
		\item[Analysis of (1): ] Recall $\check{\Psi}(\bm t) = \E e^{\i \ip{\bm t}{ \vec{\check{\bm S}}}}$ and $f(\bm x)=e^{\i\ip{\bm  t}{\bm x}}$ is bounded, $\|\bm t\|$-lipchitz function of $\bm x$. Applying the Berry-Eseen Inequality (Theorem \ref{berry_eseen_theorem}, Appendix \ref{CF_appendix}), we have, 
		\begin{align*}
		\left|\check{\Psi}(\bm t) - e^{-\frac{1}{2} \bm t^\UH \hat{\bm V} \bm t} \right| & \leq \frac{C \cdot (1+ \|\hat{\bm V}\|^{1/2}\|\bm t\|) \cdot \rho_3}{\sqrt{m \cdot \lambda_{\min}^3(\hat{\bm V})}}.
		\end{align*}
		In the above display, $C$ is a universal constant and $\rho_3$ is given by: 
		\begin{align*}
		\rho_3		 & = \frac{1}{m} \sum_{i=1}^m \E \|\vec{\bm S_i} - \E\vec{\bm S_i}\|^3 \\
		& \leq \frac{8}{m} \sum_{i=1}^m \E \|\vec{\bm S_i} \|^3 \\
		& \leq \frac{16}{m} \sum_{i=1}^m \E( s_i^3 + \E {s_i^\prime}^3) \\
		& \explain{(a)}{\leq} C \left(1 + |\hat{\lambda}_2|^3 + |\hat{\phi}|^3 + \frac{1}{m} \sum_{i=1}^m |y_i|^3 \right).
		\end{align*}
		In the step marked (a) we used the estimate on $\E s_i^3$ proved in Lemma \ref{tilted_wishart_properties}.
		Recalling the assumptions
		\begin{align*}
		K^{-1} \leq \lambda_{\min}(\hat{\bm V}) \leq \lambda_{\max}(\hat{\bm V}) \leq K, \; |\hat{\phi}| + |\hat{\lambda}_2| < K, \; \frac{1}{m} \sum_{i=1}^m |y_i|^3  \leq K ,
		\end{align*}
		
		we obtain,
		\begin{align*}
		\left|\check{\Psi}(\bm t) - e^{-\frac{1}{2} \bm t^\UH \hat{\bm V} \bm t} \right| & \leq \frac{C(K)\cdot (1+ \|\bm t\|)}{\sqrt{m}}
		\end{align*}
		Integrating the pointwise bound above we obtain: 
		\begin{align*}
		(1) & \leq \frac{C(K)\cdot (1+  t_1) \cdot  t_1^4}{\sqrt{m}}
		\end{align*}
		We set: 
		\begin{align*}
		t_1^2 & = \frac{4 \ln(m)}{\lambda_{\min}(\hat{\bm V})}
		\end{align*}
		This gives us: 
		\begin{align*}
		(1)   \leq \frac{C(K) \cdot \ln^5(m)}{\sqrt{m}}.
		\end{align*}
		\item[Analysis of (2):] Let $(\tilde{\bm S}_1, \tilde{\bm S}_2 \dots \tilde{\bm S}_m)$ be independent and identically distributed as $(\bm S_1,\bm S_2 \dots \bm S_m)$. Note that, 
		\begin{align*}
		\left| \E e^{\i \ip{\bm t}{\vec{\check{\bm S}_i}}} \right|^2  =  \left| \E e^{\i \ip{\bm t}{\vec{{\bm S}_i}}} \right|^2 
		=  \E e^{\i \ip{\bm t}{\vec{{\bm S}_i-\tilde{\bm S}_i}}} 
		\end{align*}
		Hence, 
		\begin{align*}
		\left| \check{\Psi}(\bm t) \right|^2  = \prod_{i=1}^m 	\left| \E e^{\i \ip{\bm t}{\vec{\check{\bm S}_i}}/\sqrt{m}} \right|^2 = \prod_{i=1}^m \E e^{\i \ip{\bm t}{\vec{{\bm S}_i-\tilde{\bm S}_i}}/\sqrt{m}}.
		\end{align*}
		By the Taylor's theorem for CF (Theorem \ref{taylors_cf}, Appendix \ref{CF_appendix}), we have, 
		\begin{align*}
		\E \exp\left(\i \frac{ \ip{\bm t}{\vec{{\bm S}_i-\tilde{\bm S}_i}}}{\sqrt{m}}\right)& = 1 - \frac{\E\ip{\bm t}{\vec{{\bm S}_i-\tilde{\bm S}_i}}^2 }{2m} + E_i, 
		\end{align*}
		where $|E_i|$ is controlled by: 
		\begin{align*}
		|E_i| &\leq \frac{\E|\ip{\bm t}{\vec{{\bm S}_i-\tilde{\bm S}_i}}|^3 }{6 m \sqrt{m}} \leq \frac{\|\bm t\|^3 \E \| \vec{{\bm S}_i-\tilde{\bm S}_i}\|^3}{6 m \sqrt{m}}.
		\end{align*}
		Now consider any $\|\bm t\| \leq t_2 \sqrt{m}$:
		\begin{align*}
		|\check{\Psi}(\bm t)|^2 & = \prod_{i=1}^m \left(  1 - \frac{\E\ip{\bm t}{\vec{{\bm S}_i-\tilde{\bm S}_i}}^2 }{2m} + E_i \right) \\& \explain{}{\leq} \prod_{i=1}^m  \left( 1 - \frac{\E\ip{\bm t}{\vec{{\bm S}_i-\tilde{\bm S}_i}}^2 }{2m}  + \frac{\|\bm t\|^3 \E \| \vec{{\bm S}_i-\tilde{\bm S}_i}\|^3}{6 m \sqrt{m}} \right) \\
		& \leq \exp \left( -  \sum_{i=1}^m \frac{\E\ip{\bm t}{\vec{{\bm S}_i-\tilde{\bm S}_i}}^2 }{2m}  + \frac{\|\bm t\|^3 \E \| \vec{{\bm S}_i-\tilde{\bm S}_i}\|^3}{6 m \sqrt{m}} \right).
		\end{align*}
		Next we observe that, 
		\begin{align*}
		\frac{1}{2m} \sum_{i=1}^m \E\ip{\bm t}{\vec{{\bm S}_i-\tilde{\bm S}_i}}^2  &= \bm t^\UH \hat{\bm V} \bm t.
		\end{align*}
		We set: 
		\begin{align*}
		t_2 & =  3 \lambda_{\min}(\hat{\bm V}) \cdot \left( \frac{1}{m} \sum_{i=1}^m \E \| \vec{{\bm S}_i-\tilde{\bm S}_i}\|^3\right)^{-1}.
		\end{align*}
		This ensures, for any $|t| \leq t_2 \sqrt{m}$, we have,
		\begin{align*}
		|\check{\Psi}(\bm t)|^2 & \leq \exp \left( -  \sum_{i=1}^m \frac{\E\ip{\bm t}{\vec{{\bm S}_i-\tilde{\bm S}_i}}^2 }{2m}  + \frac{\|\bm t\|^3 \E \| \vec{{\bm S}_i-\tilde{\bm S}_i}\|^3}{6 m \sqrt{m}} \right) \\ &\leq  \exp\left(- \frac{ \bm t^\UH \hat{\bm V} \bm t}{2}\right) \leq \exp \left( - \frac{\lambda_{\min}(\hat{\bm V}) \|\bm t\|^2}{2} \right).
		\end{align*}
		Consequently,
		\begin{align}
		\label{recycle_argument}
		(2) &= \int_{t_1 \leq\|\bm t\| \leq t_2 \sqrt{m}} |\check{\Psi}(\bm t)| \diff \bm t \\ 
		&\leq \int_{t_1 \leq\|\bm t\| \leq t_2 \sqrt{m}}\exp \left( - \frac{\lambda_{\min}(\hat{\bm V}) \|\bm t\|^2}{4} \right) \diff \bm t \\
		& \explain{(a)}{\leq} C \int_{t_1}^\infty \exp \left( - \frac{\lambda_{\min}(\hat{\bm V})}{4} l^2 \right) l^3 \diff l \\
		& \explain{(b)}{\leq} C(K) \cdot \left( \frac{\lambda_{\min}(\hat{\bm V})t^2_1}{4}  + 1\right) \cdot \exp \left( - \frac{\lambda_{\min}(\hat{\bm V}) t_1^2}{4} \right)
		\end{align}
		In the step marked (a), we converted the integral into polar coordinates from cartesian coordinates. In the step marked (b), we used Lemma \ref{truncated_gauss_integral} and  used the assumption that $\lambda_{\min}(\hat{\bm V}) \geq K^{-1}$. 
		Recalling that we set:
		\begin{align*}
		t_1^2 & = \frac{4 \ln(m)}{\lambda_{\min}(\hat{\bm V})},
		\end{align*}
		we obtain,
		\begin{align*}
		(2) & \leq \frac{C(K) \cdot \ln(m)}{m}.
		\end{align*}
		Finally, to wrap up this step, we note that there exists a finite positive constant $C(K)$ such that, 
		\begin{align*}
		t_2 & \geq \frac{1}{C(K)}.
		\end{align*}
		Indeed, 
		\begin{align*} \frac{1}{m} \sum_{i=1}^m \E \| \vec{{\bm S}_i-\tilde{\bm S}_i}\|^3 \leq \frac{C}{m} \sum_{i=1}^m \E |s_i |^3 \leq C \left(1 + |\lambda|^3+ |\phi|^3 + \frac{1}{m} \sum_{i=1}^m |y_i|^3 \right) \leq C(K),
		\end{align*}
		and,
		\begin{align*}
		t_2 & = 3 \lambda_{\min}(\hat{\bm V}) \cdot \left( \frac{1}{m} \sum_{i=1}^m \E \| \vec{{\bm S}_i-\tilde{\bm S}_i}\|^3\right)^{-1} \geq \frac{1}{C(K)}.
		\end{align*}
		\item[Analysis of (3): ] Recall the term (3) was given by: 
		\begin{align*}
		(3)  = \int_{\|\bm t\| \geq t_2  \sqrt{m} } |\check{\Psi}(\bm t)|  \diff \bm t = m^2 \int_{\|\bm t\| \geq t_2} |\check{\Psi}(\bm t \sqrt{m})| \diff \bm t.
		\end{align*}
		
		By AM-GM for non-negative real numbers we have, 
		\begin{align*}
		|\check{\Psi}(\bm t \sqrt{m})|^2  = \prod_{i=1}^m \left| \E e^{\i \ip{\bm t}{\vec{\bm S_i}}} \right|^2 \leq \left( \frac{1}{m} \sum_{i=1}^m   \left| \E e^{\i \ip{\bm t}{\vec{\bm S_i}}} \right|^2 \right)^m.
		\end{align*}
		We use two different strategies to further control the above bound:
		\begin{enumerate}
			\item Applying Lemma \ref{tilted_wishart_properties}, we obtain, 
			\begin{align}
			\label{multivariate_cf_bound_1}
			\frac{1}{m} \sum_{i=1}^m   \left| \E e^{\i \ip{\bm t}{\vec{\bm S_i}}} \right|^2 & \leq \frac{C}{\|\bm t\|^{\frac{2}{3}}} \cdot \frac{1}{m} \sum_{i=1}^m (1+ |\hat{\lambda}_2|^{20} + |\hat{\phi}|^{20} +  |y|_i^{20})^2 \leq \frac{C(K)}{\|\bm t\|^{\frac{2}{3}}}.
			\end{align} 
			\item The above bound tells us that for $\|\bm t\| \geq \sqrt{8 C^3(K)}$, we have, 
			\begin{align*}
			\frac{1}{m} \sum_{i=1}^m   \left| \E e^{\i \ip{\bm t}{\vec{\bm S_i}}} \right|^2 & \leq \frac{1}{2}.
			\end{align*}
			Applying Lemma \ref{cf_bound_requirement} in Appendix \ref{CF_appendix}, we can find a constant $0<\eta(K)<1$ depending only on $K$ such that, 
			\begin{align}
			\label{multivariate_cf_bound_2}
			\frac{1}{m} \sum_{i=1}^m   \left| \E e^{\i \ip{\bm t}{\vec{\bm S_i}}} \right|^2 & \leq (1-\eta(K)), \; \forall \; \|\bm t\| \geq t_2 \geq \frac{1}{C(K)}.
			\end{align}
		\end{enumerate}
		We can combine the above to bounds to control (3) as follows: 
		\begin{align*}
		(3) & =m^2 \int_{\|\bm t\| \geq t_2} |\check{\Psi}(\bm t \sqrt{m})| \diff \bm t \\
		& \leq m^2 \int_{\|\bm t\| \geq t_2} \left( \frac{1}{m} \sum_{i=1}^m   \left| \E e^{\i \ip{\bm t}{\vec{\bm S_i}}} \right|^2 \right)^{\frac{m}{2}} \diff \bm t \\
		& \explain{(a)}{\leq} m^2 \cdot C(K) \cdot \int_{\|\bm t\| \geq t_2} \left( \frac{1}{m} \sum_{i=1}^m   \left| \E e^{\i \ip{\bm t}{\vec{\bm S_i}}} \right|^2 \right)^{\frac{m}{2}-9} \cdot \frac{1}{\|\bm t\|^6} \diff \bm t \\
		& \explain{(b)}{\leq} C(K) \cdot m^2 \cdot (1-\eta(K))^{\frac{m}{2}-9} \cdot \int_{\|\bm t\| \geq t_2} \frac{1}{\|\bm t\|^6} \diff t \\
		& \explain{(c)}{\leq} \frac{C(K)}{\sqrt{m}} \cdot \int_{t_2} ^\infty \frac{1}{l^6} \cdot l^3 \diff l \\
		& \leq \frac{C(K)}{\sqrt{m}} 
		\end{align*}
		In the above display, in step (a), we utilized the bound in  \eqref{multivariate_cf_bound_1}. In the step marked (b) we utilized the bound in \eqref{multivariate_cf_bound_2}. In the equation marked (c) we converted the integral into polar coordinates and checked that the integral was finite. 
		\item [Analysis of (4): ]
		We recall that: 
		\begin{align*}
		(4) & = \int_{\|\bm t\| \geq t_1}  e^{-\frac{1}{2} \bm t^\UH \hat{\bm V} \bm t} \diff t \\
		& \leq \int_{\|\bm t\| \geq t_1}  e^{-\frac{\lambda_{\min}(\hat{\bm V})}{2} \|\bm t\|^2} \diff t.
		\end{align*}
		After this, we can exactly repeat the arguments following   \eqref{recycle_argument} and obtain, 
		\begin{align*}
		(4) & \leq \frac{C(K)\ln(m)}{m}.
		\end{align*}
	\end{description}
This concludes the proof.
\end{proof}

\section{Concentration Analysis}
\label{concentration_appendix}
This section is devoted to proving the concentration result Proposition \ref{concentration_proposition}.
Throughout this section, we will use $Y$ to denote the random variable $|Z|^2 + \sigma \epsilon$, where $Z \sim \cgauss{0}{1}$ and $\epsilon \sim \gauss{0}{1}$. Hence, for any $f: \mathbb R \rightarrow \mathbb R$, $\E f(Y) = \E f(|Z|^2 + \sigma \epsilon)$. We also recall the $\hatE$ notation, for any real valued function $f$ on $\mathbb R$:
\begin{align*}
\hatE f(Y) & \Mydef \frac{1}{m} \sum_{i=1}^m f(y_i),
\end{align*}
where $y_1, y_2 \dots , y_m$ are the observations in the phase retrieval problem. The main intuition behind all of the results in this section is that the empirical measure of the measurements converges to the law of $Y$. Hence for a large class test functions $f$, $\hatE f(Y) \approx \E f(Y)$. This intuition is made rigorous in terms of a general Weak Law of Large Numbers (WLLN) and a Uniform WLLN (ULLN) for the empirical measure of the measurements in Section \ref{generalities}. We then use these general results to prove Proposition \ref{concentration_proposition} in Section \ref{concentration_proposition_proof}.

\subsection{A General Uniform Weak Law of Large Numbers}
\label{generalities}
The following proposition establishes a  weak law of large numbers (WLLN) for empirical averages of  measurements $y_1, y_2 \dots y_m$ in the phase retrieval model. 
\begin{proposition}[A WLLN] \label{wlln} Let $y_1, y_2 \dots y_m$ be the m measurements from the Phase Retrieval model. Let $f: \mathbb R \rightarrow \mathbb R$ satisfy the local lipchitz assumption:
	\begin{align*}
	|f(a) - f(b)| & \leq L \cdot (1+ |a|^k + |b|^k ) \cdot |a-b|,
	\end{align*}
	for some $L>0, \; k \in \mathbb N$. 
	Then we have, 
	\begin{align*}
	\frac{1}{m} \sum_{i=1}^m f(y_i) & \explain{P}{\rightarrow} \E f(|Z|^2 + \sigma \epsilon).
	\end{align*}
	In the above display, $Z, \epsilon$ are independent r.v.s with the distributions: $Z \sim \cgauss{0}{1}, \; \epsilon \sim \gauss{0}{1}$.
\end{proposition}
\begin{proof}
	Recall that in the phase retrieval model, we have, 
	\begin{align*}
	(y_1, y_2 \dots y_m) & \explain{d}{=} \left( \frac{m |g_1|^2}{\|\bm g\|^2} + \sigma\epsilon_1, \frac{m |g_2|^2}{\|\bm g\|^2} + \sigma\epsilon_2 \dots \frac{m |g_m|^2}{\|\bm g\|^2} + \sigma\epsilon_m  \right).
	\end{align*}
	In the above display $\bm g$ and $\bm \epsilon$ are independent with $\bm g \sim \cgauss{\bm 0}{\bm I_m}$ and $\bm \epsilon \sim \gauss{\bm 0}{\bm I_m}$. To obtain the claim of the proposition we write,
	\begin{align*}
	\frac{1}{m} \sum_{i=1}^m f(y_i) - \E f(|Z|^2 + \sigma\epsilon)  &\explain{d}{=} \frac{1}{m} \sum_{i=1}^m f \left( \frac{m |g_i|^2}{\|\bm g\|^2} +\sigma \epsilon_i \right) - \E f(|Z|^2 + \sigma \epsilon) \\
	& = (1) + (2).
	\end{align*}
	where the terms $(1),(2)$ are defined below:
	\begin{align*}
	(1) & \Mydef \left(\frac{1}{m} \sum_{i=1}^m f \left( \frac{m |g_i|^2}{\|\bm g\|^2} + \sigma \epsilon_i \right) - \frac{1}{m} \sum_{i=1}^m f \left( |g_i|^2 + \sigma \epsilon_i \right) \right) , \\
	(2) & \Mydef \left(\frac{1}{m} \sum_{i=1}^m f \left( |g_i|^2 + \sigma \epsilon_i \right)- \E f(|Z|^2 + \sigma \epsilon) \right).
	\end{align*}
	Note that, 
	\begin{align*}
	(2) & \explain{P}{\rightarrow} 0,
	\end{align*}	
	by WLLN for sums of i.i.d. random variables. On the other hand, by the local lipchitz assumption on $f$:
	\begin{align}
	\label{analysis_1_wlln}
	(1) & \leq L \cdot \left( \frac{m}{\|\bm g\|^2} - 1 \right) \cdot \frac{1}{m} \sum_{i=1}^m |g_i|^2 \left( 1 + \frac{m^k |g_i|^{2k}}{\|\bm g\|^{2k}} + |g_i|^{2k} \right) \nonumber \\
	& = L \cdot  \left( \frac{m}{\|\bm g\|^2} - 1 \right) \cdot \left( \frac{1}{m} \sum_{i=1}^m |g_i|^2 + \left(\left(\frac{m}{\|\bm g\|^2} \right)^k +1 \right) \frac{1}{m} \sum_{i=1}^m |g_i|^{2k+2}  \right)
	\end{align}
	By WLLN and continuous mapping theorem: 
	\begin{align*}
	\frac{m}{\|\bm g\|^2} - 1 & \explain{P}{\rightarrow} 0 \\
	\frac{1}{m} \sum_{i=1}^m |g_i|^2 & \explain{P}{\rightarrow } \E|Z|^2 < \infty, \\
	\left(\left(\frac{m}{\|\bm g\|^2} \right)^k +1 \right) \cdot \frac{1}{m} \sum_{i=1}^m |g_i|^{2k+2} & \explain{P}{\rightarrow} \left( \frac{1}{(\E|Z|^2)^k} + 1 \right) \cdot \E|Z|^{2k+2}  < \infty. 
	\end{align*}
	Hence $(1) \explain{P}{\rightarrow} 0$. This proves the claim of the proposition.
\end{proof}
The following proposition proves a Uniform Law of Large Numbers (ULLN) for empirical averages of the measurements $y_1, y_2 \dots y_m$ using some results from empirical process theory \citep{vanweak}.
\begin{proposition}[A Uniform Law of Large Numbers]
	\label{ulln} Let $\mathcal{F}_T$ be a collection of functions $f_{\bm t}: \mathbb{R} \rightarrow \mathbb{R}$ indexed by a parameter $\bm t$ which takes values in the set $T$, a bounded subset of $\mathbb R^k$. Suppose that the collection $\mathcal{F}_T$ satisfies the following lipchitz conditions:
	\begin{align*}
	\text{Lipchitz in parameter: } & | f_{\bm t}(y) - f_{{\bm s}}(y) |  \leq L \cdot  \|\bm t - \bm s\| \cdot (1+|y|^l) \; \forall \; \bm t, \bm s \in T, \; y \in \mathbb{R}, \\
	\text{Lipchitz in argument: } & |f_{\bm t}(y) - f_{\bm t}(y^\prime)|  \leq L \cdot |y-y^\prime| \cdot (|y|^l + |y^\prime|^l + 1) \; \forall \; \bm t \in T, \; y,y^\prime \in \mathbb{R}.
	\end{align*}
	for some $L>0, l \in \mathbb{N}$. Then we have, 
	\begin{align*}
	\sup_{\bm t \in T} \left( \frac{1}{m} \sum_{i=1}^m f_{\bm t}(y_i) - \E f_{\bm t}(|Z|^2 + \sigma \epsilon) \right) & \explain{P}{\rightarrow} 0.
	\end{align*}
\end{proposition}
\begin{proof}
	As in the proof of Proposition \ref{wlln}, we have the decomposition: 
	\begin{align*}
	\frac{1}{m} \sum_{i=1}^m f(y_i) - \E f(|Z|^2 + \sigma \epsilon) 
	& = (1) + (2).
	\end{align*} 
	where,
	\begin{align*}
	(1) & \Mydef \left(\frac{1}{m} \sum_{i=1}^m f_{\bm t} \left( \frac{m |g_i|^2}{\|\bm g\|^2} + \sigma \epsilon_i \right) - \frac{1}{m} \sum_{i=1}^m f_{\bm t} \left( |g_i|^2 + \sigma \epsilon_i \right) \right) , \\
	(2) & \Mydef \left(\frac{1}{m} \sum_{i=1}^m f_{\bm t} \left( |g_i|^2 + \sigma \epsilon_i \right)- \E f_{\bm t}(|Z|^2 + \sigma \epsilon) \right) .
	\end{align*}
	The analysis (1) is exactly the same as in Proposition \ref{wlln}. The upper bound in  \eqref{analysis_1_wlln} holds uniformly over $T$ and hence,
	\begin{align*}
	\sup_{\bm t \in T} \; (1) & \explain{P}{\rightarrow} 0.
	\end{align*}
	For the term (2), we appeal to standard empirical process theory results from \citet{vanweak}. By Theorem 2.7.11 of \citet{vanweak}, the function class $\mathcal{F}_T$ has bounded bracketting number. Consequently, by Theorem 2.4.1 of \citet{vanweak}, $\mathcal{F}_T$ is Glivenko-Cantelli, that is, 
	\begin{align*}
	\sup_{\bm t \in T} \; (2) & \explain{P}{\rightarrow} 0.
	\end{align*}
	This concludes the proof of the proposition. 
\end{proof}
Next we will apply the ULLN of Proposition \ref{ulln} to obtain uniform convergence of empirical averages of the log-normalizing constants and moments of the Tilted Exponential and Wishart distributions. In particular, we recall the definitions: 
\begin{align*}
\ln \ZTexp{\lambda}{y} & \Mydef \ln \E_{E \sim \Exp{1}} e^{\lambda E} \gpdf(y-E), \\
\ln \ZTWis{\lambda}{\phi}{y} & \Mydef \ln \E_{\bm g \sim \cgauss{\bm 0}{\bm I_2}} e^{\lambda(|g_1|^2 + |g_2|^2) + \phi \Re(g_1 \bar{g}_2)} \gpdf(y - |g_1|^2) \gpdf(y-|g_2|^2).
\end{align*}
For any $a,b,c,d \in \mathbb N$ we also define the moments of the tilted exponential and wishart distributions: 
\begin{align*}
\MTexp{\lambda}{y}{a} & \Mydef \E T^j, \; T \sim \Texp{\lambda}{y}  \\
\MTWis{\lambda}{\phi}{y}{a,b,c,d} & \Mydef \E S_{11}^a  \Re(S_{12})^b \Im(S_{12})^c S_{22}^d, \; \bm S \sim \TWis{\lambda}{\phi}{y}.
\end{align*}
Recalling the Definitions \ref{Texp_definition} and \ref{titled_wishart_definition}, we have, 
\begin{align*}
\MTexp{\lambda}{y}{a} & =  \frac{\E_{E \sim \Exp{1}} E^a e^{\lambda E} \gpdf(y-E)}{\E_{E \sim \Exp{1}} e^{\lambda E} \gpdf(y-E)}, \\
\MTWis{\lambda}{\phi}{y}{a,b,c,d} & = \frac{\E_{\bm g \sim \cgauss{\bm 0}{\bm I_2}} |g_1|^{2a} \Re(g_1 \bar{g}_2)^b \Im(g_1 \bar{g}_2)^c |g_2|^{2d} e^{\lambda(|g_1|^2 + |g_2|^2) + \phi \Re(g_1 \bar{g}_2)} \gpdf(y - |g_1|^2) \gpdf(y-|g_2|^2)}{\E_{\bm g \sim \cgauss{\bm 0}{\bm I_2}} e^{\lambda(|g_1|^2 + |g_2|^2) + \phi \Re(g_1 \bar{g}_2)} \gpdf(y - |g_1|^2) \gpdf(y-|g_2|^2)}.
\end{align*}
The following corollary applies the obtained ULLN to the above functions to obtain uniform convergence for these functions. 
\begin{corollary}[Uniform Convergence of Log-Normalizing Constants and Moments] 
	\label{ulln_corollary}
	For any $R>0$ and  $a,b,c,d \in \mathbb{N}$, we have,
	\begin{align*}
	1) &\sup_{|\lambda| \leq R} \left( \frac{1}{m} \sum_{i=1}^m \ln \ZTexp{\lambda}{y_i} - \E_{Z,\epsilon} \ln \ZTexp{\lambda}{|Z|^2 + \sigma \epsilon} \right)  \explain{P}{\rightarrow} 0, \\
	2) & \sup_{|\lambda| + |\phi|  \leq R} \left( \frac{1}{m} \sum_{i=1}^m \ln \ZTWis{\lambda}{\phi}{y_i} - \E_{Z,\epsilon} \ln \ZTWis{\lambda}{\phi}{|Z|^2 + \sigma \epsilon} \right)  \explain{P}{\rightarrow} 0, \\
	3) & \sup_{|\lambda| \leq R} \left( \frac{1}{m} \sum_{i=1}^m \MTexp{\lambda}{y_i}{a} - \E_{Z,\epsilon} \MTexp{\lambda}{|Z|^2 + \sigma \epsilon}{a} \right)  \explain{P}{\rightarrow} 0, \\
	4) & \sup_{|\lambda| + |\phi| \leq R} \left( \frac{1}{m} \sum_{i=1}^m \MTWis{\lambda}{\phi}{y_i}{a,b,c,d} - \E_{Z,\epsilon} \MTWis{\lambda}{\phi}{|Z|^2 + \sigma \epsilon}{(a,b,c,d)} \right)  \explain{P}{\rightarrow} 0.
	\end{align*}
\end{corollary}
\begin{proof}
	In order to prove the corollary, we just need to verify the lipchitz conditions in Proposition \ref{ulln}. In order to do so, we observe that,
	\begin{align*}
	\frac{\partial}{\partial y} \ln \ZTexp{\lambda}{\phi}  =  \frac{\MTexp{\lambda}{y}{1}-y}{\sigma^2}, \; \frac{\partial}{\partial \lambda} \ln \ZTexp{\lambda}{\phi}  = \MTexp{\lambda}{y}{1}.
	\end{align*}
	The moments of the Tilted Exponential distribution are bounded in Lemma \ref{lemma_texp_properties}. Using this we obtain, 
	\begin{align*}
	\max_{|\lambda| \leq R} \left| \frac{\partial}{\partial y} \ln \ZTexp{\lambda}{y}  \right| \leq C(R + |y|), \; 
	\max_{|\lambda| \leq R} \left| \frac{\partial}{\partial \lambda} \ln \ZTexp{\lambda}{y}  \right| \leq C(R + |y|).
	\end{align*}
	Integrating these derivative bounds gives us the following lipchitz estimates: 
	\begin{align*}
	\left|  \ln \ZTexp{\lambda}{y} - \ln   \ZTexp{\lambda}{y^\prime} \right| & \leq C \cdot (R+ |y| + |y^\prime|) \cdot  |y - y^\prime| \; \forall \; |\lambda| \leq R, \; y,y^\prime \in \mathbb R, \\
	\left|  \ln \ZTexp{\lambda}{y} -   \ln \ZTexp{\lambda^\prime}{y} \right| & \leq C \cdot (R + |y|) \cdot |\lambda - \lambda^\prime| \; \forall \; |\lambda| \leq R, \; |\lambda^\prime| \leq R, \; y \in \mathbb R,
	\end{align*}
	which verifies the assumptions of Proposition \ref{ulln} and hence (1) follows. Likewise the uniform convergence in (3) follows from the observation: 
	\begin{align*}
	\frac{\partial}{\partial y} \MTexp{\lambda}{y}{a}  &= \frac{\MTexp{\lambda}{y}{a+1} - \MTexp{\lambda}{y}{a}\MTexp{\lambda}{y}{1}}{\sigma^2}, \\ \frac{\partial}{\partial \lambda} \MTexp{\lambda}{y}{a} &= \MTexp{\lambda}{y}{a+1} - \MTexp{\lambda}{y}{a}\MTexp{\lambda}{y}{1}.
	\end{align*}
	The proofs of (2) and (4) are analogous and rely on moment bounds for the tilted wishart distribution given in Lemma \ref{tilted_wishart_properties}.
\end{proof}
\subsection{Proof of Proposition \ref{concentration_proposition}}
\label{concentration_proposition_proof}
We now present the proof of Proposition \ref{concentration_proposition}.
\begin{proof}
	Since polynomial functions are locally lipchitz, the claim (1) follows from the WLLN proved in Proposition \ref{wlln}. Item (2) is a special case of Corollary \ref{ulln_corollary}. The proofs of items (3-4) is very similar to (and easier) items (5-6) and is omitted. Hence we focus on proving claims 5-8. Define the concave (in $\lambda,\phi$) potential functions:
	\begin{align*}
	V_2(\lambda,\phi; q) & \Mydef 2\lambda  + \phi q - \E_Y
	\ln \ZTWis{\lambda}{\phi}{Y},  \\
	\hat V_2(\lambda,\phi; q) & \Mydef 2\lambda  + \phi q - \hatE_Y
	\ln \ZTWis{\lambda}{\phi}{Y}.
	\end{align*} 
	The potential functions are important because:
	\begin{align*}
	(\lambda_2(q;\sigma), \phi(q;\sigma))  = \arg \max_{\lambda,\phi \in \mathbb R} V_2(\lambda,\phi ; q), \; \Xi_2(q;\sigma) = \max_{\lambda,\phi \in \mathbb R} V_2(\lambda,\phi ; q).
	\end{align*}
	And likewise,
	\begin{align*}
	(\hat\lambda_2(q;\sigma), \hat\phi(q;\sigma))  = \arg \max_{\lambda,\phi \in \mathbb R} \hat V_2(\lambda,\phi ; q), \; \hat \Xi_2(q;\sigma) = \max_{\lambda,\phi \in \mathbb R} \hat V_2(\lambda,\phi ; q).
	\end{align*}
	The proof of this proposition relies on coercivity estimates for the above potential functions which have been proved in Appendix \ref{appendix_variational_analysis}.
	
	For the ease of notation, in this proof we will short hand $\Xi_2(q;\sigma),\hat{\Xi}_2(q;\sigma), \lambda_2(q;\sigma), \hat{\lambda}_2(q;\sigma), \phi(q;\sigma)$ and $\hat{\phi}(q;\sigma)$ as $\Xi_2(q),\hat{\Xi}_2(q),\lambda_2(q), \hat \lambda_2(q), \phi(q)$ and $\hat{\phi}(q)$, omitting the dependence on $\sigma$. We consider each of the claims (5-8) one by one:
	\begin{enumerate}
		\setcounter{enumi}{4}
		\item In Proposition \ref{nr_variational_prop} (Appendix \ref{appendix_variational_analysis}), we have shown that the solutions to the variation problems lie in the compact intervals: 
		\begin{align*}
		|\lambda_2(q)| + |\phi(q)| &\leq C \left(1 + q + \frac{1}{1-q} \right) \cdot (\E |Y|^2 + 1), \\
		|\hat\lambda_2(q)| + |\hat\phi(q)| &\leq C \left(1 + q + \frac{1}{1-q} \right) \cdot (\hatE |Y|^2 + 1) 
		\end{align*}
		On the other hand we know from Proposition \ref{wlln} that, 
		\begin{align*}
		\hatE Y^2 &\explain{P}{\rightarrow} \E Y^2 < \infty.
		\end{align*}
		Consequently, we can find constant $R$ that depends only on $\eta,\sigma$ such that, 
		\begin{align*}
		\max_{0 \leq q \leq 1-\eta} |\lambda_2(q)| + |\phi(q)|  \leq R, \; 
		\P \left( \max_{0 \leq q \leq 1-\eta} |\hat\lambda_2(q)| + |\hat \phi(q)| > R \right)  \rightarrow 0.
		\end{align*}
		For instance taking $R$ as:
		\begin{align*}
		R & = C \left(2 +  \frac{1}{1-\eta} \right)(2 + \E Y^2),
		\end{align*}
		is sufficient. This proves item (5) of the proposition. 
		\item We upper bound $\Xi_2(q) - \hat \Xi_2(q)$ and $\hat \Xi_2(q) - \Xi_2(q)$ separately:
		\begin{align*}
		\Xi_2(q) - \hat \Xi_2 (q) & = V_2(\lambda_2(q),\phi(q);q) - \hat V(\hat \lambda_2(q), \hat \phi(q);q) \\
		& =   V_2(\lambda_2(q),\phi(q);q) - \hat V_2(\lambda_2(q), \phi(q);q) + \underbrace{\hat V_2(\lambda_2(q), \phi(q);q)- \hat V_2(\hat \lambda_2(q), \hat \phi(q);q)}_{\leq 0} \\
		& \leq  V_2(\lambda_2(q),\phi(q);q) - \hat V_2(\lambda_2(q), \phi(q);q)  \\
		& \leq \sup_{q \in [0,1-\eta],|\lambda|+|\phi| \leq R} |V_2(\lambda,\phi;q) - \hat V_2(\lambda,\phi;q)|.
		\end{align*}
		Analogously, we can obtain $\hat\Xi_2(q) -  \Xi_2 (q) \leq \sup_{q \in [0,1-\eta],|\lambda| +|\phi| \leq R} |V_2(\lambda,\phi;q) - \hat V_2(\lambda,\phi;q)| $. Consequently we have, 
		\begin{align*}
		\sup_{q \in [0,1-\eta]} |\Xi_2(q) - \hat \Xi_2 (q)| & \leq \sup_{q \in [0,1-\eta],\lambda,\phi \in \mathbb R} |V_2(\lambda,\phi;q) - \hat V_2(\lambda,\phi;q)| \\
		& = \sup_{\lambda, \phi : |\lambda| + |\phi| \leq R} \left|\E_Y\ln \ZTWis{\lambda}{\phi}{Y} - \hatE_Y \ln \ZTWis{\lambda}{\phi}{Y} \right| \\&\explain{P}{\rightarrow} 0.
		\end{align*}
		In the last step we appealed to Corollary \ref{ulln_corollary}. This concludes the proof of item (6). 
		\item For the purpose of demonstrating convergence in probability it is sufficient to restrict ourselves to the event: 
		\begin{align*}
		\max_{0 \leq q \leq 1-\eta} |\hat\lambda_2(q)| + |\hat\phi(q)| \leq R,
		\end{align*}
		since this event occurs with probability tending to $1$. 
		Proposition \ref{nr_variational_prop} shows that the function \\$\E_Y
		\ln \ZTWis{\lambda}{\phi}{Y}$ is strongly convex on compact intervals. Hence for some universal constant $C<\infty$, we have,
		\begin{align*}
		V_2(\lambda, \phi;q) & \leq V_2(\lambda_2(q), \phi(q);q) - \frac{1}{C} \cdot (|\lambda - \lambda_2(q)|^2 + |\phi - \phi(q)|^2) \; \forall \lambda, \phi: \;  |\lambda| + |\phi| \leq R, \; \forall q \in [0, 1-\eta].
		\end{align*}
		Applying the strong convexity estimate to $\lambda = \hat \lambda_2(q), \; \phi = \hat \phi(q)$ gives us: 
		\begin{align*}
		|\hat \lambda_2(q) - \lambda_2 (q) |^2 + |\hat \phi(q) - \phi(q)|^2 &\leq C (V_2(\lambda_2(q), \phi(q);q)- V_2(\hat \lambda_2(q), \hat \phi(q);q)) \\
		& =C \cdot \left(  (1) + (2) + (3) \right).
		\end{align*}
		In the above display, we defined the terms (1), (2) and (3) as: 
		\begin{align*}
		(1) & = V_2(\lambda_2(q), \phi(q);q)- \hat V_2( \lambda_2(q), \phi(q);q), \\
		(2) & = \hat V_2( \lambda_2(q),  \phi(q);q) - \hat V_2(\hat \lambda_2(q), \hat \phi(q);q), \\
		(3) & = \hat V_2(\hat \lambda_2(q), \hat \phi(q);q) - V_2(\hat \lambda_2(q), \hat \phi(q);q).
		\end{align*}
		Since $(\hat \lambda_2 (q), \hat \phi (q))$ maximizes $\hat V_2(\lambda,\phi;q)$, we have,
		\begin{align*}
		(2) & \leq 0.
		\end{align*}
		On the other hand, both (1) and (2) can be bounded by:
		\begin{align*}
		(1)  &\leq \sup_{\lambda, \phi : |\lambda| + |\phi| \leq R, q \in [0,1-\eta]}  \left| V_2(\lambda, \phi;q) - \hat V_2(\lambda,\phi;q) \right|, \\ (2) &\leq \sup_{\lambda, \phi : |\lambda| + |\phi| \leq R, q \in [0,1-\eta]}  \left| V_2(\lambda, \phi;q) - \hat V_2(\lambda,\phi;q) \right|.
		\end{align*}
		Hence we have obtained,
		\begin{align*}
		|\hat \lambda_2(q) - \lambda_2 (q) |^2 + |\hat \phi(q) - \phi(q)|^2 \leq 2C \cdot \sup_{\lambda, \phi : |\lambda| + |\phi| \leq R, q \in [0,1-\eta]}  \left| V_2(\lambda, \phi;q) - \hat V_2(\lambda,\phi;q) \right|.
		\end{align*}
		Corollary \ref{ulln_corollary} gives us the uniform convergence: 
		\begin{align*}
		&\sup_{\lambda, \phi : |\lambda| + |\phi| \leq R, q \in [0,1-\eta]}  \left| V_2(\lambda, \phi;q) - \hat V_2(\lambda,\phi;q) \right| =\\&\hspace{4cm} \sup_{\lambda, \phi : |\lambda| + |\phi| \leq R} \left|\E_Y\ln \ZTWis{\lambda}{\phi}{Y} - \hatE_Y \ln \ZTWis{\lambda}{\phi}{Y} \right|\explain{P}{\rightarrow} 0.
		\end{align*}
		Hence we obtain,
		\begin{align*}
		\sup_{q \in [0,1-\eta]}  |\hat \lambda_2(q) - \lambda_2(q) |^2 + |\hat \phi(q) - \phi(q)|^2 \explain{P}{\rightarrow} 0.
		\end{align*}
		This shows claim (7) of the proposition.
		\item A simple computation shows that: 
		\begin{align*}
		\frac{\diff^2 \Xi_2(q)}{\diff q^2} -\frac{\diff^2 \hat\Xi_2(q)}{\diff q^2} = \bm e_2^\UH \left( \nabla^2_{\lambda,\phi} V_2(\lambda_2(q), \phi(q);q)^{-1} - \nabla^2_{\lambda,\phi} \hat V_2(\hat \lambda_2(q), \hat \phi(q);q)^{-1}   \right) \bm e_2.
		\end{align*}
		Hence,
		\begin{align*}
		\sup_{q \in [0,1-\eta]} \left|\frac{\diff^2 \Xi_2(q)}{\diff q^2} -\frac{\diff^2 \hat\Xi_2(q)}{\diff q^2} \right| & \leq \sup_{q \in [0,1-\eta]}\left\| \nabla^2_{\lambda,\phi} V_2(\lambda_2(q), \phi(q);q)^{-1} - \nabla^2_{\lambda,\phi} \hat V_2(\hat \lambda_2(q), \hat \phi(q);q)^{-1} \right\|
		\end{align*}
		Hence it is sufficient to show that,
		\begin{align*}
		\sup_{q \in [0,1-\eta]}\left\| \nabla^2_{\lambda,\phi} V_2(\lambda_2(q), \phi(q);q)^{-1} - \nabla^2_{\lambda,\phi} \hat V_2(\hat \lambda_2(q), \hat \phi(q);q)^{-1} \right\| & \explain{P}{\rightarrow } 0.
		\end{align*}
		By triangle inequality, we can write,
		\begin{align*}
		\sup_{q \in [0,1-\eta]}\left\|\nabla^2 \hat V_2( \hat \lambda_2 (q), \hat \phi (q);q) - \nabla^2 \ V_2( \lambda_2 (q),  \phi (q);q)\right\|  &\leq (1) + (2),
		\end{align*}
		where we define the terms (1) and (2) as: 
		\begin{align*}
		(1) & \Mydef\sup_{q \in [0,1-\eta]} \left\|\nabla^2 \hat V_2( \hat \lambda_2 (q), \hat \phi (q);q) - \nabla^2 \ V_2( \hat\lambda_2 (q),  \hat\phi (q);q)\right\|, \\
		(2) & \Mydef \sup_{q \in [0,1-\eta]} \left\|\nabla^2 V_2( \hat \lambda_2 (q), \hat \phi (q);q) - \nabla^2 \ V_2( \lambda_2 (q),  \phi (q);q)\right\|.
		\end{align*}
		We control the first term as follows: 
		\begin{align*}
		(1) & \leq \sup_{q \in [0,1-\eta], \lambda, \phi: |\lambda| + |\phi| \leq R} \left\|\nabla^2 \hat V_2( \lambda ,  \phi;q ) - \nabla^2 \ V_2( \lambda ,  \phi;q)\right\| \\
		& = \sup_{\lambda, \phi: |\lambda| + |\phi| \leq R} \| \nabla^2 \E_Y\ln \ZTWis{\lambda}{\phi}{Y} - \nabla^2 \hatE_Y \ln \ZTWis{\lambda}{\phi}{Y} \|
		\end{align*}
		Noting that the entries of matrix  $\nabla^2_{\lambda, \phi}\ln \ZTWis{\lambda}{\phi}{Y}$ are moments of the Tilted Wishart distribution and appealing to Corollary \ref{ulln_corollary} gives us the uniform convergence:
		\begin{align}
		\label{Pconv_hessian_uniform_eq}
		(1) \leq \sup_{q \in [0,1-\eta], \lambda, \phi: |\lambda| + |\phi| \leq R} \left\|\nabla^2 \hat V_2( \lambda ,  \phi;q ) - \nabla^2 \ V_2( \lambda ,  \phi;q)\right\| \explain{P}{\rightarrow} 0.
		\end{align}
		To control the second term, we first note that $\nabla^2 V_2(\lambda, \phi;q)$ is independent of $q$. It is also easy to check that it is locally lipchitz of $\lambda,\phi$, consequently we have the estimate,
		\begin{align*}
		\left\|\nabla^2 V_2( \hat \lambda_2 (q), \hat \phi (q);q) - \nabla^2 \ V_2( \lambda_2(q),  \phi (q);q)\right\| & \leq C \left( |\lambda_2(q) - \hat{\lambda}_2(q)| + |\phi(q) - \hat{\phi}(q)| \right),
		\end{align*}
		for some constant $C$ depending only on $R$ (in particular, $C$ does not depend on $q$). Combining this with the conclusion obtained in item (4) of the lemma gives us: 
		\begin{align*}
		(2) & \explain{P}{\rightarrow} 0.
		\end{align*}
		Hence we  have,
		\begin{align}
		\label{unif_conv_hess_eq}
		\sup_{q \in [0,1-\eta]}\left\|\nabla^2 \hat V_2( \hat \lambda_2 (q), \hat \phi (q);q) - \nabla^2 \ V_2( \lambda_2 (q),  \phi (q);q)\right\|  & \explain{P}{\rightarrow } 0.
		\end{align}
		In order to obtain the analogous result for the inverse-hessian, we note that by Proposition \ref{nr_variational_prop}, $V_2(\lambda,\phi;q)$ is strongly concave on compact sets. Furthermore, $\nabla^2 V_2(\lambda,\phi;q)$ does not depend on $q$. Hence we have, 
		\begin{align*}
		\lambda_{\max}(\nabla^2 V_2(\lambda,\phi;q)) & \leq - \frac{1}{C}, \; \forall |\lambda| + |\phi| \leq R, \forall q,
		\end{align*}
		for a large enough universal constant $C$. Recalling the uniform convergence in \eqref{Pconv_hessian_uniform_eq}, we have, 
		\begin{align*}
		\P \left( \max_{\lambda,\phi: |\lambda| + |\phi| \leq R} \lambda_{\max}(\nabla^2 \hat V_2(\lambda,\phi;q))  \leq - \frac{1}{2C} \right) & \rightarrow 1.
		\end{align*}
		Since both $V,\hat V$ are concave functions (c.f. Proposition \ref{nr_variational_prop}), we have, 
		\begin{align}
		\label{tightness_op_norm_eq}
		\sup_{q \in [0,1-\eta]}\| \nabla^2 V_2(\lambda_2(q),\phi(q);q)^{-1} \|_{op} = O(1), \; \sup_{q \in [0,1-\eta]} \| \nabla^2 \hat V_2(\hat \lambda_2(q),\hat \phi(q);q)^{-1} \|_{op} = O_P(1). 
		\end{align}
		\citet{wedin1973perturbation} has shown the following perturbation bounds for matrix inverse for any two invertible matrices $A,B$:
		\begin{align*}
		\|A^{-1} - B^{-1}\| & \leq \sqrt{2} \cdot \max(\|A^{-1}\|_{op}, \|B^{-1}\|_{op} ) \cdot \|A - B\|.
		\end{align*}
		Combining the tightness result in \eqref{tightness_op_norm_eq} and the uniform convergence of hessians (see \eqref{Pconv_hessian_uniform_eq}) gives us,
		\begin{align*}
		\sup_{q \in [0,1-\eta]}\| \nabla^2 V_2(\lambda_2(q),\phi(q);q)^{-1} -  \nabla^2 \hat V_2(\hat \lambda_2(q),\hat \phi(q);q)^{-1} \| & \explain{P}{\rightarrow} 0. 
		\end{align*}
		This concludes the proof of item (8).
	\end{enumerate}
\end{proof}
\section{Proof of Proposition \ref{delta_Delta_criteria}}
\label{delta_Delta_criteria_appendix}
Recall that we had introduced the following functions:
\begin{align*}
\CalF(q;\delta,\Delta,\sigma) & = {\Xi}_2(q;\sigma) - 2\Xi_1(\sigma)  + \left( 1 - \frac{1}{\delta} \right) \ln(1-q^2) + \Delta \ln \left( 1- \frac{q^2}{2} \right) \\
\CalFhat(q;\delta,\Delta,\sigma) & = \hat{\Xi}_2(q;\sigma) - 2\hat \Xi_1(\sigma)  + \left( 1 - \frac{1}{\delta} \right) \ln(1-q^2) + \Delta \ln \left( 1- \frac{q^2}{2} \right),
\end{align*}
where,
\begin{align*}
{\Xi}_2(q;\sigma) & \Mydef \max_{(\lambda,\phi) \in \mathbb{R}} \left( 2\lambda  + \phi q - \E_Y
\ln \ZTWis{\lambda}{\phi}{Y} \right), \\
\hat{\Xi}_2(q;\sigma) & \Mydef \max_{(\lambda,\phi) \in \mathbb{R}} \left( 2\lambda  + \phi q - \hatE_Y
\ln \ZTWis{\lambda}{\phi}{Y} \right), \\
\Xi_1(\sigma) &\Mydef  \max_{\lambda \in \mathbb{R}} \left( \lambda - \E_Y  
\ln \E_{E \sim \Exp{1}} e^{\lambda E} \gpdf(E-Y) \right),\\
\hat \Xi_1(\sigma) & \Mydef  \max_{\lambda \in \mathbb{R}} \left( \lambda - \E_Y  
\ln \hatE_{E \sim \Exp{1}} e^{\lambda E} \gpdf(E-Y) \right).
\end{align*}
Consider any $\delta$ that satisfies the assumptions of  Proposition \ref{delta_Delta_criteria}:
\begin{align}
\label{below_pt_consequence1}
\CalF(0;\delta,\Delta,\sigma) < \CalF(q;\delta,\Delta,\sigma) \; \forall \; q \; \in \; (0,1),
\end{align}
and,
\begin{align}
\label{below_pt_consequence2}
\frac{\diff^2 \CalF }{\diff q^2}(0;\delta,\Delta,\sigma) > 0.
\end{align}
In Lemmas \ref{second_moment_conditional} and \ref{introduce_main_integrals_lemma}, we showed that,
\begin{align}
\label{MI_decomposition}
\MI{\bm y, \bm z}{\bm A, \bm W}  & \leq \frac{2}{n-1} \E_{\bm y} \left[ \frac{ \int_0^1 \Nr \left(\bm y, \begin{bmatrix} 1  & q \\ q & 1 \end{bmatrix}\right) \cdot \frac{q(1-q^2)^{n-2}}{(1-q^2/2)^{\Delta m}}   \diff q}{\Dr^2(\bm y, 1)} \cdot  \mathbf{1}_{\mathcal{E}_m}\right] + C \cdot m \cdot \sqrt{\P(\mathcal{E}_m^c)}.
\end{align}
We will set $\mathcal{E}_m$ as: 
\begin{align}
\mathcal{E}_m & = \mathcal{E}^{(1)}_m(L) \cap \mathcal{E}^{(2)}_m(R,\eta) \cap \mathcal{E}^{(3)}_m(R,\eta) \cap \mathcal{E}^{(4)}_m(\eta)  \cap \mathcal{E}^{(5)}_m(\eta,\epsilon_2) \cap \mathcal{E}_m^{(6)}(R,\epsilon_2)
\end{align}
where:
\begin{align}
\label{important_events_eqs}
\mathcal{E}^{(1)}_m(L)  &= \left\{\bm y: 1 + \hatE Y^{40} \leq L \right\}, \\
\mathcal{E}^{(2)}_m(R,\eta)  &= \left\{\bm y: \sup_{|\lambda| \leq R} \left| \hatE \VTexp{\lambda}{Y} - \E \VTexp{\lambda}{Y}  \right| \leq \eta \right\}, \\ 
\mathcal{E}^{(3)}_m(R,\eta) & = \left\{  \bm y:  \sup_{|\lambda| + |\phi| \leq R} \left\| \hatE \VTWis{\lambda}{\phi}{Y} - \E \VTWis{\lambda}{\phi}{Y}  \right\| \leq \eta\right\}, \\
\mathcal{E}^{(4)}_m(\eta) & = \left\{  \bm y:  \sup_{q \leq 1/2} \left| \frac{\diff^2}{\diff q^2} \CalF(q;\delta,\Delta,\sigma)  - \frac{\diff^2}{\diff q^2} \CalFhat(q;\delta,\Delta,\sigma)   \right| \leq \eta\right\}, \\
\mathcal{E}^{(5)}_m(\eta,\epsilon_2) & = \left\{ \bm y: |\Xi_1(\sigma) - \hat \Xi_1(\sigma) | \leq \eta, \; \sup_{q \in [0,1-\epsilon_2]} |\Xi_2(q;\sigma) - \hat \Xi_2(q;\sigma)| \leq \eta \right\}, \\
\mathcal{E}_m^{(6)}(R,\epsilon_2) &  = \left\{\bm y: |\hat{\lambda}_1(\sigma)| \leq R, \; \sup_{q \in [0,1-\epsilon_2]} |\hat{\lambda}_2(q;\sigma)| + |\hat{\phi}(q;\sigma)|  \leq R  \right\}.
\end{align}
In the above display $L,R,\eta,\epsilon_2$ are parameters which will be set appropriately later. Recall that the notation $\hatE$ is used to denote empirical averages: 
\begin{align*}
\hatE f(Y) & = \frac{1}{m} \sum_{i=1}^m f(y_i),
\end{align*}
and the notation $\E f(Y) = \E_{Z,\epsilon} f(|Z|^2 + \sigma \epsilon)$ where $Z \sim \cgauss{0}{1}, \; \epsilon \sim \gauss{0}{1}$. 
Recall the upper bound in  \eqref{MI_decomposition}. Our goal in this section is to show $\MI{\bm y, \bm z}{\bm A, \bm W} = o(m)$. Towards this goal, the remainder of this section is organized as follows:
\begin{enumerate}
	\item In Lemma \ref{failure_event_bound} we show that $\P(\mathcal{E}_m^c) = o(1)$. 
	\item In Lemmas \ref{lb_dr_final_lemma} and \ref{final_ub_nr_lemma} we show that under the event $\mathcal{E}_m$, the assumptions of Corollary \ref{local_clt_denom_corollary} and \ref{local_clt_num_corollary} are met, and hence we can use them to obtain an upper bound on $\Nr$ and a lower bound on $\Dr$. 
	\item Finally the proof of Proposition \ref{delta_Delta_criteria} is restated and proved.
\end{enumerate}
\begin{lemma}[Analysis of $\P(\mathcal{E}_m)$]\label{failure_event_bound} For any $\epsilon_2 \in (0,1)$, there exists a critical value $R_c(\epsilon_2)$ such that, for any $L> 1 + \E Y^{40}$, any $R>R_c(\epsilon_2)$ and any $\eta > 0$, we have, for the event,
\begin{align*}
\mathcal{E}_m  = \mathcal{E}^{(1)}_m(L) \cap \mathcal{E}^{(2)}_m(R,\eta) \cap \mathcal{E}^{(3)}_m(R,\eta) & \cap \mathcal{E}^{(4)}_m(\eta)  \cap \mathcal{E}^{(5)}_m(\eta,\epsilon_2) \cap\mathcal{E}_m^{(6)}(R,\epsilon_2), \\ &\P(\mathcal{E}_m) \rightarrow 1.
\end{align*}
\end{lemma}
\begin{proof}
	This lemma is essentially a consequence of the concentration analysis in Proposition \ref{concentration_proposition}.
	By claim (1) of  Proposition \ref{wlln} we know that,
	\begin{align*}
	\hatE Y^{40} & \explain{P}{\rightarrow} \ \E Y^{40} < \infty.
	\end{align*}
	Consequently any $L > \E Y^{40}$ we have,
	\begin{align*}
	\P(\mathcal{E}^{(1)}_m(L)) & \rightarrow 1.
	\end{align*}
	For any $\epsilon_2 >0$. Claims (3) and (5) of Proposition \ref{concentration_proposition} guarantee the existence of $R_c(\epsilon_2)$ such that,
	\begin{align*}
	\P\mathcal(\mathcal{E}_m^{(6)}(R,\epsilon_2)) & \rightarrow 0, \; \forall \; R > R_c(\epsilon_2), \; 
	\forall \; \epsilon_2 > 0.
	\end{align*}
	Claim (2) of Proposition \ref{concentration_proposition} gives for any $R \in (0,\infty), \;  \eta > 0$,
	\begin{align*}
	\P(\mathcal{E}^{(2)}_m(R,\eta))  \rightarrow 1, \; \mathcal{E}^{(3)}_m(R,\eta)  \rightarrow 1. 
	\end{align*}
	Like wise Claim (4) and (6) \ref{concentration_proposition} guarantee for any $\epsilon_2 \in (0,1)$ and in $\eta > 0$, we have, $\P(\mathcal{E}^{(5)}_m(\eta,\epsilon_2)) \rightarrow 1$. Finally we observe that:
	\begin{align*}
	 \left| \frac{\diff^2}{\diff q^2} \CalF(q;\delta,\Delta,\sigma)  - \frac{\diff^2}{\diff q^2} \CalFhat(q;\delta,\Delta,\sigma)   \right| & =  \left| \frac{\diff^2}{\diff q^2} \Xi_2(q;\sigma)  - \frac{\diff^2}{\diff q^2} \Xi_2(q;\sigma)   \right|,
	\end{align*}
	Hence Claim (8) of Proposition \ref{concentration_proposition} shows that for any $\eta > 0$, we have,$\P(\mathcal{E}_n^{(4)}(\eta))  \rightarrow 1$. Finally a union bound gives us the claim $\P(\mathcal{E}_m) \rightarrow 1$.
\end{proof}
\begin{lemma}[A Lower Bound on $\Dr$] \label{lb_dr_final_lemma} For any $R,L \in (0,\infty)$, there exists a critical value of $\eta$ denoted by $\eta_1(R)$ depending only on $R$ such that for any $\eta < \eta_1(R), \; \epsilon_2>0$ on the event $\mathcal{E}_m^{(1)}(L) \cap \mathcal{E}^{(6)}_m(R,\epsilon_2) \cap \mathcal{E}_m^{(2)}(R,\eta) \cap \mathcal{E}_m^{(6)}(R,\epsilon_2)$, we have the lower bound,
	\begin{align}
	\label{lb_dr_eq}
	\Dr(\bm y, 1) & \geq \frac{1}{C(L,R)} e^{ - m  \cdot \hat \Xi_1}, \; \forall m \geq M(L,R).
	\end{align}	
	where $C(L,R), M(L,R)$ are large enough, finite constants depending only on $L,R$.
\end{lemma}
\begin{proof}
	Recall that from Corollary \ref{local_clt_denom_corollary}, we obtained the lower bound: 
	\begin{align*}
	\Dr(\bm y, 1) & \geq \frac{1}{2\sqrt{K}} \exp \left( -m \max_{\lambda \in \mathbb{R}} \left( \lambda - \hatE_Y  
	\ln \E_{E \sim \Exp{1}} e^{\lambda E} \gpdf(E-Y) \right) \right) = \frac{1}{2\sqrt{K}} e^{-m \hat \Xi_1}, \; \forall m \geq M(K).
	\end{align*}
	provided we can verify: 
	\begin{itemize}
		\item $\hatE (|Y|+|Y|^2 + |Y|^3) \leq K$: This can be ensured by taking $K \geq  3L$ and observing that under event $\mathcal{E}_m^{(1)}(L)$ we have $1 + \hatE Y^{40} \leq L$.
		\item $\hat{\lambda}$ which is the solution of the variational problem: 
		\begin{align*}
		\hat{\lambda} & = \arg \max_{\lambda \in \mathbb{R}} \left( \lambda - \hatE_Y  
		\ln \E_{E \sim \Exp{1}} e^{\lambda E} \gpdf(E-Y) \right), 
		\end{align*}
		lies in a compact set $|\hat{\lambda}_1(\sigma)| \leq K$. Taking $K \geq R$  guarantees this under the event $\mathcal{E}_m^{(6)}(R,\epsilon_2)$.
		\item Finally we need to check: 
		\begin{align}
		\label{verify_eq}
		\frac{1}{K} \leq \hatE \VTexp{\hat\lambda}{Y} \leq K,
		\end{align}
		for some value of $K$. Note that event $\mathcal{E}_m^{(6)}(R,\epsilon_2)$, guarantees $|\hat{\lambda}_1(\sigma)| \leq R$.
		The function $\lambda \mapsto  \E \VTexp{\lambda}{Y}$ is strictly positive and finite on compact sets, that is:
		\begin{align*}
		0 < \min_{|\lambda| \leq R} \E \VTexp{\lambda}{Y} \leq \max_{|\lambda| \leq R} \E \VTexp{\lambda}{Y} < \infty.
		\end{align*}
		This can be checked by observing $\lambda \mapsto  \E \VTexp{\lambda}{Y}$ is continuous and if $ \E \VTexp{\lambda}{Y} = 0$ for some $\lambda$ then, $\VTexp{\lambda}{Y} \explain{a.s.}{=}  0$. This is clearly not possible since $\Texp{\lambda}{y}$ is not deterministic for any finite $\lambda,y$. Hence there exists a constant depending only on $R$ such that,
		\begin{align*}
		\frac{1}{C_1(R)}\leq  \E \VTexp{\hat \lambda}{Y} \leq C_1(R).
		\end{align*}
		The event $\mathcal{E}_m^{(2)}(R,\eta)$ guarantees:
		\begin{align*}
		\sup_{|\lambda| \leq R} \left| \hatE \VTexp{\lambda}{Y} - \E \VTexp{\lambda}{Y}  \right| \leq \eta.
		\end{align*} 
		Since $|\hat{\lambda}_1(\sigma)| \leq R$, the above error bound holds for $\lambda = \hat \lambda_1(\sigma)$. Taking $\eta \leq (2 C_1(R))^{-1}$ guarantees:
		\begin{align*}
		\frac{1}{2C_1(R)}  \leq \hatE \VTexp{\hat\lambda_1(\sigma)}{Y} \leq C_2(R) + \frac{1}{C_2(R)}.
		\end{align*} This verifies \eqref{verify_eq} for a suitable $K$.
	\end{itemize}
	Hence, all the requirements of Proposition \ref{denominator_variational_problem_prop} are satisfied which gives us the claim of the lemma.
\end{proof}
\begin{lemma}[An Upper Bound on $\Nr$]  \label{final_ub_nr_lemma} We have the following upper bounds on $\Nr$: 
	\begin{enumerate}
		\item Unconditional Upper Bound: For any $\bm y \in \mathbb R^m$, for any $q \in [0,1)$, we have, 
		\begin{align*}
		\Nr\left(\bm y, \begin{bmatrix} 1 & q \\ q & 1 \end{bmatrix}\right) & \leq e^{C_\Nr \cdot m},
		\end{align*}
		for a universal constant $C_\Nr$ which depends only on the noise level $\sigma$.
		\item For any $R,L \in (0,\infty)$, there exists a critical value of $\eta$ denoted by $\eta_2(R)$ depending only on $R$ such that for any $\eta < \eta_2(R), \; \epsilon_2>0$, we have the upper bound,
		\begin{align*}
		\Nr\left(\bm y, \begin{bmatrix} 1 & q \\ q & 1 \end{bmatrix}\right) & \leq\frac{C(L,R) \cdot e^{-m \hat \Xi_2(q)}}{m^2 \cdot (1-q^2)^{m-2}} \; \forall \; q \in [0,1-\epsilon_2], \; \forall \; \bm y \in \mathcal{E}_m^{(1)}(L) \cap \mathcal{E}_m^{(3)}(R,\eta)\cap \mathcal{E}_m^{(6)}(R,\epsilon_2).
		\end{align*} 
		In the above display, $C(L,R)$ is a constant depending only on the choice of $L,R$.
	\end{enumerate}
\end{lemma}
\begin{proof}
	\begin{enumerate}
		\item We recall the definition of $\Nr$:
		\begin{align*}
		\Nr(\bm y, \bm Q) & \Mydef \E \left[ \prod_{i=1}^m \gpdf(y_i - m | G_{1i}|^2 ) \gpdf(y_i - m |G_{2i}|^2 ) \bigg| \bm G ^\UH \bm G = {\bm Q} \right].
		\end{align*}
		Observing that, $\gpdf(x)  \leq (2\pi \sigma^2)^{-1/2}$,	
		we obtain, $\forall \bm y \in \mathbb R^m, \; \forall q \in (0,1)$,
		\begin{align*}
		\Nr\left(\bm y, \begin{bmatrix} 1 & q \\ q & 1\end{bmatrix}\right) & \leq e^{C_\Nr \cdot m},
		\end{align*}
		for a universal constant $C_\Nr < \infty$ that depends only on $\sigma$. 
		\item Recall that  Corollary \ref{local_clt_num_corollary} shows,
		\begin{align*}
		\Nr\left(\bm y, \begin{bmatrix} 1 & q \\ q & 1 \end{bmatrix}\right) & \leq  \frac{C(K)e^{-m \hat \Xi_2(q)}}{m^2 \cdot (1-q^2)^{m-2}}.
		\end{align*}
		provided we can show: 
		\begin{itemize}
			\item $\hatE Y^{40} \leq K$. This is true under the event $\mathcal{E}_m^{(1)}(L)$ if we choose $K \geq L$. 
			\item The minimizing arguments $(\hat{\lambda}_2(q;\sigma), \hat{\phi}(q;\sigma) )$ satisfy $|\hat{\lambda}_2(q;\sigma)| +  |\hat{\phi}(q;\sigma)| \leq K$ for any $q \in [0,1-\epsilon_2]$. This is guaranteed by the event  $\mathcal{E}_m^{(6)}(R,\epsilon_2)$ if $K \geq R$.
			\item  Finally, we need to check: 
			\begin{align}
			\frac{1}{K}  \leq \lambda_{\min}\left(\hatE\VTWis{\hat\lambda}{\hat \phi}{Y}\right) \leq  \lambda_{\max}\left(\hatE\VTWis{\hat\lambda}{\hat \phi}{Y}\right) \leq K.
			\label{verify_eq_2}
			\end{align}
			The event $\mathcal{E}_m^{(6)}(R,\epsilon_2)$ guarantees $|\hat \lambda_2(q;\sigma)| + |\hat \phi(q;\sigma)| \leq R, \; \forall \; q \; \in [0,1-\epsilon_2]$. The matrix  function $(\lambda,\phi) \mapsto \E\VTWis{\lambda}{ \phi}{Y}$ is:
			\begin{enumerate}
				\item Bounded on the compact set $|\lambda| + |\phi| \leq R$. Indeed:
				\begin{align*}
				\| \E\VTWis{\lambda}{ \phi}{Y} \|  \leq \E\|\VTWis{\lambda}{ \phi}{Y}\| \explain{(a)}{\leq} C(1 + |\lambda|^2 + |\phi|^2 + \E Y^{2}) \leq C(1+R^2).
				\end{align*}
				In the inequality marked (a), we used the moment bounds for the tilted Wishart distribution derived in Claim (4) of Lemma \ref{tilted_wishart_properties}. 
				\item Strictly positive definite on the compact set $|\lambda| + |\phi| \leq R$. To see this we note that if $\lambda_{\min}(\E\VTWis{\lambda}{ \phi}{Y}) = 0$ for some $\lambda,\phi$ then since $$\lambda_{\min}(\E\VTWis{\lambda}{ \phi}{Y}) \geq \E \lambda_{\min}(\VTWis{\lambda}{ \phi}{Y}),$$ we have $\lambda_{\min}(\VTWis{\lambda}{ \phi}{Y}) = 0$ almost surely (with respect to the distribution of $Y$). This contradicts Claim (6) of Lemma \ref{tilted_wishart_properties}.
			\end{enumerate}
			Hence, there exists a positive and finite constant $C_2(R)$ depending only on $R$ such that,
			\begin{align*}
			\frac{1}{C_2(R)} & \leq \lambda_{\min}\left(\E\VTWis{\hat \lambda}{ \hat\phi}{Y}\right) \leq  \lambda_{\max}\left(\E\VTWis{\hat \lambda}{\hat \phi}{Y}\right) \leq C_2(R).
			\end{align*}
			The event $\mathcal{E}_m^{(3)}(R,\eta)$ guarantees: 
			\begin{align*}
			\left| \lambda_{\min}\left(\E\VTWis{\hat\lambda}{ \hat\phi}{Y}\right) - \lambda_{\min}\left(\hatE\VTWis{\hat\lambda}{ \phi}{Y}\right) \right| &\leq \eta, \\ \left| \lambda_{\max}\left(\E\VTWis{\lambda}{ \phi}{Y}\right) - \lambda_{\max}\left(\hatE\VTWis{\lambda}{ \phi}{Y}\right) \right| &\leq \eta.
			\end{align*}
			Choosing $\eta \leq  (2C_2(R))^{-1}$, we have,
			\begin{align*}
			\frac{1}{2 C_2(R)}  \leq \lambda_{\min}\left(\hatE\VTWis{\hat\lambda}{\hat \phi}{Y}\right) \leq  \lambda_{\max}\left(\hatE\VTWis{\hat\lambda}{\hat \phi}{Y}\right) \leq C_2(R) + \frac{1}{2C_2(R)},
			\end{align*} which verifies \eqref{verify_eq_2} for a suitable $K$. 
		\end{itemize}
		Hence all the assumptions of Corollary \ref{local_clt_num_corollary} have been verified, which gives us the claim in item (2) of the lemma.
	\end{enumerate}
\end{proof}

Finally we restate and prove Proposition \ref{delta_Delta_criteria}.
\deltaDeltaCriteria*

\begin{proof}
	In Lemmas \ref{second_moment_conditional} and \ref{introduce_main_integrals_lemma}, we showed that,
	\begin{align}
	&\MI{\bm y, \bm z}{\bm A, \bm W}  \leq \frac{1}{n-1} \E_{\bm y} \left[ \frac{ \int_0^1 \Nr \left(\bm y, \begin{bmatrix} 1  &\sqrt{b} \nonumber \\ \sqrt{b} & 1 \end{bmatrix}\right) \cdot \frac{(1-b)^{n-2}}{(1-b/2)^{\Delta m}} \diff b}{\Dr^2(\bm y, 1)} \cdot  \mathbf{1}_{\mathcal{E}_m}\right] + C \cdot m \cdot \sqrt{\P(\mathcal{E}_m^c)} \nonumber\\
	& = \frac{2}{n-1} \E_{\bm y} \left[ \frac{ \int_0^1 \Nr \left(\bm y, \begin{bmatrix} 1  & q \\ q & 1 \end{bmatrix}\right) \cdot \frac{q(1-q^2)^{n-2}}{(1-q^2/2)^{\Delta m}}   \diff q}{\Dr^2(\bm y, 1)} \cdot  \mathbf{1}_{\mathcal{E}_m}\right] + C \cdot m \cdot \sqrt{\P(\mathcal{E}_m^c)} \nonumber\\
	& \explain{(a)}{=} (1) + (2)+ (3) +  C \cdot m \cdot \sqrt{\P(\mathcal{E}_m^c)}. \label{upper_bound_MI_eq}
	\end{align}
	In the step marked (a), we split the integral into three parts: 
	\begin{align*}
	(1) & \Mydef \frac{2}{n-1} \E_{\bm y} \left[ \frac{ \int_{1-\epsilon_2}^1 \Nr \left(\bm y, \begin{bmatrix} 1  & q \\ q & 1 \end{bmatrix}\right) \cdot \frac{q(1-q^2)^{n-2}}{(1-q^2/2)^{\Delta m}}   \diff q}{\Dr^2(\bm y, 1)} \cdot  \mathbf{1}_{\mathcal{E}_m}\right], \\
	(2) & \Mydef \frac{2}{n-1} \E_{\bm y} \left[ \frac{ \int_{\epsilon_1}^{1-\epsilon_2} \Nr \left(\bm y, \begin{bmatrix} 1  & q \\ q & 1 \end{bmatrix}\right) \cdot \frac{q(1-q^2)^{n-2}}{(1-q^2/2)^{\Delta m}}   \diff q}{\Dr^2(\bm y, 1)} \cdot  \mathbf{1}_{\mathcal{E}_m}\right], \\
	(3) & \Mydef \frac{2}{n-1} \E_{\bm y} \left[ \frac{ \int_{0}^{\epsilon_1} \Nr \left(\bm y, \begin{bmatrix} 1  & q \\ q & 1 \end{bmatrix}\right) \cdot \frac{q(1-q^2)^{n-2}}{(1-q^2/2)^{\Delta m}}   \diff q}{\Dr^2(\bm y, 1)} \cdot  \mathbf{1}_{\mathcal{E}_m}\right]. \\
	\end{align*}In the above display $\epsilon_1,\epsilon_2 \in (0,1)$ are parameters which will be set appropriately. We also recall that we had set: \begin{align*}
	\mathcal{E}_m & = \mathcal{E}^{(1)}_m(L) \cap \mathcal{E}^{(2)}_m(R,\eta) \cap \mathcal{E}^{(3)}_m(R,\eta) \cap \mathcal{E}^{(4)}_m(\eta)  \cap \mathcal{E}^{(5)}_m(\eta,\epsilon_2) \cap \mathcal{E}_m^{(6)}(R,\epsilon_2). 
	\end{align*}
	where the various events have been defined in Equations \ref{important_events_eqs}. We now describe how to set the parameters $L,R,\eta, \epsilon_1,\epsilon_2$ so that each of the terms in \eqref{upper_bound_MI_eq} is $o(m)$. We also draw the readers attention to the point that the parameter $\epsilon_2$ used to define the cutoff points for the integrals (1) and (2) is same as the $\epsilon_2$ in the definition of the event $\mathcal{E}^{(5)}_m(\eta,\epsilon_2), \mathcal{E}_m^{(6)}(R,\epsilon_2)$. Notice also the same parameter $R$ is involved in the definitions of the events $\mathcal{E}^{(2)}_m(R,\eta), \mathcal{E}^{(3)}_m(R,\eta),\mathcal{E}_m^{(6)}(R,\epsilon_2)$.
	\begin{description}
		\item [Analysis of $(1)$: ] By Lemma \ref{final_ub_nr_lemma}, we know that,
		\begin{align*}
		\Nr \left(\bm y, \begin{bmatrix} 1  & q \\ q & 1 \end{bmatrix}\right) & \leq e^{C_\Nr \cdot m}.
		\end{align*}
		We next appeal to Lemma \ref{lb_dr_final_lemma}. We enforce the requirement
		\begin{align}
		\eta &< \eta_1(R) \label{enforce_eq1}
		\end{align} and obtain,
		\begin{align*}
		\Dr(\bm y, 1) & \geq \frac{1}{C(L,R)} \cdot e^{-m \hat \Xi_1}.
		\end{align*}
		The event $\mathcal{E}_m^{(5)}(\eta,\epsilon_2)$ guarantees that $\hat \Xi_1(\sigma) \leq \Xi_1(\sigma) + \eta$. By enforcing:
		\begin{align}
		\eta &\leq 1, \label{enforced_eq2}
		\end{align} 
		we have $\hat \Xi_1(\sigma) \leq \Xi_1(\sigma)+1$ which is an absolute constant (depending only on the noise level).  Consequently, we have $\Dr(\bm y, 1) \geq C(L,R)^{-1} \cdot e^{- C_\Dr \cdot m}$ for some universal constant $C_\Dr \in (0,\infty)$ depending only on the noise level. Hence we have, for $\min(n,m) \geq 4$
		\begin{align*}
		(1) & \leq \frac{2 \cdot C^2(L,R) \cdot e^{(C_\Nr + 2C_\Dr) \cdot m}}{n-1} \cdot \int_{1-\epsilon_2}^1 (1-q^2)^{n-2} \diff q \\&\leq \frac{2 \cdot C^2(L,R)}{n-1} \cdot e^{(C_\Nr+ 2 C_\Dr) \cdot m} \cdot (1-(1-\epsilon_2)^2)^{\frac{n}{2}} \\ & \leq  \frac{2 \cdot C^2(L,R)}{n-1} \cdot e^{(C_\Nr+ 2 C_\Dr) \cdot m} \cdot (2 \epsilon_2)^{\frac{n}{2}}  \\& =  \frac{2 \cdot C^2(L,R)}{n-1} \cdot \exp \left(  m \cdot  \left( C_\Nr + 2 C_\Dr + \frac{\ln(2)}{2\delta} - \frac{\ln \frac{1}{\epsilon_2}}{2\delta} \right)\right)
		\end{align*}
		We set: 
		\begin{align}
		\epsilon_2  &= \frac{1}{2} \cdot e^{-2\delta(C_\Nr + 2C_\Dr)}< 1, \label{epsilon2_setting}
		\end{align}
		which gives us $(1) = O(1/n) = o(1)$.
		\item[Analysis of $\P(\mathcal{E}_m^c)$: ] As  suggested by Lemma  \ref{failure_event_bound}, we set $L > \E |Y|^{40}$. For example, we can set $L = 1 + \E Y^{40}$. We will also enforce the constraint $R > R_c(\epsilon_2)$ for example by setting $R = R_c(\epsilon_2) + 1$ (note that $\epsilon_2$ has been set in  \eqref{epsilon2_setting}).  This ensures that $\P(\mathcal{E}_m^c) = o(1)$. At this set we have set $R, \epsilon_2, L$ and we are still free to set $\eta>0,\epsilon_1 \in (0,1)$ arbitrarily subject to the requirements in  \eqref{enforce_eq1}-\eqref{enforced_eq2}. 
		\item [Analysis of (2):] We enforce:
		\begin{align}
		\eta & < \min( \eta_1(R), \eta_2(R)) \label{enforce_eq3}
		\end{align} which is enough to satisfy the assumptions of Lemma \ref{lb_dr_final_lemma} and item (2) of Lemma \ref{final_ub_nr_lemma}, which  gives us,
		\begin{align}
		\label{q_not_close_to_1_ub}
		\frac{\Nr \left(\bm y, \begin{bmatrix} 1  & q \\ q & 1 \end{bmatrix}\right)}{\Dr^2(\bm y,1)} & \leq \frac{C \cdot e^{-m (\hat \Xi_2(q)-2\hat \Xi_1)}}{m^2 \cdot (1-q^2)^{m-2}} \; \forall \; q \in [0,1-\epsilon_2].
		\end{align}
		This allows us to upper bound the term (2) as follows: 
		\begin{align*}
		(2) & \Mydef \frac{2}{n-1} \E_{\bm y} \left[ \frac{ \int_{\epsilon_1}^{1-\epsilon_2} \Nr \left(\bm y, \begin{bmatrix} 1  & q \\ q & 1 \end{bmatrix}\right) \cdot \frac{q(1-q^2)^{n-2}}{(1-q^2/2)^{\Delta m}}   \diff q}{\Dr^2(\bm y, 1)} \cdot  \mathbf{1}_{\mathcal{E}_m}\right] \\
		& \leq \frac{C}{(n-1) \cdot m^2} \cdot  \E_{\bm y} \left[ \int_{\epsilon_1}^{1-\epsilon_2} e^{-m\CalFhat(q;\delta,\Delta,\sigma)} \diff q\cdot \mathbf{1}_{\mathcal{E}_m}\right]
		\end{align*}
		Since event $\mathcal{E}_m^{(5)}(\eta,\epsilon_2)$ guarantees $|\hat \Xi_1(\sigma) - \Xi_1(\sigma)| \leq \eta, \; \sup_{q \in [0,1-\epsilon_2]} |\hat \Xi_2(q;\sigma) - \Xi_2(q;\sigma)| \leq \eta$, we have, $$|\CalFhat(q;\delta,\Delta,\sigma) - \CalF(q;\delta,\Delta,\sigma)| \leq 3 \eta \; \forall \; q \in [0,1-\epsilon_2].$$ Since $\delta < \critthr(\sigma^2,\Delta)$ and $\epsilon_1 > 0$, Recall that we have, $\inf_{q \in [\epsilon_1,1]} \CalF(q;\delta,\Delta,\sigma) > 0 = \CalF(0;\delta,\Delta,\sigma)$ (see  \eqref{below_pt_consequence1}). Hence, we can enforce that $\eta,\epsilon_1$  satisfy:
		\begin{align}
		\eta <  \frac{1}{6}\inf_{q \in [\epsilon_1,1]} \CalF(q;\delta,\Delta,\sigma) \label{enforce_eq4}
		\end{align}
		This guarantees, for any $q \in [\epsilon_1, 1-\epsilon_2]$, 
		\begin{align*}
		\CalFhat(q;\delta,\Delta,\sigma) \geq \CalF(q;\delta,\Delta,\sigma) - 3 \eta \geq \inf_{q \in [\epsilon_1,1]} \CalF(q;\delta,\Delta,\sigma) - 3 \eta \geq \frac{1}{2} \inf_{q \in [\epsilon_1,1]} \CalF(q;\delta,\Delta,\sigma) > 0.
		\end{align*}
		Hence,
		\begin{align*}
		(2)  \leq \frac{C}{(n-1) \cdot m^2} \cdot \exp \left( - \frac{m}{2} \cdot \inf_{q \in [\epsilon_1,1]} \CalF(q;\delta,\Delta,\sigma) \right) = o(1).
		\end{align*}
		\item[Analysis of (3): ]
		We recall that Term (3) was given by: 
		\begin{align*}
		(3) & = \frac{2}{n-1} \E_{\bm y} \left[ \frac{ \int_0^{\epsilon_1} \Nr \left(\bm y, \begin{bmatrix} 1  & q \\ q & 1 \end{bmatrix}\right) \cdot \frac{q(1-q^2)^{n-2}}{(1-q^2/2)^{\Delta m}}   \diff q}{\Dr^2(\bm y, 1)} \cdot  \mathbf{1}_{\mathcal{E}_m}\right].
		\end{align*}
		The upper bound in \eqref{q_not_close_to_1_ub} applies to $q \in [0,\epsilon_1]$. Hence we obtain, 
		\begin{align*}
		(3) & \leq \frac{C}{(n-1) \cdot m^2} \cdot  \E_{\bm y} \left[ \int_{0}^{\epsilon_1} e^{-m\CalFhat(q;\delta,\Delta,\sigma)} \diff q\cdot \mathbf{1}_{\mathcal{E}_m}\right]
		\end{align*}
		Next we approximate $\CalFhat$ by its taylors expansion at $q = 0$. First observe that, $\CalFhat(0;\delta,\Delta,\sigma) = 0$ and,
		\begin{align*}
		\frac{\diff \CalFhat}{\diff q}(q; \delta,\Delta,\sigma) & = \hat{\phi}(q;\sigma) - 2\left( 1- \frac{1}{\delta} \right) \frac{q}{1-q^2} - \frac{\Delta q}{1- q^2/2} \implies  \frac{\diff \CalFhat}{\diff q}(0;\delta,\Delta,\sigma) = 0.
		\end{align*}
		We enforce the constraint \begin{align}\eta &< \frac{1}{4}\frac{\diff^2 \CalF}{\diff q^2}(0;\delta,\Delta,\sigma). \label{enforce_eq5} \end{align} We set $\epsilon_1 \in (0,1/2)$ which guarantees:
		\begin{align*}
		\left| \frac{\diff^2 \CalF}{\diff q^2}(q;\delta,\Delta,\sigma) - \frac{\diff^2 \CalF}{\diff q^2}(0;\delta,\Delta,\sigma) \right| & \leq \frac{1}{2} \frac{\diff^2 \CalF}{\diff q^2}(0;\delta,\Delta,\sigma).
		\end{align*}
		\eqref{below_pt_consequence2} and the fact that $\CalF(\cdot; \delta,\Delta,\sigma)$ has a continuous second derivative at $q=0$ ensures this is possible. Hence we have,
		\begin{align*}
		\frac{\diff^2 \CalF}{\diff q^2}(q;\delta,\Delta,\sigma) > 2 \eta, \; \forall \; q < \epsilon_1.
		\end{align*}
		The event $\mathcal{E}_m^{(4)}(\eta)$ guarantees: 
		\begin{align*}
		\sup_{q \in [0,1/2]}\left| \frac{\diff^2 \CalFhat}{\diff q^2}(q;\delta,\Delta,\sigma) - \frac{\diff^2 \CalF}{\diff q^2}(q;\delta,\Delta,\sigma) \right| \leq \eta \implies \frac{\diff^2 \CalFhat}{\diff q^2}(q;\delta,\Delta,\sigma) > \eta, \; \forall \; q < \epsilon_1.
		\end{align*}
		Then by Taylor's theorem, we have, $\forall \; q \in [0,\epsilon_1)$,
		\begin{align*}
		\CalFhat(q;\delta,\Delta,\sigma) & \geq \CalFhat(0;\delta,\Delta,\sigma) + \frac{\diff \CalFhat}{\diff q}(0;\delta,\Delta,\sigma) \cdot q + \left( \inf_{x \in [0,\epsilon_1)} \frac{\diff^2 \CalFhat}{\diff q^2}(x;\delta,\Delta,\sigma)  \right) \cdot \frac{q^2}{2} \geq \frac{\eta q^2}{2}.
		\end{align*}
		Hence we obtain,
		\begin{align*}
		(3) & \leq \frac{C}{(n-1) \cdot m^2} \cdot \int_0^{\epsilon_1} e^{-\frac{\eta q^2}{2} \cdot m}  \leq \frac{C}{(n-1)\cdot m^2} =  o(1).
		\end{align*}
		Finally we note that set $\eta$ as,
		\begin{align*}
		\eta = \min \left( 1, \eta_1(R), \eta_2(R),\frac{1}{6}\inf_{q \in [\epsilon_1,1]} \CalF(q;\delta,\Delta,\sigma), \frac{1}{4}\frac{\diff^2 \CalF}{\diff q^2}(0;\delta,\Delta,\sigma) \right)
		\end{align*}
		satisfies requirements in  \eqref{enforce_eq1},\eqref{enforced_eq2},\eqref{enforce_eq3}, \eqref{enforce_eq4} and \eqref{enforce_eq5} and also ensures $\eta$ is a fixed positive constant.
	\end{description}
This concludes the proof of the proposition.
\end{proof}
\section{Proofs from Section \ref{low_noise_section}}
\label{low_noise_section_proofs}
This section is devoted to proving Proposition \ref{lownoiseprop}. Recall that the function $\CalF(q;\delta,\Delta, \sigma)$ was defined as:
\begin{align*}
\CalF(q; \delta, \Delta,\sigma) & \Mydef {\Xi}_2(q;\sigma) - 2\Xi_1(\sigma)  + \left( 1 - \frac{1}{\delta} \right) \ln(1-q^2) + \Delta \ln \left( 1- \frac{q^2}{2} \right),
\end{align*}
where the functions, $\Xi_1,\Xi_2$ are defined as follows:
\begin{align*}
{\Xi}_2(q;\sigma) & \Mydef \max_{(\lambda,\phi) \in \mathbb{R}} \left( 2\lambda  + \phi q - \E_Y
\ln \ZTWis{\lambda}{\phi}{Y} \right), \\
\Xi_1(\sigma) &\Mydef  \max_{\lambda \in \mathbb{R}} \left( \lambda - \E_Y  
\ln \ZTexp{\lambda}{Y} \right).
\end{align*}
In the above display the random variable $Y = |G|^2 + \sigma \epsilon$, where $G \sim \cgauss{0}{1}$ and $\epsilon \sim \gauss{0}{1}$. Our goal is to identify conditions on $(\delta,\Delta,\sigma)$ such that,
\begin{align}
\CalF(q;\delta,\Delta,\sigma) > \CalF(0;\delta,\Delta,\sigma) \; \forall \; q \in (0,1), \; \frac{\diff^2}{\diff q^2} \CalF(q;\delta,\Delta,\sigma) > 0. \label{goal_low_noise_section_proof}
\end{align}
We will not be able to solve this for a general $\sigma > 0 $, but only for small enough $\sigma$ since in the limit $\sigma \rightarrow 0$, the variational problems in the definition of $\Xi_2, \Xi_1$ simplify considerably. 

We first begin with a heuristic derivation of the zero noise limit of the functions $\Xi_2(q;\sigma)$ and $\Xi_1(\sigma)$. Recalling the definition of $\ZTexp{\lambda}{y}$ (Definition \ref{Texp_definition}):
\begin{align*}
\ln \ZTexp{\lambda}{Y} &= \E_{E \sim \Exp{1}} e^{\lambda E} \gpdf(E-|G|^2 - \sigma \epsilon)  \\ &= e^{(\lambda - 1)(|G|^2 + \sigma \epsilon)} \E_{\omega \sim \gauss{0}{1}} e^{\sigma(\lambda-1) \omega}  \mathbf{1}_{|G|^2 + \sigma \epsilon + \sigma\omega \geq 0} \\ &\explain{$\sigma \rightarrow 0$}{\rightarrow} e^{(\lambda - 1)|G|^2}.
\end{align*}
This gives us,
\begin{align*}
 \lambda - \E_Y  \ln \ZTexp{\lambda}{Y} & \explain{$\sigma \rightarrow 0$}{\rightarrow} 1.
\end{align*}
In the zero noise limit, the variational problem in the definition of $\Xi_1$ is trivial. Hence, it makes sense to extend the definition of $\Xi_1(\sigma)$ to include $\sigma = 0$ as $\Xi_1(0) \Mydef 1$. 
Likewise, recalling Definition \ref{titled_wishart_definition}, we have,
\begin{align*}
 \ZTWis{\lambda}{\phi}{Y} & = \ZTWis{\lambda}{\phi}{|G|^2 + \sigma \epsilon} \\ & =   
 \E \exp \left((\lambda-1)(2Y + \sigma(\omega_1 + \omega_2)) + \phi \sqrt{(Y+\sigma \omega_1)(Y+\sigma \omega_2)} \cos(\theta) \right) \mathbf{1}_{Y + \sigma \omega_1 \geq 0, y + \sigma \omega_2 \geq 0} \\
 & \explain{$\sigma  \rightarrow 0$}{\rightarrow} e^{2(\lambda-1)|G|^2} \E_\theta e^{\phi |G|^2 \cos(\theta)}\\
 & = e^{2(\lambda-1)|G|^2} I_0(|G|^2 \phi).
\end{align*}
In the last step we used the definition of Modified Bessel function $I_0(x) \Mydef \E e^{x \cos \theta}$. Hence we extend the definition of $\Xi_2(q;\sigma)$ to $\sigma = 0$ as:
\begin{align*}
\Xi_2(q;0) & \Mydef 2 + \max_{\phi \in \mathbb R} q \phi -  \E_{Z \sim \cgauss{0}{1}} \ln I_0(\phi |Z|^2).
\end{align*}
This allows to guess the correct zero noise limit of $\CalF(q;\delta,\Delta,\sigma)$ as:
\begin{align*}
\CalF(q;\delta,\Delta,0) & \Mydef {\Xi}_2(q;0) - 2\Xi_1(0)  + \left( 1 - \frac{1}{\delta} \right) \ln(1-q^2) + \Delta \ln \left( 1- \frac{q^2}{2} \right).
\end{align*}
The remainder of this section is organized as follows:
\begin{enumerate}
	\item In Section \ref{zero_noise_analysis} we analyze the zero noise limit function $\CalF(q;\delta,\Delta,0)$ and find a condition on $(\delta,\Delta)$ such that  \eqref{goal_low_noise_section_proof} holds for $\CalF(q;\delta,\Delta,0)$. 
	\item In Section \ref{zero_noise_convergence}, we show that $\Xi_1(\sigma)$ converges to $\Xi_1(0)$ and $\Xi_2(q;\sigma)$ converges to $\Xi_2(q;0)$ in an appropriate sense. 
	\item Finally Section \ref{low_noise_prop_proof} contains the proof of Proposition \ref{lownoiseprop}.
\end{enumerate}
Throughout this section, unlike in other parts of this paper, $C$ denotes a universal constant that does \emph{not} depend on $\sigma$. As before this constant may change from line to line.
\subsection{Analysis in the Low Noise Limit} \label{zero_noise_analysis}
The following lemma shows that if $\delta < 2$, and $\Delta$ is small enough (but positive), the function $\CalF(q;\delta,\Delta,0)$ is strictly increasing. 
\begin{lemma}[Limiting Variational Problems] \label{zero_noise_variational_problems} Consider the following functions for $q \in [0,1)$: 
	\begin{align*}
	\Xi_2(q;0) & \Mydef 2 + \max_{\phi \in \mathbb R} q \phi -  \E_{Z \sim \cgauss{0}{1}} \ln I_0(\phi |Z|^2), \\
	\phi_2(q;0) & \Mydef \arg \max_{\phi \in \mathbb R} q \phi - \E_{Z \sim \cgauss{0}{1}} \ln I_0(\phi |Z|^2).
	\end{align*}
	Then we have, 
	\begin{enumerate}
		\item The function $\phi \mapsto q \phi -  \E_{Z \sim \cgauss{0}{1}} \ln I_0(\phi |Z|^2)$ has a unique maximizer $\phi_2(q;0)$ which satisfies: $0 \leq \phi_2(q;0)<\infty$ for any $q \in [0, 1)$.   Furthermore, $\max_{q \in [0,1-\eta]} \phi_2(q;0) < \infty$ for any $\eta \in (0,1)$. 
		\item The function: 
		\begin{align*}
		f(q) \Mydef \Xi_2(q;0)  + \left( 1 - \frac{1}{\delta} \right) \ln(1-q^2) + \Delta \ln \left( 1 - \frac{q^2}{2} \right),
		\end{align*}
		is a strictly increasing function of $q$ with $f(0)< f(q) \; \forall q \in [0,1)$,  provided,
		\begin{align*}
		0<\delta < 2, \; 0<  \Delta < \frac{2-\delta}{\delta}.
		\end{align*} 
	\end{enumerate}
\end{lemma}
\begin{proof}
	\begin{enumerate}
		\item The function $\phi \mapsto q \phi -  \E_{Z \sim \cgauss{0}{1}} \ln I_0(\phi |Z|^2)$ is strictly concave (see Fact \ref{bessel_properties}, item (5), Appendix \ref{misc_appendix}). Hence, $q \phi -  \E_{Z \sim \cgauss{0}{1}} \ln I_0(\phi |Z|^2)$ has at most one maximizer. Next observe that any maximizer must lie in $[0,\infty]$. This is because $\E \ln I_0(|\phi| |Z|^2) = \E \ln I_0(-|\phi| |Z|^2)$  since $I_0$ is even (see Fact \ref{bessel_properties}, Appendix \ref{misc_appendix}), but $q |\phi| \geq - q |\phi|$. This shows that if $\phi_2(q;0)$ exists, we must have, $\phi_2(q; 0) \geq 0$. In order to show existence of $\phi_2(0,q)$ it is sufficient to find a solution to the first order optimality conditions:
		\begin{align*}
		\E |Z|^2 \cdot \frac{I_0^\prime(\phi |Z|^2)}{I_0(\phi |Z|^2)}   = q.
		\end{align*}
		Note that:
		\begin{align*}
		\E |Z|^2 \cdot \frac{I_0^\prime(\phi |Z|^2)}{I_0(\phi |Z|^2)} \bigg | _{\phi = 0} = 0.
		\end{align*}
		In order to check that $\phi_2(q;0) < \infty$, it is sufficient to show that, 
		\begin{align*}
		\lim_{\phi \rightarrow \infty} q - \E |Z|^2 \cdot  \frac{I_0^\prime(\phi |Z|^2)}{I_0(\phi |Z|^2)} < 0.
		\end{align*}
		By Monotone convergence theorem and the fact that $\frac{I_0^\prime(x)}{I_0(x)} \uparrow 1$ as $x \uparrow \infty$ (see Fact \ref{bessel_properties}, Appendix \ref{misc_appendix}), we have, 
		\begin{align*}
		\lim_{\phi \rightarrow \infty} q - \E |Z|^2 \cdot  \frac{I_0^\prime(\phi |Z|^2)}{I_0(\phi |Z|^2)} = q - 1< 0 , \; \forall \; q < 1.
		\end{align*}
		This confirms $\phi_2(q;0) < \infty$ for any $q<1$. Further inspection of the stationarity condition: 
		\begin{align*}
		\E |Z|^2 \cdot \frac{I_0^\prime(\phi |Z|^2)}{I_0(\phi |Z|^2)} \bigg | _{\phi = \phi_2(q;0)}  = q
		\end{align*}
		reveals that $\phi_2(q;0)$ is an increasing function of $q$ since the function on the left is an increasing function of $\phi$ (see Fact \ref{bessel_properties}, Appendix \ref{misc_appendix}). Hence,
		\begin{align*}
		\max_{q \in [0,1-\eta]} \phi_2(q;0)  = \phi_2(q,1-\eta)< \infty.
		\end{align*}
		This concludes the proof of item (1).
		\item It is sufficient to show that the function
		\begin{align*}
		f(q) \Mydef \Xi_2(q;0)  + \left( 1 - \frac{1}{\delta} \right) \ln(1-q^2) + \Delta \ln \left( 1 - \frac{q^2}{2} \right),
		\end{align*}
		is strictly increasing or: 
		\begin{align*}
		\frac{\diff f(q)}{\diff q}> 0.
		\end{align*}
		We can compute the first derivative:
		\begin{align*}
		\frac{\diff f(q) }{\diff q} = \phi_2(q;0) - 2 \left(1 - \frac{1}{\delta} \right) \frac{q}{1-q^2} - \frac{\Delta q}{1-\frac{q^2}{2}}.
		\end{align*} 
		Hence,
		\begin{align*}
		\frac{\diff f(q)}{\diff q} > 0 \Leftrightarrow \phi_2(q;0) > 2 \left(1 - \frac{1}{\delta} \right) \frac{q}{1-q^2} + \frac{\Delta q}{1-\frac{q^2}{2}} \Mydef \phi_3(q)
		\end{align*}
		Note that since $\phi_2(q;0)$ is the maximizing argument of the strictly concave function $q \phi - \E \ln I_0(\phi |Z|^2)$, we have, 
		\begin{align*}
		\frac{\diff f(q)}{\diff q} > 0 \Leftrightarrow \frac{\diff}{\diff \phi} \left( q \phi - \E \ln I_0(\phi |Z|^2) \right) \bigg|_{\phi = \phi_3(q)} > 0 \Leftrightarrow q > \E |Z|^2 \cdot  \frac{I_0^\prime(\phi_3(q) \cdot|Z|^2)}{I_0(\phi_3(q) \cdot |Z|^2)}
		\end{align*}
		Next we make the following sequence of observations: 
		\begin{enumerate}
			\item $\phi_3(0) = 0$, hence,
			\begin{align*}
			\E \frac{I_0^\prime(\phi_3(q) \cdot|Z|^2)}{I_0(\phi_3(q) \cdot |Z|^2)} \bigg |_{q=0} & = 0.
			\end{align*}
			\item We can compute the first derivative: 
			\begin{align*}
			\frac{\diff}{\diff q} \E |Z|^2 \frac{I_0^\prime(\phi_3(q) \cdot|Z|^2)}{I_0(\phi_3(q) \cdot |Z|^2)} \bigg | _{q=0} & = \frac{\diff}{\diff \phi} \left( \E |Z|^2 \frac{I_0^\prime(\phi \cdot|Z|^2)}{I_0(\phi \cdot |Z|^2)} \right) \bigg |_{\phi = 0} \cdot \frac{\diff \phi_3(q)}{\diff q} \bigg |_{q=0} \\ &\explain{(a)}{=} \frac{\E |Z|^4}{2} \cdot \left(2 \left( 1 - \frac{1}{\delta} \right) + \Delta \right) \\ &= 1 - \left( \frac{2-\delta}{\delta} - \Delta \right) < 1,
			\end{align*}
			where the step marked (a) used Fact \ref{bessel_properties} and the definition of $\phi_3(q)$ to compute the relevant derivatives.
			\item Finally we note that, the function,
			\begin{align*}
			q \mapsto \E |Z|^2 \cdot  \frac{I_0^\prime(\phi_3(q) \cdot|Z|^2)}{I_0(\phi_3(q) \cdot |Z|^2)},
			\end{align*}
			is concave and increasing since $\frac{I_0^\prime(x)}{I_0(x)}$ is concave and increasing (Fact \ref{bessel_properties}, Appendix \ref{misc_appendix}) and $\phi_3(q)$ is convex and increasing. 
		\end{enumerate}
		The above three observations immediately imply: 
		\begin{align*}
		\E |Z|^2 \cdot  \frac{I_0^\prime(\phi_3(q) \cdot|Z|^2)}{I_0(\phi_3(q) \cdot |Z|^2)} < q, \; \forall \; q > 0 \implies \frac{\diff f }{\diff q}(q) > 0, \; \forall \; q > 0.
		\end{align*}
		Furthermore,
		\begin{align*}
		\frac{\diff f }{\diff q}(0) = 0.
		\end{align*}
		Hence $f(q)$ is a stricly increasing function of $q$ and hence so is $\CalF(q;\delta,\Delta,0)$.
		This concludes the proof of item (2).
	\end{enumerate}
\end{proof}
\subsection{Convergence to the Low Noise Limit}\label{zero_noise_convergence}
The following lemma shows that $\lim_{\sigma \rightarrow 0} \Xi_1(\sigma) = \Xi_1(0) = 1$.
\begin{lemma} \label{low_noise_denom_variation_pr}
	Recall that $\Xi_1(\sigma)$ and $\lambda_1(\sigma)$ denote the optimal value and solution of the variational problem: 
	\begin{align*}
	\Xi_1(\sigma) &\Mydef  \max_{\lambda \in \mathbb{R}} \left( \lambda - \E_Y  
	\ln \E_{E \sim \Exp{1}} e^{\lambda E} \gpdf(E-Y) \right), \\
	\lambda_1(\sigma) & \Mydef \arg\max_{\lambda \in \mathbb{R}} \left( \lambda - \E_Y  
	\ln \E_{E \sim \Exp{1}} e^{\lambda E} \gpdf(E-Y) \right).
	\end{align*}
	Then we have,
	\begin{enumerate}
		\item $\lambda_1(\sigma) \leq 1$ for all $\sigma > 0$.
		\item $\Xi_1(\sigma)$ is a decreasing function of $\sigma$.
		\item $\lim_{\sigma \rightarrow 0} \Xi_1(\sigma) = 1$. 
	\end{enumerate}
\end{lemma}
\begin{proof}
	First we can write $\E_{E \sim \Exp{1}} e^{\lambda E} \gpdf(E-Y)$ as follows: 
	\begin{align*}
	\E_{E \sim \Exp{1}} e^{\lambda E} \gpdf(E-Y)  &= e^{(\lambda - 1)Y} \E_{\omega \sim \gauss{0}{1}} e^{\sigma(\lambda-1) \omega}  \mathbf{1}_{Y + \sigma\omega \geq 0}
	\end{align*}
	Note that $Y \explain{d}{=} |Z|^2 + \epsilon, \; Z \sim \cgauss{0}{1}, \; \epsilon \sim \gauss{0}{\sigma^2}$. 
	Hence, we have,
	\begin{align}
	\Xi_1(\sigma) & = \max_{\lambda \in \mathbb{R}} 1 - \E_Y \ln  \E_{\omega \sim \gauss{0}{1}} e^{\sigma(\lambda-1) \omega}  \mathbf{1}_{Y + \sigma\omega \geq 0} \nonumber\\
	& = 1 - \min_{\gamma \in \mathbb{R}} \E_Y \ln  \E_{\omega \sim \gauss{0}{1}} e^{\sigma \gamma \omega}  \mathbf{1}_{Y + \sigma\omega \geq 0}. \label{need_a_reference_later}
	\end{align}
	Likewise,
	\begin{align*}
	\lambda_1(\sigma^2) = 1 +  \arg \min_{\gamma \in \mathbb{R}} \E_Y \ln  \E_{\omega \sim \gauss{0}{1}} e^{\sigma \gamma  \omega}  \mathbf{1}_{Y + \sigma\omega \geq 0}
	\end{align*}
	Now we consider the three claims one by one: 
	\begin{enumerate}
		\item Observe that, by the Chebychev Association Inequality ( Fact \ref{chebychev_association} , Appendix \ref{misc_appendix}), 
		\begin{align*}
		\E_\omega e^{-\sigma |\gamma| \omega}  \mathbf{1}_{Y + \sigma\omega \geq 0}  \leq  \E_{\omega} e^{-\sigma |\gamma| \omega} \cdot \P (Y + \sigma\omega \geq 0) = \E_{\omega} e^{\sigma |\gamma| \omega} \cdot \P (Y + \sigma\omega \geq 0) \leq \E_{\omega} e^{\sigma |\gamma| \omega}  \mathbf{1}_{Y + \sigma\omega \geq 0}.
		\end{align*}
		This shows that $\lambda_1(\sigma^2) \leq 1$.
		\item A gaussian integral shows that: 
		\begin{align*}
		\E_\omega e^{\sigma \gamma \omega}  \mathbf{1}_{Y + \sigma\omega \geq 0}  = \frac{1}{\sqrt{2\pi}} \int_{-\frac{Y}{\sigma}}^\infty e^{\sigma \gamma \omega - \frac{\omega^2}{2}} = \frac{e^{\frac{\gamma^2 \sigma^2}{2}}}{\sqrt{2\pi}} \int_{-\frac{Y}{\sigma}}^\infty e^{-\frac{(\omega - \gamma \sigma)^2}{2}} = e^{\frac{\gamma^2 \sigma^2}{2}} \cdot \Phi \left( \frac{Y}{\sigma} + \gamma \sigma  \right)
		\end{align*}
		Hence,
		\begin{align*}
		\Xi_1(\sigma) = 1 - \min_{t} \left( \E_Y \ln \Phi \left( \frac{Y}{\sigma} + t  \right) + \frac{t^2}{2} \right) =  1 - \min_{t} \left( \E_{Z,\epsilon} \ln \Phi \left( \frac{|Z|^2}{\sigma} + \epsilon + t  \right) + \frac{t^2}{2} \right)
		\end{align*}
		Note that $\Phi \left( \frac{|Z|^2}{\sigma} + \epsilon + t  \right)$ increases as $\sigma \downarrow 0$. 
		Consequently, we have, $\Xi_1(\sigma)$ is a decreasing function of $\sigma$.
		\item Recall that in the previous step, we showed that,
		\begin{align*}
		\Xi_1(\sigma)  =  1 - \min_{t} \left( \E_{Z,\epsilon} \ln \Phi \left( \frac{|Z|^2}{\sigma} + \epsilon + t  \right) + \frac{t^2}{2} \right).
		\end{align*} 
		Proposition \ref{denominator_variational_problem_prop} shows that for any $\sigma > 0$, the objective in the definition of $\Xi_1$ is cooercive. Consequently we can identify $-\infty < t_1 < t_2< \infty $ such that, 
		\begin{align*}
		\left( \E_{Z,\epsilon} \ln \Phi \left( |Z|^2 + \epsilon + t  \right) + \frac{t^2}{2} \right) > 0, \; \forall \; t \in (-\infty,t_1) \cup (t_2,\infty).
		\end{align*}
		Since $\Phi$ is an increasing function,
		\begin{align*}
		\left( \E_{Z,\epsilon} \ln \Phi \left( \frac{|Z|^2}{\sigma} + \epsilon + t  \right) + \frac{t^2}{2} \right) > 0, \; \forall \; t \in (-\infty,t_1) \cup (t_2,\infty), \; \forall \sigma \leq 1.
		\end{align*}
		On the other hand,
		\begin{align*}
		\left( \E_{Z,\epsilon} \ln \Phi \left( \frac{|Z|^2}{\sigma} + \epsilon + t  \right) + \frac{t^2}{2} \right) \bigg|_{t = 0} & \leq 0.
		\end{align*}
		Hence we have,
		\begin{align*}
		\Xi_1(\sigma)  =  1 - \min_{t \in [t_1,t_2]} \left( \E_{Z,\epsilon} \ln \Phi \left( \frac{|Z|^2}{\sigma} + \epsilon + t  \right) + \frac{t^2}{2} \right).
		\end{align*}
		Observing that $\left( \E_{Z,\epsilon} \ln \Phi \left( \frac{|Z|^2}{\sigma} + \epsilon + t  \right) + \frac{t^2}{2} \right)$ is a convex function such that for every fixed $t$ we have,
		\begin{align*}
		\lim_{\sigma \rightarrow 0} \left( \E_{Z,\epsilon} \ln \Phi \left( \frac{|Z|^2}{\sigma} + \epsilon + t  \right) + \frac{t^2}{2} \right) = \frac{t^2}{2}.
		\end{align*}
		Due to convexity, this convergence can be made uniform on compact sets: 
		\begin{align*}
		\lim_{\sigma \rightarrow 0} \max_{t \in [t_1,t_2]} \left| \E_{Z,\epsilon} \ln \Phi \left( \frac{|Z|^2}{\sigma} + \epsilon + t  \right) \right| = 0.
		\end{align*}
		This uniform convergence immediately yields $\lim_{\sigma \rightarrow 0} \Xi_1(\sigma) = 1$. 
	\end{enumerate}
\end{proof}
The following lemma analyzes the convergence of $\Xi_2(q;\sigma)$ to $\Xi_2(q;0)$. For our purposes, it turns out, that we don't need to show that $\Xi_2(q;\sigma) \rightarrow \Xi_2(q;0)$ as $\sigma \rightarrow 0$. It is sufficient to show the weaker result that $\Xi_2(q;\sigma)$ is asymptotically lower bounded by $\Xi_2(q;0)$ as $\sigma \rightarrow 0$. This is the content of the following lemma.
\begin{lemma} \label{low_noise_convergence} For any $0< \eta< 1$, we have, $$\liminf_{\sigma \rightarrow 0} \min_{q \in [0,1-\eta]} \Xi_2(q,\sigma) - \Xi_2(q;0) \geq 0.$$ Furthermore we have,
	\begin{align*}
	\lim_{\sigma \rightarrow 0} \Xi_2(0,\sigma) & = \Xi_2(0,0).
	\end{align*}
\end{lemma}
\begin{proof}
	We lower bound $\Xi_2(q,\sigma^2)$ as follows:
	\begin{align*}
	&{\Xi}_2(q,\sigma)  \Mydef \max_{(\lambda,\phi) \in \mathbb{R}} \left( 2\lambda  + \phi q - \E_Y
	\ln \E e^{((\lambda-1)(2Y + \sigma(\omega_1 + \omega_2)) + \phi \sqrt{(Y+\sigma \omega_1)(Y+\sigma \omega_2)} \cos(\theta)} \mathbf{1}_{Y + \sigma \omega_1 \geq 0, y + \sigma \omega_2 \geq 0} \right) \\
	& \geq 2 + q \phi_2(q;0) - \E_Y \ln \E_{\omega_1,\omega_2} I_0 \left( \phi_2(q;0) \cdot \sqrt{(Y+\sigma \omega_1)(Y+\sigma \omega_2)} \right) \mathbf{1}_{Y + \sigma \omega_1 \geq 0, y + \sigma \omega_2 \geq 0} \\
	& = \Xi_2(q;0) + \E_{} \ln I_0(\phi_2(q;0) |Z|^2) - \E_Y \ln \E_{\omega_1,\omega_2} I_0 \left( \phi_2(q;0)  \sqrt{(Y+\sigma \omega_1)(Y+\sigma \omega_2)} \right) \mathbf{1}_{Y + \sigma \omega_1 \geq 0, y + \sigma \omega_2 \geq 0}.
	\end{align*}
	In the above display, we recall that $\phi_2(q;0)$ was defined as,
	\begin{align*}
	\phi_2(q;0) & \Mydef \arg \max_{\phi \in \mathbb R} q \phi - \E_{Z \sim \cgauss{0}{1}} \ln I_0(\phi |Z|^2).
	\end{align*}
	Note that for any fixed $\phi \in \mathbb R$, we have, by Dominated convergence,
	\begin{align*}
	\E_Y \ln \E_{\omega_1,\omega_2} I_0 \left( \phi \sqrt{(Y+\sigma \omega_1)(Y+\sigma \omega_2)} \right) \mathbf{1}_{Y + \sigma \omega_1 \geq 0, y + \sigma \omega_2 \geq 0} \rightarrow \E_{Z \sim \cgauss{0}{1}} \ln I_0(\phi |Z|^2), \; \text{as } \sigma \rightarrow 0.
	\end{align*}
	Observing that the function on the left hand side is convex in $\phi$ we have the above convergence holds uniformly on all compact sets. Lemma \ref{zero_noise_variational_problems} guarantees that $\sup_{q \in [0,1-\eta]} |\phi_2(q;0)| < \infty$. Consequently, we have,
	\begin{align*}
	\sup_{q \in [0,1-\eta]}  \left|  \E_{} \ln I_0(\phi_2(q;0) |Z|^2) - \E_Y \ln \E_{\omega_1,\omega_2} I_0 \left( \phi_2(q;0) \cdot \sqrt{(Y+\sigma \omega_1)(Y+\sigma \omega_2)} \right) \mathbf{1}_{Y + \sigma \omega_1 \geq 0, y + \sigma \omega_2 \geq 0} \right| \explain{$\sigma \rightarrow 0$}{\rightarrow} 0.
	\end{align*}
	Combining this with the lower bound on $\Xi_2(q,\sigma)$ immediately gives:
	\begin{align*}
	\liminf_{\sigma \rightarrow 0} \min_{q \in [0,1-\eta]} \Xi_2(q,\sigma) - \Xi_2(q;0) &\geq 0.
	\end{align*}
	Finally when $q = 0$ we note that, $\Xi_2(0,\sigma) = 2 \Xi_1(0,\sigma)$. Lemma \ref{low_noise_denom_variation_pr} guarantees that $\Xi_2(0,\sigma) \rightarrow 1$ as $\sigma \rightarrow 0$. Note that since $I_0(x)$ is minimized at $x = 0$ (see Fact \ref{bessel_properties}, Appendix \ref{misc_appendix}), we have $\Xi_2(0,0) =2$. Hence we indeed have $\Xi_2(0,\sigma) \rightarrow \Xi_2(0,0)$ as $\sigma \rightarrow 0$. 
\end{proof}
\subsection{Proof of Proposition \ref{lownoiseprop}} \label{low_noise_prop_proof}
Recall that our goal is to find conditions on $(\delta,\Delta,\sigma)$ such that $\CalF(q;\delta,\Delta,\sigma) > \CalF(0;\delta,\Delta,\sigma)$, where,
\begin{align*}
\CalF(q; \delta, \Delta,\sigma) & \Mydef {\Xi}_2(q;\sigma) - 2\Xi_1(\sigma)  + \left( 1 - \frac{1}{\delta} \right) \ln(1-q^2) + \Delta \ln \left( 1- \frac{q^2}{2} \right).
\end{align*}
The following lemma provides a lower bound on the curvature of $\Xi_2(q;\sigma) - 2 \Xi_1(\sigma)$ in the neighborhood of $q \approx 0$.  
\begin{lemma}[Analysis for $q \approx 0$] \label{small_q_analysis}
	There exists a universal constant $C$ (independent of $\sigma$) such that, for any $0 \leq q < 1/2, \; \sigma < 1$ we have, 
	\begin{align*}
	\Xi_2(q,\sigma) - 2 \Xi_1(\sigma) & \geq (1-\sigma^2) \cdot\frac{q^2}{2} - Cq^3.
	\end{align*}
\end{lemma}
\begin{proof}
	We can write $\ZTWis{\lambda}{\phi}{Y}$ (c.f. Definition \ref{titled_wishart_definition}) as: 
	\begin{align*}
	\ZTWis{{\lambda}}{\phi}{y} & \Mydef  \frac{1}{2\pi }\int_0^\infty \int_0^\infty \int_{-\pi}^\pi \exp(-(1-\lambda)(s + s^\prime) + \phi \sqrt{s s^\prime} \cos(\theta)) \cdot \gpdf(s - y) \cdot  \gpdf(s^\prime - y) \diff \theta \diff s \diff s^\prime \\
	& = \E \exp \left((\lambda-1)(2y + \sigma(\omega_1 + \omega_2)) + \phi \sqrt{(y+\sigma \omega_1)(y+\sigma \omega_2)} \cos(\theta) \right) \mathbf{1}_{y + \sigma \omega_1 \geq 0, y + \sigma \omega_2 \geq 0}.
	\end{align*}
	In the above display $\omega_1, \omega_2, \theta$ are independent r.v.s with distributions $\omega_1 \sim \gauss{0}{1}, \; \omega_2 \sim \gauss{0}{1}, \; \theta \sim \text{Uniform}[-\pi,\pi]$.
	We lower bound $\Xi_2$ as follows:
	\begin{align*}
	{\Xi}_2(q,\sigma) & = \max_{(\lambda,\phi) \in \mathbb{R}} \left( 2\lambda  + \phi q - \E_Y
	\ln \ZTWis{\lambda}{\phi}{Y} \right) \\
	& \geq  \left( 2\lambda_1(\sigma^2)  + q^2 - \E_Y
	\ln \ZTWis{\lambda_1(\sigma)}{q}{Y} \right).
	\end{align*}
	Next we will approximate $ \ln \ZTWis{\lambda_1(\sigma)}{q}{Y}$ by its taylor series around $q \approx 0$. We can compute the first three derivatives: 
	\begin{align*}
	\frac{ \diff  }{\diff q} \ln \ZTWis{\lambda_1(\sigma)}{q}{y} \bigg|_{q = 0}  &= 0, \\  \frac{ \diff^2  }{\diff q^2} \ln \ZTWis{\lambda_1(\sigma)}{q}{Y} \bigg|_{q = 0}  &=  \frac{\E (y + \sigma \omega_1)(y + \sigma \omega_2) \cos^2 (\theta) e^{(\lambda_1(\sigma) -1)(2y + \sigma(\omega_1 + \omega_2))} \mathbf{1}_{y + \sigma \omega_1 \geq 0, y + \sigma \omega_2 \geq 0}}{\E  e^{(\lambda_1(\sigma) -1)(2y + \sigma(\omega_1 + \omega_2))} \mathbf{1}_{y + \sigma \omega_1 \geq 0, y + \sigma \omega_2 \geq 0}} \\
	& = \frac{1}{2} \cdot \frac{\E (y + \sigma \omega_1)(y + \sigma \omega_2)  e^{(\lambda_1(\sigma) -1)(2y + \sigma(\omega_1 + \omega_2))} \mathbf{1}_{y + \sigma \omega_1 \geq 0, y + \sigma \omega_2 \geq 0}}{\E  e^{(\lambda_1(\sigma) -1)(2y + \sigma(\omega_1 + \omega_2))} \mathbf{1}_{y + \sigma \omega_1 \geq 0, y + \sigma \omega_2 \geq 0}} \\
	& = \frac{1}{2} \left( \frac{\E (y + \sigma \omega_1)e^{(\lambda_1(\sigma) -1)(y + \sigma \omega_1)} \mathbf{1}_{y + \sigma \omega_1 \geq 0}}{\E e^{(\lambda_1(\sigma) -1)(y + \sigma \omega_1)} \mathbf{1}_{y + \sigma \omega_1 \geq 0}} \right)^2 \\
	& \explain{(a)}{\leq} \frac{1}{2} \cdot  \left( \E (y + \sigma \omega_1) \mathbf{1}_{y + \sigma \omega_1 \geq 0} \right)^2\\
	& \leq \frac{1}{2} \cdot \E (y + \sigma \omega_1)^2
	\\ & =  \frac{y^2}{2} + \frac{\sigma^2}{2}.
	\end{align*}
	In the step marked (a), we used the fact that $\lambda_1(\sigma^2) \leq 1$ (see Lemma \ref{low_noise_denom_variation_pr}) and Chebychev's Association Inequality (Fact \ref{chebychev_association}, Appendix \ref{misc_appendix}). Similarly we control the third derivative: 
	\begin{align*}
	\frac{ \diff^3  }{\diff q^3} \ln \ZTWis{\lambda_1(\sigma)}{q}{Y} & = T_3-3T_2 T_1 + 2T_1^3,
	\end{align*}
	where, for $\; i = 1,2,3$:
	\begin{align*}
	T_i & \Mydef \frac{\E (y + \sigma \omega_1)^{\frac{i}{2}}(y + \sigma \omega_2)^{\frac{i}{2}} \cos^i (\theta) e^{(\lambda_1(\sigma) -1)(2y + \sigma(\omega_1 + \omega_2)) + q \sqrt{(y+\sigma \omega_1)(y+\sigma \omega_2)} \cos(\theta)} \mathbf{1}_{y + \sigma \omega_1 \geq 0, y + \sigma \omega_2 \geq 0}}{\E  e^{(\lambda_1(\sigma) -1)(2y + \sigma(\omega_1 + \omega_2)) + q \sqrt{(y+\sigma \omega_1)(y+\sigma \omega_2)} \cos(\theta)} \mathbf{1}_{y + \sigma \omega_1 \geq 0, y + \sigma \omega_2 \geq 0}}.
	\end{align*}
	We can control $T_i$ as follows, for any $q \in [0,1]$ and any $\sigma \leq 1$, we have,
	\begin{align*}
	|T_i| & \leq \frac{\E (y + \sigma \omega_1)^{\frac{i}{2}}(y + \sigma \omega_2)^{\frac{i}{2}} e^{(\lambda_1(\sigma) -1)(2y + \sigma(\omega_1 + \omega_2)) + q \sqrt{(y+\sigma \omega_1)(y+\sigma \omega_2)} \cos(\theta)} \mathbf{1}_{y + \sigma \omega_1 \geq 0, y + \sigma \omega_2 \geq 0}}{\E  e^{(\lambda_1(\sigma) -1)(2y + \sigma(\omega_1 + \omega_2)) + q \sqrt{(y+\sigma \omega_1)(y+\sigma \omega_2)} \cos(\theta)} \mathbf{1}_{y + \sigma \omega_1 \geq 0, y + \sigma \omega_2 \geq 0}} \\
	& \explain{(a)}{=} \frac{\E (y + \sigma \omega_1)^{\frac{i}{2}}(y + \sigma \omega_2)^{\frac{i}{2}} e^{(\lambda_1(\sigma) -1)(2y + \sigma(\omega_1 + \omega_2))} I_0( q \sqrt{(y+\sigma \omega_1)(y+\sigma \omega_2)}) \mathbf{1}_{y + \sigma \omega_1 \geq 0, y + \sigma \omega_2 \geq 0}}{\E  e^{(\lambda_1(\sigma) -1)(2y + \sigma(\omega_1 + \omega_2))} I_0( q \sqrt{(y+\sigma \omega_1)(y+\sigma \omega_2)}) \mathbf{1}_{y + \sigma \omega_1 \geq 0, y + \sigma \omega_2 \geq 0}} \\
	& \explain{(b)}{\leq} \frac{\E (y + \sigma \omega_1)^{\frac{i}{2}}(y + \sigma \omega_2)^{\frac{i}{2}} e^{(\lambda_1(\sigma) -1)(2y + \sigma(\omega_1 + \omega_2))} e^{ q \sqrt{(y+\sigma \omega_1)(y+\sigma \omega_2)}} \mathbf{1}_{y + \sigma \omega_1 \geq 0, y + \sigma \omega_2 \geq 0}}{\E  e^{(\lambda_1(\sigma) -1)(2y + \sigma(\omega_1 + \omega_2))}  \mathbf{1}_{y + \sigma \omega_1 \geq 0, y + \sigma \omega_2 \geq 0}} \\
	& \leq \frac{\E (y + \sigma \omega_1)^{\frac{i}{2}}(y + \sigma \omega_2)^{\frac{i}{2}} e^{(\lambda_1(\sigma) -1)(2y + \sigma(\omega_1 + \omega_2))} e^{ \frac{q}{2} \cdot  (y+\sigma \omega_1+ y+\sigma \omega_2)} \mathbf{1}_{y + \sigma \omega_1 \geq 0, y + \sigma \omega_2 \geq 0}}{\E  e^{(\lambda_1(\sigma) -1)(2y + \sigma(\omega_1 + \omega_2))}  \mathbf{1}_{y + \sigma \omega_1 \geq 0, y + \sigma \omega_2 \geq 0}} \\
	& = \left( \frac{\E (y + \sigma \omega_1)^{\frac{i}{2}}e^{(\lambda_1(\sigma) -1 + \frac{q}{2})(y + \sigma \omega_1)} \mathbf{1}_{y + \sigma \omega_1 \geq 0}}{\E e^{(\lambda_1(\sigma) -1  + \frac{q}{2})(y + \sigma \omega_1 )} \mathbf{1}_{y + \sigma \omega_1 \geq 0}} \right)^2 \\
	& \explain{(c)}{\leq} \left( \E (y + \sigma \omega_1)^{\frac{i}{2}} e^{\frac{q}{2}(y + \sigma \omega_1)} \mathbf{1}_{y + \sigma \omega_1 \geq 0}\right)^2 \\
	& \leq \E |y + \sigma \omega_1|^{i} e^{q(y + \sigma \omega_1)} \\
	& \leq C (|y|^3 + 1) e^{qy}.
	\end{align*}
	where, $C$ is a universal constant independent of $\sigma$. In the step marked (a), we used the definition of Modified Bessel Function (see Fact \ref{bessel_properties}, Appendix \ref{misc_appendix}). In the step marked (b), we used $1 \leq I_0(x) \leq e^x$ for any $x \in \mathbb R$. 
	In the step marked (c) we recalled $\lambda_1(\sigma) \leq 1$ and applied Chebychev's Association Inequality. In conclusion, we obtained the following: 
	\begin{align*}
	\frac{ \diff  }{\diff q} \ln \ZTWis{\lambda_1(\sigma)}{q}{y} \bigg|_{q = 0} &= 0, \\ \frac{ \diff^2  }{\diff q^2} \ln \ZTWis{\lambda_1(\sigma)}{q}{y} \bigg|_{q = 0} &\leq \frac{y^2 + \sigma^2}{2}, \\ \frac{ \diff^3  }{\diff q^3} \ln \ZTWis{\lambda_1(\sigma)}{q}{y} &\leq C(|y|^3 +1) e^{qy}. 
	\end{align*}
	This allows us to upper bound $ \ln \ZTWis{\lambda_1(\sigma)}{q}{Y}$ for any $q<1/2, \sigma < 1$ as follows:
	\begin{align*}
	\E\ln \ZTWis{\lambda_1(\sigma)}{q}{Y}& \leq \E_Y \left[ \ln \ZTWis{\lambda_1(\sigma)}{0}{Y} + \frac{q^2}{4} \cdot (\sigma^2 + Y^2) + C q^3 (|Y|^3 + 1)e^{\frac{Y}{2}}  \right].
	\end{align*}
	Observing that $\E Y^2 = \E|Z|^4 + \sigma^2 = 2 + \sigma^2$. We obtain,
	\begin{align*}
	\Xi_2(q,\sigma) - 2\Xi_1(\sigma) \geq \frac{q^2}{2} \left(1 - \sigma^2 \right) - C q^3.
	\end{align*}
\end{proof}
Next we show that at $q \rightarrow 1$, $\Xi_2(q;\sigma) - 2 \Xi_1(\sigma) \rightarrow \infty$ in the following lemma. 
\begin{lemma}[Analysis at $q \approx 1$] \label{q_close_to_1_lemma} There exists a universal finite constant $C>0$ (independent of $\sigma,q$) such that, for all $\sigma \leq 1$,
	\begin{align*}
	\Xi_2(q,\sigma) - 2 \Xi_1(\sigma) & \geq - C  - \frac{\ln(1-q)}{2}, \; \forall \sigma > 0, \; \forall \; q \; \in \; [0,1).
	\end{align*}
\end{lemma}
\begin{proof}
	We lower bound $\Xi_2$ as follows:
	\begin{align*}
	{\Xi}_2(q,\sigma) & \Mydef \max_{(\lambda,\phi) \in \mathbb{R}} \left( 2\lambda  + \phi q - \E_Y
	\ln \ZTWis{\lambda}{\phi}{Y} \right) \\
	& \geq  -\frac{1}{2(1-q)}+ \frac{q}{2(1-q)} - \E_Y \left[  \ln \ZTWis{-\frac{1}{4(1-q)}}{\frac{1}{2(1-q)}}{Y} \right] \\
	& = -2 - \E_Y \left[  \ln \ZTWis{-\frac{1}{4(1-q)}}{\frac{1}{2(1-q)}}{Y} \right].
	\end{align*}
	Recall that, 
	\begin{align*}
	\ZTWis{-\frac{1}{4(1-q)}}{\frac{1}{2(1-q)}}{y} & =  \E e^{ -(2y + \sigma(\omega_1 + \omega_2)) \left( \frac{1}{4(1-q)} + 1 \right) +  \frac{ \sqrt{y + \sigma \omega_1} \sqrt{y + \sigma \omega_2} \cos(\theta) }{2(1-q)}} \mathbf{1}_{y + \sigma \omega_1 \geq 0,y + \sigma \omega_2 \geq 0 } \\
	& \explain{(a)}{\leq} e^{-\frac{y}{2(1-q)}} \cdot \E e^{-\frac{\sigma(\omega_1 + \omega_2)}{4(1-q)}} \cdot I_0 \left( \frac{\sqrt{y + \sigma \omega_1} \sqrt{y + \sigma \omega_2}}{2(1-q)} \cdot   \right) \mathbf{1}_{y + \sigma \omega_1 \geq 0,y + \sigma \omega_2 \geq 0 } \\
	& \explain{(b)}{\leq} e^{-\frac{y}{2(1-q)}} \cdot\E e^{-\frac{\sigma(\omega_1 + \omega_2)}{4(1-q)}} \cdot I_0 \left( \frac{2y + \sigma \omega_1 + \sigma \omega_2}{4(1-q)}\right) \mathbf{1}_{y + \sigma \omega_1 \geq 0,y + \sigma \omega_2 \geq 0 }.
	\end{align*}
	In the above display $\omega_2,\omega_2, \theta$ are independent with $\omega \sim \gauss{0}{1}, \omega_2 \sim \gauss{0}{1}, \theta \sim \text{Uniform}[-\pi,\pi]$. In the step marked (a), we used the definition of Bessel Function $I_0$ (see Fact \ref{bessel_properties}). In the step marked (b), we used AM-GM Inequality and the fact that $I_0(x)$ is increasing on $x \geq 0$ (Fact \ref{bessel_properties}). Further applying the upper bound $I_0(x) \leq C x^{-\frac{1}{2}} \cdot e^{x}, x \geq 0$ (see Fact \ref{bessel_properties}), gives,
	\begin{align*}
	\ZTWis{-\frac{1}{4(1-q)}}{\frac{1}{2(1-q)}}{y} & \leq C \cdot  \sqrt{1-q} \cdot   \E  \frac{1}{\sqrt{|2y + \sigma \omega_1 + \sigma  \omega_2}|}.
	\end{align*}
	Hence we have,
	\begin{align*}
	\E_{Y}  \ln  \ZTWis{-\frac{1}{4(1-q)}}{\frac{1}{2(1-q)}}{Y} &  = \E_{Z,\epsilon} \ln  \ZTWis{-\frac{1}{4(1-q)}}{\frac{1}{2(1-q)}}{|Z|^2 + \epsilon} \\
	& \leq \E_{Z} \ln \E_\epsilon  \ZTWis{-\frac{1}{4(1-q)}}{\frac{1}{2(1-q)}}{|Z|^2 + \epsilon} \\
	& \leq \ln C + \frac{\ln(1-q)}{2} + \E_Z \ln \E_{\epsilon,\omega_1,\omega_2} \frac{1}{\sqrt{|2y + \sigma \omega_1 + \sigma  \omega_2}|} \\
	& \explain{(c)}{\leq} \ln C + \frac{\ln(1-q)}{2} + \ln(4) - \frac{1}{2} \E_Z \ln |Z|^2 \\
	& \explain{(d)}{\leq} C + \frac{\ln(1-q)}{2}
	\end{align*}
	In the step marked (c), we appealed to Lemma \ref{fractional_moments}. In the step marked (d), we used the fact that $\E \ln |Z|^2 = \int_0^\infty \ln(r) e^{-r}\diff r \approx - 0. 58$ is finite.  Hence we have the lower bound on $\Xi_2$:
	\begin{align*}
	\Xi_2(q,\sigma) & \geq - C  - \frac{\ln(1-q)}{2}.
	\end{align*}
	Lemma \ref{low_noise_denom_variation_pr} shows that $\Xi_1(\sigma) \leq \Xi_1(1)$ which is an absolute constant, consequently,
	\begin{align*}
	\Xi_2(q,\sigma) - 2 \Xi_1(\sigma) & \geq - C - \frac{\ln(1-q)}{2}
	\end{align*}
\end{proof}

We finally put together all the different auxiliary results we have established so far and prove Proposition \ref{lownoiseprop} which is restated below for convenience. 
\lownoiseprop*
\begin{proof}
	We will prove the above claims in 3 steps: 1) Step 1: Analysis around $q \approx 0$, 2) Step 2: Analysis around $q \approx 1$ and 3) Step 3: Analysis for all other values of $q$. 
	\begin{description}
		\item[Step 1: $q \approx 0.$] Lemma \ref{small_q_analysis} guarantees the existence of a universal constant $C_1>0$ independent of $\sigma,\delta,\Delta$ such that, for any $q \in [0,0.25], \; \sigma < 1$, we have, 
		\begin{align*}
		\CalF(q; \delta,\Delta, \sigma) - \CalF(0;\delta,\Delta,\sigma) & \geq  \frac{q^2}{2} \cdot \left(\frac{2-\delta}{\delta} -\Delta - \sigma^2\right) - C_1 q^3 - \left( 1 - \frac{1}{\delta} \right) q^4 - \frac{\Delta}{2} q^4 \\
		& \geq \frac{q^2}{2} \cdot \left(\frac{2-\delta}{\delta} -\Delta - \sigma^2\right) - (C_1 + 2) \cdot  q^3
		\end{align*}
		In particular ensuring that, 
		\begin{align*}
		\sigma & \leq \frac{1}{2} \left( \frac{2-\delta}{\delta} -\Delta \right), \; q \leq \frac{1}{8(C_1 + 2)} \cdot \left( \frac{2-\delta}{\delta} - \Delta \right), 
		\end{align*}
		gives us,
		\begin{align}
		\label{small_q_conclusion}
		\CalF(q; \delta,\Delta, \sigma) & \geq \frac{q^2}{8} \cdot \left( \frac{2-\delta}{\delta} -\Delta \right).
		\end{align}
		Note that $\CalF(0; \delta, \Delta, \sigma) = 0$. Hence,  \eqref{small_q_conclusion} verifies claim (1) of the proposition for small $q$: 
		\begin{align*}
		\CalF(q; \delta,\Delta, \sigma) & \geq  \CalF(0; \delta, \Delta, \sigma) + \frac{q^2}{8} \cdot \left( \frac{2-\delta}{\delta} -\Delta \right), \; \forall \; q \in [0,\eta_1(\delta,\Delta)],  \; \sigma \leq \sigma_1(\delta,\Delta), 
		\end{align*}
		where,
		\begin{align*}
		\eta_1(\delta,\Delta) \Mydef \frac{1}{8(C_1 + 2)} \cdot \left( \frac{2-\delta}{\delta} - \Delta \right), \; \sigma_1(\delta,\Delta) \Mydef \frac{1}{2} \left( \frac{2-\delta}{\delta} -\Delta \right).
		\end{align*}
		Furthermore since,
		\begin{align*}
		\frac{\diff  \CalF}{\diff q}(q; \delta,\Delta,\sigma) \bigg |_{q= 0} = 0 \implies \CalF(q;\delta,\Delta,\sigma) = \CalF(0;\delta,\Delta,\sigma) + \frac{q^2}{2} \cdot \frac{\diff^2 \CalF}{\diff q^2}(q; \delta,\Delta,\sigma) \bigg|_{q=0} + o(q).
		\end{align*}
		Comparing the above display with \eqref{small_q_conclusion}, gives us claim (2) of the proposition: 
		\begin{align*}
		\frac{\diff^2 \CalF}{\diff q^2}(q; \delta,\Delta,\sigma) \bigg|_{q=0} & \geq \frac{1}{4} \cdot \left( \frac{2-\delta}{\delta} -\Delta \right) > 0.
		\end{align*}
		\item [Step 2: $q \approx 1$.] Lemma \ref{q_close_to_1_lemma} guarantees the existence of a universal constant $C_2$ such that,
		\begin{align*}
		\Xi_2(q,\sigma) - 2 \Xi_1(\sigma) & \geq - C_2  - \frac{\ln(1-q)}{2}, \; \forall \sigma > 0, \; \forall \; q \; \in \; [0,1), \; \forall \sigma \leq 1.
		\end{align*}
		Consequently, 
		\begin{align*}
		\CalF(q; \delta,\Delta,\sigma) & \geq  - \left( \frac{2-\delta}{2\delta} \right) \ln(1-q) + \left(1 - \frac{1}{\delta} \right) \ln(1+q) + \Delta \ln \left(1 - \frac{q^2}{2}  \right) - C_2 \\& \geq  - \left( \frac{2-\delta}{2\delta} \right) \ln(1-q) - (C_2 + 1).
		\end{align*}
		Hence we have,
		\begin{align*}
		\CalF(q; \delta,\Delta,\sigma)  \geq 1 \geq 0 = \CalF(0; \delta,\Delta,\sigma)  \; \forall \; q \in [1-\eta_2(\delta),1), \; \sigma \leq 1,
		\end{align*}
		where,
		\begin{align*}
		\eta_2(\delta)  & = \exp \left( - \frac{\delta(C_2 + 2)}{2-\delta} \right) > 0.
		\end{align*}
		This verifies claim (1) of the proposition for large $q$. 
		\item [Case 3: Other values of $q$.] In Steps (1) and (2), we have verified Claim (1) for $q \in [0,\eta_1(\delta,\Delta)] \cup [\eta_2(\delta),1]$. Now we focus our attention to:
		\begin{align*}
		q \in [\eta_1(\delta,\Delta), 1 - \eta_2(\delta)].
		\end{align*}
		Note that it is sufficient to show that,
		\begin{align*}
		f(q; \delta,\Delta,\sigma) \Mydef \Xi_2(q,\sigma) + \left( 1 - \frac{1}{\delta} \right) \ln(1-q^2) + \Delta \ln \left( 1 - \frac{q^2}{2} \right),
		\end{align*}
		satisfies $f(q; \delta,\Delta,\sigma) > f(0; \delta,\Delta, \sigma) \; \forall q \in [\eta_1(\delta,\Delta), 1 - \eta_2(\delta)]$.  
		In Lemma \ref{zero_noise_variational_problems}, we had shown that the function:
		\begin{align*}
		f(q; \delta,\Delta,0) \Mydef \Xi_2(q;0)  + \left( 1 - \frac{1}{\delta} \right) \ln(1-q^2) + \Delta \ln \left( 1 - \frac{q^2}{2} \right),
		\end{align*}
		is strictly increasing and has the property that $f(0; \delta,\Delta,0)< f(q; \delta,\Delta,0), \; \forall \; q \in (0,1)$. Consequently, $\eta_3(\delta,\Delta)$ defined below is strictly positive: 
		\begin{align}
		\label{eta_3_def}
		\eta_3(\delta,\Delta) \Mydef \min_{q \in [\eta_1(\delta,\Delta),1)} f(q; \delta,\Delta,0) - f(0; \delta,\Delta,0) = f(\eta_1(\delta,\Delta); \delta, \Delta,0) - f(0; \delta,\Delta,0) > 0.
		\end{align}
		Furthermore, Lemma \ref{low_noise_convergence} shows that $f(q; \delta,\Delta,\sigma)$ is asymptotically lower bounded by $f(q; \delta,\Delta,0)$ in the following sense: 
		\begin{align*}
		\liminf_{\sigma \rightarrow 0} \min_{q \in [0,1-\eta_2(\delta)]} f(q; \delta,\Delta,\sigma) - f(q; \delta, \Delta,0) & \geq 0.
		\end{align*}
		Furthermore, it also guarantees $f(0; \delta,\Delta, \sigma) \rightarrow f(0;\delta,\Delta,0)$ as $\sigma \rightarrow 0$. 
		Consequently there exists $\sigma_3(\delta,\Delta)>0$ such that, 
		\begin{align}
		f(q; \delta,\Delta,\sigma) - f(q; \delta, \Delta,0) &\geq - \frac{\eta_3(\delta,\Delta)}{3}, \; \forall q \in [0,1-\eta_2(\delta)], \; \forall \sigma \leq \sigma_3(\delta,\Delta), \label{small_sigma_conseq_1}\\
		|f(0; \delta,\Delta,\sigma) - f(0; \delta, \Delta,0)| & \leq \frac{\eta_3(\delta,\Delta)}{3} \; \forall \sigma \leq \sigma_3(\delta,\Delta) \label{small_sigma_conseq_2}.
		\end{align}
		Hence, $ \forall q \in [\eta_1(\delta,\Delta),1-\eta_2(\delta)]$,
		\begin{align*}
		f(q; \delta,\Delta,\sigma) &  \explain{ \eqref{small_sigma_conseq_1}}{\geq}  f(q; \delta, \Delta, 0)  - \frac{\eta_3(\delta,\Delta)}{3} \\
		& \hspace{0.2cm}\explain{}{\geq}\hspace{0.2cm} f(0; \delta, \Delta,0) + \left(f(q; \delta, \Delta, 0)  -f(0; \delta, \Delta,0)\right) - \frac{\eta_3(\delta,\Delta)}{3} \\
		& \explain{ \eqref{eta_3_def}}{\geq} f(0; \delta, \Delta,0) + \frac{2\eta_3(\delta,\Delta)}{3} \\
		& \explain{\eqref{small_sigma_conseq_2}}{\geq} f(0;\delta,\Delta,\sigma) + \frac{\eta_3(\delta,\Delta)}{3} \\
		& \hspace{0.2cm} > \hspace{0.2cm} f(0; \delta,\Delta,\sigma).
		\end{align*}
	\end{description}
	This concludes the proof of the proposition. 
\end{proof}
\section{Properties of the Tilted Exponential and Wishart Distributions}
\subsection{Properties of the Tilted Exponential Distribution}
\label{texp_properties_appendix}
The following lemma collects some properties of $\Texp{\lambda}{y}$ random variables which were used to prove the local CLT given in Proposition \ref{local_clt_denom}.
\begin{lemma}[Properties of $\Texp{\lambda}{y}$ Distribution] \label{lemma_texp_properties}
	Let $T \sim \Texp{\lambda}{y}$. We have,
	\begin{enumerate}
		\item Moment Bounds: For any $k \in \mathbb{N}$ we have: 
		\begin{align*}
		\E |T|^k & \leq C_k \left( |y|^k + |\lambda|^k + 1 \right).
		\end{align*}
		In the above display, $C_k$ is a universal constant independent of $y,\lambda$ but depends on $k$.
		\item Decay of characteristic function: For any $t \in \mathbb R$
		\begin{align*}
		| \E e^{it T}| & \leq \frac{C(1+|y|+|\lambda|)}{|t|},
		\end{align*}
		where $C$ is a constant independent of $t,y,\lambda$. 
	\end{enumerate}
\end{lemma}
\begin{proof}
	\begin{enumerate}
		\item Since $T \geq 0$, $|T| = T$. We first observe that, 
		\begin{align*}
		\E T^k & =\E T^k \mathbf{1}_{T \leq |y| + \sigma^2 |\lambda|}+ \E T^k \mathbf{1}_{T \geq |y| + \sigma^2 |\lambda|} \\
		& \leq (|y| + \sigma^2 |\lambda|)^k + \E T^k \mathbf{1}_{T \geq |y| + \sigma^2 |\lambda|}.
		\end{align*} Let $E \sim \Exp{1}$. Using the formula for the density of $\Texp{\lambda}{y}$ distribution in Definition \ref{Dr_representation_formula}, it is easy to see that, 
		\begin{align*}
		\E T^k \mathbf{1}_{T \geq |y| + \sigma^2 |\lambda|} & = \frac{\E E^k e^{\lambda E} \gpdf(E-y)\mathbf{1}_{E \geq |y| + \sigma^2 |\lambda|}}{\E e^{\lambda E} \gpdf(E-y)}.
		\end{align*}
		We observe that $f(e) = e^k$ is increasing and $g(e) = e^{\lambda e} \gpdf(e-y)$ is decreasing when $e \geq |y| + \sigma^2 |\lambda|$. Consequently by Chebychev's Association Inequality (Lemma \ref{chebychev_association}, Appendix \ref{misc_appendix}) we obtain,
		\begin{align*}
		\E T^3 \mathbf{1}_{T \geq |y| + \sigma^2 |\lambda|} &  \leq \E E^k = k!.
		\end{align*}
		Hence for a suitable constant $C_k$, independent of $\lambda,y$ we have, 
		\begin{align*}
		\E T^k & \leq C_k(|y|^k + |\lambda|^k + 1).
		\end{align*}
		\item Let $f(u)$ denote the pdf of $\Texp{\lambda}{y}$. We bound the characteristic function as follows: 
		\begin{align*}
		|\E e^{it T} |& = \left| \int_0^\infty e^{it u} f(u)  \diff u \right| \\
		& = \frac{1}{|t|} \left| \int_0^\infty \frac{\diff}{\diff u} e^{it u} f(u)  \diff u \right| \\
		& = \frac{1}{|t|} \cdot \left|  - f(0) + \int_0^\infty f^\prime(u) e^{it u} \diff u \right| \\
		& \leq \frac{f(0) + \|f^\prime\|_1}{|t|}.
		\end{align*}
		We further upper bound $\|f^\prime\|_1$. Note that:
		\begin{align*}
		f^\prime(u) & = f(u) \cdot \left( \lambda - 1 - \frac{u}{\sigma^2} \right).
		\end{align*}
		Consequently, for a suitable constant $C$ (independent of $\lambda,y$) we obtain the estimate,
		\begin{align*}
		\|f^\prime\|_1 & \leq |\lambda| + 1 + \frac{\E T}{\sigma^2} \\& \leq C(1+ |y| + |\lambda|).
		\end{align*}
		In the last step, we used the estimate on $\E T$ from part (1) of this lemma. 
		Next we upper bound $f(0)$. Note that,
		\begin{align*}
		f(0) & = \left( \int_0^\infty \exp\left( u\left(\lambda - 1 + \frac{y}{\sigma^2}\right)  - \frac{u^2}{2\sigma^2} \right) \diff u \right)^{-1} \\
		& \leq \left( \int_0^1 \exp\left( u\left(\lambda - 1 + \frac{y}{\sigma^2}\right)  - \frac{u^2}{2\sigma^2} \right) \diff u \right)^{-1} \\
		& \explain{(a)}{\leq} \left( \int_0^1 \exp\left( u\left(\lambda - 1 + \frac{y}{\sigma^2}- \frac{1}{2\sigma^2}\right)   \right) \diff u \right)^{-1} \\
		& \leq  \left( \int_0^{(|\lambda| + |y|/\sigma^2 + 1 + 1/2\sigma^2)^{-1}/2} \exp\left( u\left(\lambda - 1 + \frac{y}{\sigma^2}- \frac{1}{2\sigma^2}\right)   \right) \diff u \right)^{-1} \\
		& \explain{(b)}{\leq}  \left( \int_0^{(|\lambda| + |y|/\sigma^2 + 1 + 1/2\sigma^2)^{-1}/2} \left( 1+ u\left(\lambda - 1 + \frac{y}{\sigma^2}- \frac{1}{2\sigma^2}\right)   \right) \diff u \right)^{-1} \\
		& \explain{(c)}{\leq} C(|\lambda| + |y| + 1).
		\end{align*}
		In the step marked (a) we used $u^2 \leq u, \; u \in (0,1)$. In the step marked (b) we used the lower bound $e^{x} \geq 1+ x$. Finally in the step marked (c), we observed that the integrand is larger than $1/2$ in the domain of integration. 
		Combining the bounds on $f(0)$ and $\|f^\prime\|_1$ gives us the required result: 
		\begin{align*}
		| \E e^{it T}| & \leq \frac{C(1+|y|+|\lambda|)}{|t|}.
		\end{align*}
	\end{enumerate}
\end{proof}
\subsection{Properties of the Tilted Wishart Distribution}
\label{twis_properties_appendix}
The following lemma collects some properties of the tilted Wishart distribution which were used to prove the Local CLT given in Proposition \ref{local_clt_num}.
\begin{lemma}\label{tilted_wishart_properties}
	Suppose that:
	\begin{align*}
	\bm S = \begin{bmatrix} r & \sqrt{r r^\prime} e^{\i \theta} \\ \sqrt{r r^\prime} e^{-\i \theta} & r_2 \end{bmatrix} \sim \TWis{\lambda}{\phi}{y} 
	\end{align*} Then, there exists a universal constant $0<C< \infty$ depending only on $\sigma$ such that:
	\begin{enumerate}
		\item Equivalent Characterization: For any bounded measurable function $f$ we have,
		\begin{align*}
		\E f(\bm S) & = \frac{\E_{\bm g \sim \cgauss{\bm 0}{\bm I_2}} f(\bm g \bm g^\UH) e^{\ip{\bm \Lambda}{\bm g \bm g^\UH}} \gpdf(y-|g_1|^2)\gpdf(y-|g_2|^2)}{ \E_{\bm g \sim \cgauss{\bm 0}{\bm I_2}} e^{\ip{\bm \Lambda}{\bm g \bm g^\UH}}\gpdf(y-|g_1|^2)\gpdf(y-|g_2|^2)}.
		\end{align*}
		In the above display, $\bm \Lambda$ is given by: 
		\begin{align*}
		\bm \Lambda & = \begin{bmatrix} \lambda & \phi/2 \\ \phi/2 & \lambda \end{bmatrix}.
		\end{align*}
		\item The density:
		\begin{align*}
		\tilde{h}_{\lambda, \phi, y}(g,g^\prime) & = \frac{e^{-(1-\lambda)(|g| + |g^\prime|^2) + \phi \Re(g \bar{g}^\prime)} \gpdf(|g|^2 - y) \gpdf(|g^\prime|^2 - y)}{\ZTWis{\lambda}{\phi}{y}},
		\end{align*}
		on $\mathbb{C}^2$ is locally bounded, that is: 
		\begin{align*}
		\tilde{h}_{\lambda,\phi,y} (g_0,g^\prime_0) & \leq C(1+|\lambda|^4 + |\phi|^4 + |y|^4) (1 + |g_0|^{12} + |g_0^\prime|^{12}).
		\end{align*}
		\item Tail Bound: With probability $1-\epsilon$,
		\begin{align*}
		r & \leq C \sqrt{1+ |y|^2 + |\phi|^2 + |\lambda|^2} + C \sqrt{\ln \frac{1}{\epsilon}}.
		\end{align*} 
		The analogous result holds for $r^\prime$.
		\item Moment Bounds: For any $k \in \mathbb N$,  There exists a universal constant $C_k$ depending only on $k$ such that 
		\begin{align*}
		\E r^k &\leq C (1+ |\lambda|^k + |\phi|^k + |y|^k).
		\end{align*}
		\item Decay of Characteristic Function: 
		\begin{align*}
		\left|\E e^{\i \ip{\bm T}{\bm S}} \right | & \leq \frac{ C\cdot(1+|\lambda|^{20} + |\phi|^{20} + |y|^{20})}{\|\bm T\|^{\frac{1}{3}}}.
		\end{align*}
		\item For any $y,\lambda,\phi \in \mathbb R$, we have,
		\begin{align*}
		0< \lambda_{\min} \left( \VTWis{\lambda}{\phi}{y} \right) < \lambda_{\max} \left( \VTWis{\lambda}{\phi}{y} \right)< \infty.
		\end{align*}
	\end{enumerate}
\end{lemma}
\begin{proof}
	Throughout this proof, we use $C$ to denote constants that depend only on the noise level $\sigma$ and in particular are independent of the parameters $\lambda,\phi,y$. 
	\begin{enumerate}
		\item We write $g_1 = \sqrt{r} e^{\i \omega}, g_2 = \sqrt{r^\prime} e^{\i \omega^\prime}$. Using standard properties of the complex  gaussian distribution, we know that $r, r^\prime \sim \Exp{1}$ and $\omega, \omega^\prime \sim \text{Unif}(-\pi,\pi]$. Consequently, 
		\begin{align*}
		\bm g \bm g^\UH & = \begin{bmatrix} r & \sqrt{r r^\prime} e^{\i (\omega - \omega^\prime)} \\ \sqrt{r r^\prime} e^{\i (\omega^\prime - \omega)} & r^\prime  \end{bmatrix}
		\end{align*}
		Let $\theta \sim \text{Unif}(-\pi,\pi]$. Then we have $e^{\i (\omega - \omega^\prime)} \explain{d}{=} e^{\i \theta}$. Consequently, 
		\begin{align*}
		\E f(\bm g \bm g^\UH) e^{\ip{\bm \Lambda}{\bm g \bm g^\UH}}\gpdf(y-|g_1|^2)\gpdf(y-|g_2|^2) =\hspace{8cm} \\ \hspace{2cm} \frac{1}{2\pi}\int_0^\infty \int_0^\infty \int_{-\pi}^\pi f(r,r^\prime,\theta) e^{-(1-\lambda)(r + r^\prime) + \phi \sqrt{r r^\prime} \cos(\theta)} \gpdf(y-r)\gpdf(y-r^\prime) \diff r \diff r^\prime \diff \theta.
		\end{align*}
		Comparing this with the density of $\bm S$ from Definition \ref{titled_wishart_definition} gives us the claim of item (1). Note that the Tilted Wishart distribution is supported on rank-1 Hermitian matrices. In particular, it does not have a density with respect to the Lebesgue measure on Hermitian matrices. The advantage of the alternate way of computing expectations of functions of $\bm S$ is that they can be computed by integrating with respect to the proper PDF $\tilde{h}_{\lambda, \phi, y}(g,g^\prime)$ on $\mathbb{C}^2$: 
		\begin{align*}
		\E f(\bm S) & = \int_{\mathbb{C}^2} f \left( \begin{bmatrix} |g|^2 & g\bar{g}^\prime \\ \bar{g}g^\prime & |g^\prime|^2 \end{bmatrix}  \right) \tilde{h}_{\lambda,\phi,y}(g,g^\prime) \diff g \diff g^\prime,
		\end{align*}
		where the pdf $\tilde{h}_{\lambda, \phi, y}(g,g^\prime)$ is given by: 
		\begin{align*}
		\tilde{h}_{\lambda, \phi, y}(g,g^\prime) & = \frac{e^{-(1-\lambda)(|g| + |g^\prime|^2) + \phi \Re(g \bar{g}^\prime)} \gpdf(|g|^2 - y) \gpdf(|g^\prime|^2 - y)}{\ZTWis{\lambda}{\phi}{y}}.
		\end{align*}
		This density function is much nicer, in particular it is locally bounded.
		\item  We first note that $\ln \tilde{h}_{\lambda,\phi,y}$ is a degree 4 polynomial in $g,g^\prime$. Consequently it is local lipchitz, that is, 
		\begin{align*}
		&|\ln \tilde{h}_{\lambda,\phi,y}(g,g^\prime) - \ln \tilde{h}_{\lambda,\phi,y}(g_0,g_0^\prime)| \\& \hspace{2cm} \leq C (1+ |g| + |g^\prime| + |g_0| + |g_0^\prime|)^3 (1+ |\lambda| + |\phi| + |y|)( |g - g_0| + |g^\prime - g_0^\prime|).
		\end{align*}
		In particular this means that there exists a large enough constant $C$ depending only on $\sigma$ such that, 
		\begin{align*}
		\forall g,g^\prime : \; \max(|g - g_0|, |g^\prime - g_0^\prime|) \leq R, \; |\ln \tilde{h}_{\lambda,\phi,y}(g,g^\prime) - \ln \tilde{h}_{\lambda,\phi,y}(g_0,g_0^\prime)| \leq \ln(2),
		\end{align*}
		where,
		\begin{align*}
		R & \Mydef \frac{1}{C(1+ |\lambda| + |\phi| + |y|)(1+ |g_0|^3 + |g_0^\prime|^3)}.
		\end{align*}
		We can use this to show that $\tilde{h}_{\lambda,\phi,y}$ is locally bounded. We have: 
		\begin{align*}
		\tilde{h}_{\lambda,\phi,y} (g_0,g^\prime_0) & = \frac{\tilde{h}_{\lambda,\phi,y}(g_0,g_0^\prime)}{\int_{\mathbb{C}^2} \tilde{h}_{\lambda,\phi,y}(g,g^\prime) \diff g \diff g^\prime } \\ 
		& = \left( \int_{\mathbb{C}^2} \exp \left(\ln \tilde{h}_{\lambda,\phi,y}(g,g^\prime) - \ln \tilde{h}_{\lambda,\phi,y}(g_0,g^\prime_0)\right) \diff g \diff g^\prime \right)^{-1} \\
		& \leq \left( \int_{|g-g_0| \leq R, |g^\prime - g_0^\prime| \leq R} \exp \left(\ln \tilde{h}_{\lambda,\phi,y}(g,g^\prime) - \ln \tilde{h}_{\lambda,\phi,y}(g_0,g^\prime_0)\right) \diff g \diff g^\prime \right)^{-1} \\
		& \explain{}{\leq} \left(\frac{1}{2} \int_{|g-g_0| \leq R, |g^\prime - g_0^\prime| \leq R} \diff g \diff g^\prime \right)^{-1} \\
		& = \frac{2}{\pi^2 R^4} \\
		& \leq C(1+|\lambda|^4 + |\phi|^4 + |y|^4) (1 + |g_0|^{12} + |g_0^\prime|^{12}).
		\end{align*}
		\item We begin by computing the log-mgf of $r$: 
		\begin{align*}
		\ln \E e^{t r} & = A-B,
		\end{align*}
		where,
		\begin{align*}
		A & \Mydef  \left( \frac{1}{2\pi }\int_0^\infty \int_0^\infty \int_{-\pi}^\pi e^{(\lambda - 1 + t) r + (\lambda - 1)r^\prime + \phi \sqrt{r r^\prime} \cos(\theta)} \gpdf(r_1-y) \gpdf(r_2 - y) \diff \theta \diff r \diff r^\prime \right) \\ B & \Mydef \ln \ZTWis{\lambda}{\phi}{y}.
		\end{align*}
		We upper bound (A): 
		\begin{align*}
		A & = \frac{1}{2\pi }\int_0^\infty \int_0^\infty \int_{-\pi}^\pi e^{(\lambda - 1 + t) r + (\lambda - 1)r^\prime + \phi \sqrt{r r^\prime} \cos(\theta)} \gpdf(r_1-y) \gpdf(r_2 - y) \diff \theta \diff r \diff r^\prime \\
		& \explain{(a)}{\leq} \int_0^\infty \int_0^\infty  e^{(\lambda - 1 + t+|\phi|) r + (\lambda - 1+|\phi|)r^\prime } \gpdf(r_1-y) \gpdf(r_2 - y)  \diff r \diff r^\prime \\
		& \leq  \int_{-\infty}^\infty \int_{-\infty}^\infty  e^{(|\lambda| + 1 + |t|+|\phi|) (r+r^\prime)  } \gpdf(r_1-y) \gpdf(r_2 - y)  \diff r \diff r^\prime \\
		& \explain{(b)}{=} \exp \left( 2y(|\lambda| + 1 + |t| + |\phi| ) + \sigma^2 (|\lambda| + 1 + |t| + |\phi|)^2 \right).
		\end{align*}
		In the step marked (a), we used the fact the fact that $\sqrt{r r^\prime} \cos(\theta) \leq r+ r^\prime$. In the step marked (b) we used the formula for the MGF of a gaussian distribution. 
		Hence, there exists a universal constant $C$ depending only on $\sigma$ such that:
		\begin{align*}
		\ln A & \leq C \left( 1 + |\lambda|^2 + |\phi|^2 + |y|^2 + |t|^2 \right).
		\end{align*}
		Next we upper bound $(B)$: 
		\begin{align*}
		B & = \ln\left( \frac{1}{2\pi }\int_0^\infty \int_0^\infty \int_{-\pi}^\pi e^{(\lambda - 1 + t) r + (\lambda - 1)r^\prime + \phi \sqrt{r r^\prime} \cos(\theta)} \gpdf(r-y) \gpdf(r^\prime - y) \diff \theta \diff r \diff r^\prime \right) \\& \explain{(c)}{\geq } \frac{1}{2\pi} \int_0^\infty \int_0^\infty \int_{-\pi}^\pi \left( (\lambda + t) r +  \lambda  r^\prime + \phi \sqrt{r r^\prime} \cos(\theta) + \ln \gpdf(r - y) + \ln \gpdf(r^\prime - y) \right)  e^{-r - r^\prime} \diff \theta \diff r \diff r^\prime \\
		& \explain{(d)}{\geq} - C \left(1 + |y|^2 + |\lambda| + |t| \right)
		\end{align*}
		In the step marked (c) we applied Jensen's inequality and in the step marked (d) we used performed the integration involving the moments of the $\Exp{1}$ distribution and used straightforward algebraic bounds. 
		This gives us: 
		\begin{align*}
		\ln \E e^{t r} & \leq C \left( 1 + |\lambda|^2 + |\phi|^2 + |y|^2 + |t|^2 \right).
		\end{align*}
		For notational convenience we define: 
		\begin{align*}
		\kappa \Mydef  1 + |\lambda|^2 + |\phi|^2 + |y|^2.
		\end{align*}
		By Markov's Inequality we have,
		\begin{align*}
		\P (r > x) & \leq \exp \left( C\kappa  - \frac{x^2}{4C} \right).
		\end{align*}
		Setting the tail probability to $\epsilon$ gives us the result: 
		\begin{align*}
		\P \left( r > C \sqrt{1+ |y|^2 + |\phi|^2 + |\lambda|^2} + C \sqrt{\ln \frac{1}{\epsilon}} \right) & \leq \epsilon,
		\end{align*}
		for a suitable constant $C$.
		\item Integrating the tail bound obtained above gives us the moment bound:
		\begin{align*}
		\E r^k & = k \int_0^\infty  x^{k-1} \P(r > x) \diff x  \\ 
		& = k \left( \int_0^{2C \sqrt{\kappa}} x^{k-1} \cdot 1 \diff x + \int_{2C \sqrt{\kappa}}^\infty x^{k-1} \cdot e^{ C\kappa  - \frac{x^2}{4C}}  \right) \\
		&  \leq C_k \cdot  \kappa^{\frac{k-1}{2}}  + \int_{2 C \sqrt{\kappa}}^\infty \exp \left( C\kappa  - \frac{x^2}{4C} \right) \cdot k \cdot x^{k-1}  \diff x \\
		& \explain{(a)}{\leq } C_k \cdot  \kappa^{\frac{k-1}{2}} \\
		& \leq C_k (1+ |\lambda|^k + |\phi|^k + |y|^k).
		\end{align*}
		In the step marked (a), we used the a bound on the truncated gaussian integral given in Lemma \ref{truncated_gauss_integral} in Appendix \ref{misc_appendix}.
		\item Let $\bm T$ be a $2 \times 2$ Hermitian matrix. The characteristic function of the Tilted Wishart distribution evaluated at $\bm T$ is given by $\E e^{\i \ip{\bm T}{ \bm S}}$. 
		Let $\bm w \in \mathbb{C}^2$ be a random vector sampled from the pdf $h_{\lambda, \phi, y}$: 
		\begin{align*}
		\bm w & = \begin{bmatrix} w \\ w^\prime \end{bmatrix} \sim h_{\lambda, \phi, y}.
		\end{align*}
		Using the alternate characterization derived in item (1) of this lemma we have, 
		\begin{align*}
		\E e^{\i \ip{\bm T}{\bm S}} & = \E e^{\i \ip{\bm T} {\bm w \bm w^\UH}}
		\end{align*}
		Consider the spectral decomposition of $\bm T$: 
		\begin{align*}
		\bm T & = \bm B  \begin{bmatrix} t & 0 \\ 0 & t^\prime \end{bmatrix} \bm B^\UH, \; \bm B = \begin{bmatrix} \bm b_1^\UH \\ \bm b_2^\UH \end{bmatrix}.
		\end{align*}
		Since we have, 
		\begin{align*}
		\|\bm T\|^2 & = t^2 + {t^\prime}^2 \implies \max(|t|,|t^\prime|)  \geq \frac{\|\bm T \|}{\sqrt{2}}.
		\end{align*}
		We will assume that infact, 
		\begin{align*}
		|t| & \geq  \frac{\|\bm T \|}{\sqrt{2}}.
		\end{align*}
		Define the random vector:
		\begin{align*}
		\bm z & = \begin{bmatrix} z \\ z^\prime \end{bmatrix} = \bm B^\UH\bm g.
		\end{align*}
		We will often use the polar representation of $\bm z$: 
		\begin{align*}
		\bm z & = \begin{bmatrix} s e^{\i \nu} \\ s^\prime e^{\i \nu^\prime} \end{bmatrix}
		\end{align*}
		Let $d(z,z^\prime)$ denote the density of $\bm z$. This density can be obtained by a simple unitary transformation of the density $h_{\lambda,\phi,y}$. While the exact formula is complicated it is easy to see that it is of the form: 
		\begin{align*}
		d(s e^{\i \nu}, s^\prime e^{i \nu^\prime}) & = \frac{\exp \left( \sum_{k,l: k+l \leq 4} a_{k,l}(\nu,\nu^\prime) s^k {s^\prime}^l\right)}{\ZTWis{\lambda}{\phi}{y}}
		\end{align*}
		The exact formula for the coefficients $a_{k,l}(\nu)$ is not important. It is sufficient to see they satisfy the bound: 
		\begin{align*}
		|a_{k,l}(\nu,\nu^\prime)| & \leq C(1 + |y| + |\lambda| + |\phi| ).
		\end{align*}
		We can now analyze the decay of the characteristic function:
		\begin{align*}
		|\E e^{\i \ip{\bm T} {\bm S}}|  & = |\E \exp \left( \i t |z|^2 + \i t^\prime |z^\prime|^2 \right)| \\& \leq |\E e^ { \i t |z|^2 + \i t^\prime |z^\prime|^2 } \mathbf{1}_{|z| \leq \epsilon, |z^\prime| < R}| + |\E e^ { \i t |z|^2 + \i t^\prime |z^\prime|^2 }\mathbf{1}_{|z| > \epsilon,|z^\prime| < R}| + |\E e^{ \i t |z|^2 + \i t^\prime |z^\prime|^2 }\mathbf{1}_{|z^\prime| > R}| \\
		& \leq \underbrace{\P(|z| \leq \epsilon,|z^\prime| \leq R)}_{(\mathsf{I})} +  \underbrace{|\E \exp \left( \i t |z|^2 + \i t^\prime |z^\prime|^2 \right)\mathbf{1}_{|z| >\epsilon, |z^\prime| < R }|}_{(\mathsf{II})} + \underbrace{\P(|z^\prime|> R)}_{(\mathsf{III})}
		\end{align*}
		In the above display $0< \epsilon < 1 < R$ are parameters which will be chosen later. We analyze the terms $(\mathsf{I}), (\mathsf{II})$  and $(\mathsf{III})$ separately. 
		\begin{description}
			\item[Analysis of $(\mathsf{I})$:] Recall that in part (2) of this Lemma we had shown the density of $\bm w$ is locally bounded: 
			\begin{align*}
			h_{\lambda, \phi,y}(w,w^\prime)  &  \leq C(1+|\lambda|^4 + |\phi|^4 + |y|^4) (1 + |w|^{12} + |w^\prime|^{12}).
			\end{align*}
			The density of $z,z^\prime$, denoted by $d(z,z^\prime)$ is a unitary transformation of the density $h_{\lambda,\phi,\phi}$. Consequently, we have the estimate: 
			\begin{align}
			\label{transformed_density_form}
			d(z,z^\prime) & \leq  C(1+|\lambda|^4 + |\phi|^4 + |y|^4) \cdot R^{12}, \; \forall |z| \leq \epsilon, \; |z^\prime| \leq R.
			\end{align}
			Using this we can easily bound $A$:
			\begin{align*}
			(1) & = \P(|z| \leq \epsilon,|z^\prime| \leq R)  \leq C(1+|\lambda|^4 + |\phi|^4 + |y|^4) \cdot R^{12} \cdot \epsilon^2 \cdot R^2  \\
			& \leq C \cdot (1+|\lambda|^4 + |\phi|^4 + |y|^4) R^{14}\cdot \epsilon^2.
			\end{align*}
			\item[Analysis of $(\mathsf{II})$:] Recall that term $B$ was given by: 
			\begin{align*}
			(\mathsf{II})  &= |\E \exp \left( \i t |z|^2 + \i t^\prime |z^\prime|^2 \right)\mathbf{1}_{|z| >\epsilon, |z^\prime| < R }| \\
			& = \left| \int_{-\pi}^\pi \int_{-\pi}^\pi \int_{0}^R \int_{\epsilon}^\infty \exp(\i t s^2 + \i t^\prime {s^\prime}^2) d(s e^{\i \nu},s^\prime e^{\i \nu^\prime})  s s^\prime \diff s \diff s^\prime \diff \nu \diff \nu^\prime\right| \\
			& =  \frac{1}{|t|}\left| \int_{-\pi}^\pi \int_{-\pi}^\pi \int_{0}^R \int_{\epsilon}^\infty \frac{\partial\exp(\i t s^2 + \i t^\prime {s^\prime}^2)}{\partial s} d(s e^{\i \nu},s^\prime e^{\i \nu^\prime})   s^\prime \diff s \diff s^\prime \diff \nu \diff \nu^\prime\right| \\
			&\explain{(a)}{\leq} (\mathsf{IIa}) + (\mathsf{IIb}).
			\end{align*}
			In the step marked (a), we applied integration by parts and defined the terms $(\mathsf{IIa}),(\mathsf{IIb})$ as follows:
			\begin{align*}
			(\mathsf{IIa}) & =\frac{1}{|t|} \left| \int_{-\pi}^\pi \int_{-\pi}^\pi \int_{0}^R  e^{\i t \epsilon^2 + \i t^\prime {s^\prime}^2} d(\epsilon e^{\i \nu},s^\prime e^{\i \nu^\prime})   s^\prime  \diff s^\prime \diff \nu \diff \nu^\prime \right|. \\
			(\mathsf{IIb}) & = \frac{1}{|t|} \left|\int_{-\pi}^\pi \int_{-\pi}^\pi \int_{0}^R \int_{\epsilon}^\infty e^{\i t s^2 + \i t^\prime {s^\prime}^2} \frac{\partial }{\partial s}d(s e^{\i \nu},s^\prime e^{\i \nu^\prime})   s^\prime \diff s \diff s^\prime \diff \nu \diff \nu^\prime \right|.
			\end{align*}
			The previously obtained bound on $d(z,z^\prime)$ immediately gives the following bound:
			\begin{align*}
			(\mathsf{IIa}) & \leq \frac{C}{|t|} \cdot (1+|\lambda|^4 + |\phi|^4 + |y|^4) \cdot R^{14}.
			\end{align*}
			We can control $(\mathsf{IIb})$ as follows:
			\begin{align*}
			(\mathsf{IIb}) &\explain{(b)}{\leq} \frac{1}{|t|} \left( \int_{-\pi}^\pi \int_{-\pi}^\pi \int_{0}^R \int_{\epsilon}^\infty \left| \frac{\partial }{\partial s}d(s e^{\i \nu},s^\prime e^{\i \nu^\prime}) \right|   s^\prime \diff s \diff s^\prime \diff \nu \diff \nu^\prime\right) \\
			& \explain{(c)}{=}  \frac{1}{|t|} \left(  \int_{-\pi}^\pi \int_{-\pi}^\pi \int_{0}^R \int_{\epsilon}^\infty \left| \sum_{ k+l \leq 4} a_{ij}(\nu,\nu^\prime) s^{k-1} {s^{\prime}}^{l} \right| d(s e^{\i \nu},s^\prime e^{\i \nu^\prime})   s^\prime \diff s \diff s^\prime \diff \nu \diff \nu^\prime\right) \\
			& \leq  \frac{1}{|t|} \left( \frac{1}{\epsilon} \int_{-\pi}^\pi \int_{-\pi}^\pi \int_{0}^R \int_{\epsilon}^\infty \left| \sum_{ k+l \leq 4} a_{ij}(\nu,\nu^\prime) s^{k-1} {s^{\prime}}^{l} \right| d(s e^{\i \nu},s^\prime e^{\i \nu^\prime})   ss^\prime \diff s \diff s^\prime \diff \nu \diff \nu^\prime\right) \\
			& \explain{(d)}{\leq} \frac{C(1+|\lambda|^4 + |\phi|^4 + |y|^4)}{|t|} \cdot \left( \frac{1}{\epsilon} \sum_{k+l \leq 3} \E |z|^{k} {|z^\prime|}^{l} \right) \\
			& \explain{(e)}{\leq} \frac{C(1+|\lambda|^4 + |\phi|^4 + |y|^4)}{|t|} \cdot \left( \frac{1+|\lambda|^2 + |\phi|^2 + |y|^2}{\epsilon}  \right) \\
			\end{align*}
			In the step marked (b) we used the local lipchitz bound on $d$. In the step marked (c) we recalled the formula for the density $d$ (see \eqref{transformed_density_form}). In the step marked (d) we used the bound on the coefficients $a_{k,l}(\nu)$. In the step marked (e) we used the fact that the random vector $\bm z$ is a unitary transformation of $\bm w$ and the third moment of $\bm w$ was bounded in item (4) of this lemma. Combining the bounds on $\mathsf{IIa},\mathsf{IIb}$ we obtain, 
			\begin{align*}
			(\mathsf{II}) & \leq \frac{C(1+|\lambda|^4 + |\phi|^4 + |y|^4)}{\|\bm T\|} \cdot \left( R^{14} + \frac{1+|\lambda|^2 + |\phi|^2 + |y|^2}{\epsilon}  \right).
			\end{align*}
			\item[Analysis of $(\mathsf{III})$:] We have:
			\begin{align*}
			(\mathsf{III}) & =\P \left[ |z^\prime|  \geq R \right] \\
			& \leq  \P \left[ \|\bm z\|  \geq R \right] \\
			& \explain{(f)}{\leq} \P \left[ \|\bm w\|  \geq R \right] \\
			& \explain{}{\leq}  \P \left[ |w|  \geq R \right] + \P \left[ |w^\prime| \geq R  \right] \\
			& \explain{(g)}{\leq} 2 \P \left[ |w| \geq R \right].
			\end{align*}
			In the step marked (f) we used the fact that since $\bm z$ is a unitary transformation of $\bm w$, we have $\|\bm z \| = \| \bm w\|$. In the step marked (g) we used the fact that $|w|,|w|^\prime$ are identically distributed. Finally we set $R$ as: 
			\begin{align*}
			R & = C \sqrt{1+ |y|^2 + |\phi|^2 + |\lambda|^2} + C \sqrt{\ln \frac{1}{\epsilon}},
			\end{align*} 
			and apply the concentration inequality from item (3) of this lemma to obtain: 
			\begin{align*}
			(3) \leq 2 \epsilon.
			\end{align*}
		\end{description}
		Combining the bounds on $(1),(2)$ and $(3)$ and setting $\epsilon = O(1/\|\bm T\|^{1/2})$  gives us the final bound on the characteristic function of $\bm S$: 
		\begin{align*}
		\left|\E e^{\i \ip{\bm \Lambda}{\bm S}} \right | & \leq  C\cdot(1+|\lambda|^{20} + |\phi|^{20} + |y|^{20}) \cdot \frac{\ln^{10}(\|\bm T\|)}{\sqrt{\|\bm T\|}} \\
		& \leq \frac{ C\cdot(1+|\lambda|^{20} + |\phi|^{20} + |y|^{20})}{\|\bm T\|^{\frac{1}{3}}}.
		\end{align*}
		\item The claim $\lambda_{\max} \left( \VTWis{\lambda}{\phi}{y} \right)< \infty$ follows from the moment estimates derived in claim (4) of the lemma. To show that $\lambda_{\min} \left( \VTWis{\lambda}{\phi}{y} \right)>0$ we note that if $\lambda_{\min} \left( \VTWis{\lambda}{\phi}{y} \right) = 0$, we can find a matrix $\bm T $ with $\|\bm T\| =1$ such that $\ip{\bm T}{\bm S}$ is deterministic. If this happens then the characteristic function $\E e^{\i t\ip{\bm S}{\bm T}} = 1$ which contradicts the $O(t^{-\frac{1}{3}})$ decay proved in Claim (5) of this lemma.
	\end{enumerate}
\end{proof}
\section{Analysis of the Variational Problems}
\label{appendix_variational_analysis}
In this section, we study the potential functions involved in the definition of the key functions $\Xi_1(\sigma), \hat{\Xi}_1(\sigma)$ and $\Xi_2(q;\sigma), \hat{\Xi}_2(q;\sigma)$.
Define the two concave potential functions:
\begin{align*}
V_1(\lambda;r) & = \lambda r - \E_{Y} \ln \ZTexp{\lambda}{Y}, \; \lambda \in \mathbb R\\
V_2(\lambda,\phi; q) & = 2 \alpha \lambda + \beta \phi - \E_{Y}
\ln \ZTWis{\lambda}{\phi}{Y}, \; \lambda,\phi \in \mathbb R.
\end{align*}
In this section, we study the two variational problems: 
\begin{align*}
\text{P1: } & \max_{\lambda \in \mathbb{R}} V_1(\lambda), \\
\text{P2: }&  \max_{\lambda,\phi \in \mathbb{R}} V_2(\lambda, \phi; q).
\end{align*}
The analysis in this section will consider an arbitrary distribution on the random variable $Y$. The reason for doing so is to handle the following two cases in a unified way:
\begin{enumerate}
	\item $Y$ is sampled from the empirical distribution of the phase retrieval observations:
	\begin{align*}
	Y & \sim \frac{1}{m} \sum_{i=1}^m \delta_{y_i}.
	\end{align*}
	This case covers the analysis of $\hat{\Xi}_1(\sigma),\hat{\Xi}_2(q;\sigma)$.
	\item $Y = |Z|^2 + \sigma \epsilon$  where $Z \sim \cgauss{0}{1}$ and $\epsilon \sim \gauss{0}{1}$. This case covers the analysis of ${\Xi}_1(\sigma),{\Xi}_2(q;\sigma)$.
\end{enumerate}
We also note that the potential functions $V_1, V_2$ depend on the noise level $\sigma$ even though the dependence is not explicit in our notation. In this section, we consider a fixed $\sigma > 0$ and the universal constants $C$ of this section may depend on $\sigma$. However, they do not depend on the distribution of $Y$. Finally we note that the variation problem P1 is more general than we require in the sense that for the analysis of $\Xi_1, \hat{\Xi}_1$, we can set $r = 1$. The reason for studying this more general variational problem is that we can reduce the analysis of P2 to this more general variational problem. 

\subsection{Analysis of Variational Problem P1}
The following proposition analyzes the variational problem P1 and shows that it has a unique minimizer which is guaranteed to lie in a ball of a certain radius.
\begin{proposition}[Analysis of P1]
	\label{denominator_variational_problem_prop}
	There exists a universal constant $0<C<\infty$ depending only on the noise level $\sigma$ such that: 
	\begin{enumerate}
		\item The following coercivity estimate holds:
		\begin{align*}
		V_1(\lambda;r) & \leq -\frac{r|\lambda|}{2C}, \; \forall |\lambda| \geq C \left( r + \frac{1}{r} \right)(\E |Y|^2 + 1).
		\end{align*} 
		\item All minimizers of the variational problem lie in the compact set: 
		\begin{align*}
		\left\{\lambda: |\lambda| \leq C \left( r + \frac{1}{r} \right)(\E |Y|^2 + 1) \right\}.
		\end{align*}
		\item The function $V_1(\lambda;r)$ is strongly concave on every compact set. Consequently, the variational problem has a unique minimizer. 
	\end{enumerate}
\end{proposition}
\begin{proof}
	Throughout this proof, $C$ refers to a universal constant depending only on $\sigma$ which may change from line to line.
	\begin{enumerate}
		\item We need to show that $V_1$ is coercive, that is: 
		\begin{align*}
		V_1(\lambda;r) \rightarrow - \infty \; \text{as }  |\lambda| \rightarrow \infty.
		\end{align*}
		In order to do so we need to obtain lower bounds on $\ln \ZTexp{\pm |\lambda|}{Y}$. First we consider:
		\begin{align*}
		\ZTexp{|\lambda|}{y} & \explain{(a)}{=} \int_0^\infty e^{-(1-|\lambda|)u}\gpdf(u-y)  \diff u \\
		& = \frac{e^{-\frac{y^2}{2\sigma^2}}}{\sqrt{2\pi \sigma^2}}\int_0^\infty \exp\left(|\lambda| u - \frac{u^2}{2\sigma^2} + \frac{uy}{2\sigma^2} - u\right) \diff u \\
		& \explain{(b)}{=} \frac{|\lambda| e^{-\frac{y^2}{2\sigma^2}}}{\sqrt{2\pi \sigma^2} } \int_0^\infty \exp\left(|\lambda|^2 u \left(1 - \frac{u}{2\sigma^2} \right) + \frac{|\lambda|uy}{2\sigma^2} -|\lambda| u\right) \diff u \\
		& \explain{}{\geq } \frac{|\lambda| e^{-\frac{y^2}{2\sigma^2}}}{\sqrt{2\pi \sigma^2} } \int_{\frac{\sigma^2}{2}}^{\sigma^2} \exp\left(|\lambda|^2 u \left(1 - \frac{u}{2\sigma^2} \right) + \frac{|\lambda|uy}{2\sigma^2} -|\lambda| u\right) \diff u \\
		& \explain{(c)}{\geq}  \frac{|\lambda| e^{-\frac{y^2}{2\sigma^2}}}{\sqrt{2\pi \sigma^2} } \int_{\frac{\sigma^2}{2}}^{\sigma^2} \exp\left(\frac{|\lambda|^2 \sigma^2}{4} -|\lambda| \left(\frac{|y|}{2} + \sigma^2\right)\right) \diff u.
		\end{align*}
		In the step marked (a), we used Definition \ref{Texp_definition}. In the step marked (b), we performed a change of variable $u = |\lambda | u$.  In the step marked (c),  we used the fact that: 
		\begin{align*}
		u \left(1 - \frac{u}{2\sigma^2} \right) & \leq \frac{|\lambda|^2 \sigma^2}{4}, \; \frac{\sigma^2}{2} \leq u \leq  \sigma^2.
		\end{align*} 
		Hence we obtain, for a universal constant $0<C<\infty$ depending only on $\sigma^2$, we have,
		\begin{align*}
		\E_{Y} \ln \ZTexp{|\lambda|}{y} & \geq \ln |\lambda| + \frac{|\lambda|^2}{C} - C |\lambda|(\E|Y|^2 + 1).
		\end{align*}
		Hence, 
		\begin{align}
		\label{coercive_estimate_denom_I}
		V_1(|\lambda|;r)  & \leq | \lambda|(r + C(y^2 + 1)) - \ln |\lambda| - \frac{|\lambda|^2}{C} \nonumber \\
		& \leq - \frac{|\lambda|^2}{2C}, \; \forall |\lambda| \geq 2 C (r + C(\E |Y|^2 + 1)).
		\end{align}
		Next we consider: 
		\begin{align*}
		\ZTexp{-|\lambda|}{y} & = \int_0^\infty e^{-(1+|\lambda|)u}\gpdf(u-y)  \diff u \\
		& = |\lambda| \int_0^\infty e^{-u} \cdot \gpdf \left(\frac{u}{|\lambda|}-y \right) \cdot e^{-\frac{u}{|\lambda|}} \diff u
		\end{align*}
		By Jensen's Inequality, 
		\begin{align*}
		\ln 	\ZTexp{-|\lambda|}{y} & = \ln |\lambda|  + \ln \left( \E_{E \sim \Exp{1}} e^{-\frac{E}{|\lambda|}} \gpdf \left( \frac{E}{|\lambda|} -y \right) \right) \\
		& \geq \ln |\lambda| - \frac{1}{|\lambda|} - \frac{1}{2\sigma^2} \left( \frac{2}{|\lambda|^2} + y^2 - \frac{2y}{|\lambda|} \right) \\
		& \geq \ln |\lambda| - C(y^2 + 1), \; \forall  \; |\lambda| \geq 1.
		\end{align*}
		Consequently we have,
		\begin{align}
		\label{coercive_estimate_denom_II}
		V_1(-|\lambda|;r) & \leq - r |\lambda| - \ln | \lambda | + C(\E |Y|^2+ 1), \; \forall \;  |\lambda| \geq 1 \nonumber \\& \leq - \frac{r |\lambda|}{2}, \; \forall |\lambda| \geq \frac{2C(\E |Y|^2 + 1))}{r} + 1.
		\end{align}
		Combining the estimates in \eqref{coercive_estimate_denom_I} and \eqref{coercive_estimate_denom_II}, we obtain that for a large enough constant $C$,
		\begin{align*}
		V_1(\lambda;r) & \leq -\frac{r|\lambda|}{2C}, \; \forall |\lambda| \geq C \left( r + \frac{1}{r} \right)(\E |Y|^2 + 1).
		\end{align*} 
		\item We observe that,
		\begin{align*}
		V_1(0;r) & = - \E_Y \ln \E_{E \sim \Exp{1}} \gpdf(E-Y) \\
		& \geq -  \ln \left( \frac{1}{\sqrt{2\pi \sigma^2}} \right) \\
		& \geq -C.
		\end{align*}
		Hence,
		\begin{align*}
		|\lambda| \geq C \left( r + \frac{1}{r} \right)(\E |Y|^2 + 1) \implies V_1(\lambda;r) \leq V_1(0;r).
		\end{align*}
		Hence,
		\begin{align*}
		\arg \min_{\lambda \in \mathbb R} V_1(\lambda;r) \subset \left\{\lambda: |\lambda| \leq C \left( r + \frac{1}{r} \right)(\E |Y|^2 + 1) \right\}.
		\end{align*}
		\item In the light of item (2) of the lemma, it is sufficient to study the variational problem: 
		\begin{align*}
		\max_{|\lambda| \leq R} V_1(\lambda;r), \; R \Mydef C \left( r + \frac{1}{r} \right)(\E |Y|^2 + 1). 
		\end{align*}
		In order to show uniqueness of the solution it is sufficient to show that $V_1(\lambda;r)$ is strictly concave on $|\lambda| \leq R$, for which it is sufficient to check that: 
		\begin{align*}
		\min_{|\lambda| \leq R} \frac{\diff^2 V_1}{\diff \lambda ^2}(\lambda) < 0 \Leftrightarrow \frac{\diff^2}{\diff \lambda ^2} \E_Y \ln \ZTexp{\lambda}{Y}  > 0.
		\end{align*}
		Note that by convexity we have, 
		\begin{align*}
		\frac{\diff^2}{\diff \lambda ^2} \E_Y \ln \ZTexp{\lambda}{Y}  \geq 0.
		\end{align*}
		In order to obtain a strict inequality, suppose there is a $\lambda_0$ such that: 
		\begin{align*}
		\frac{\diff^2}{\diff \lambda ^2} \E_Y \ln \ZTexp{\lambda}{Y} \bigg \lvert_{\lambda = \lambda_0}  = 0 \Leftrightarrow \E_{Y} \left[ \frac{\E E^2 e^{\lambda_0 E} \gpdf(E-Y) }{\E e^{\lambda_0 E} \gpdf(E-Y)} - \left( \frac{\E E e^{\lambda_0 E} \gpdf(E-Y) }{\E e^{\lambda_0 E} \gpdf(E-Y)} \right)^2 \right] & = 0.
		\end{align*}
		Recalling Definition \ref{Texp_definition},
		\begin{align*}
		\VTexp{\lambda_0}{Y} & \explain{a.s.}{=} 0.
		\end{align*}
		However this contradicts the decay rate property of the characteristic function of Tilted Exponential distribution proved in Lemma \ref{lemma_texp_properties} in Appendix \ref{texp_properties_appendix} since the amplitude of the characteristic function of deterministic random variables is constant.
	\end{enumerate}
\end{proof}
\subsection{Analysis of Variational Problem P2}
The following proposition analyzes the variational problem P2 and shows that it has a unique minimizer which is guaranteed to lie in a ball of a certain radius. 
\begin{proposition}[Analysis of P2]
	\label{nr_variational_prop}
	Suppose that $q \in (0,1)$. There exists a universal constant $0<C<\infty$ depending only on the noise level $\sigma$ such that: 
	\begin{enumerate}
		\item The following coercivity estimate holds:
		\begin{align*}
		V_2(\lambda, \phi;q) & \leq - \frac{(1-q)}{2C} \cdot (|\lambda| + |\phi|),  \; |\lambda| + |\phi| \geq C \left( 1+q + \frac{1}{1-q} \right) (\E Y^2 + 1).
		\end{align*}
		\item All minimizers of the variational problem lie in the compact set: 
		\begin{align*}
		\left\{(\lambda,\phi) \in \mathbb{R}^2: |\lambda| + |\phi| \leq C \left( 1+q + \frac{1}{1-q} \right)(\E |Y|^2 + 1) \right\}.
		\end{align*}
		\item The function $V_2(\lambda,\phi;q)$ is strongly concave on any compact set. Consequently, the variational problem has a unique minimizer. 
	\end{enumerate}
\end{proposition}
\begin{proof}
	Throughout this proof, $C$ refers to a universal constant depending only on $\sigma$ which may change from line to line. It will be helpful to write the variational problem in the following matrix notation. Define,
	\begin{align*}
	\bm \Lambda & = \begin{bmatrix} \lambda & \frac{\phi}{2} \\ \frac{\phi}{2} & \lambda \end{bmatrix}
	\end{align*}
	Then the problem P2 can be rewritten as: 
	\begin{align*}
	\text{P2: } & \max_{\bm \Lambda }V_2(\bm \Lambda), \; V_2(\bm \Lambda) = \ip{\bm \Lambda}{\bm Q} - \E_Y \ln \E_{\bm g \sim \cgauss{\bm 0}{\bm I_2}} \exp(\ip{\bm \Lambda}{\bm g \bm g^\UH}) \gpdf(Y-|g_1|^2) \gpdf(Y-|g_2|^2),
	\end{align*}
	where,
	\begin{align*}
	\bm Q & = \begin{bmatrix} 1 & q \\ q & 1\end{bmatrix}.
	\end{align*}
	To obtain the above display, we recalled the definition of the normalizing constant of the Tilted Wishart Distribution (Definition \ref{titled_wishart_definition}).
	\begin{enumerate}
		\item In order to obtain a coercivity estimate we need to lower bound $\ln \ZTWis{\lambda}{\phi}{y}$.  Our lower bound will depend only on the spectrum of $\bm \Lambda$. We consider the eigendecomposition of $\bm \Lambda$: 
		\begin{align*}
		\bm \Lambda & = \begin{bmatrix} \bm b_1^\UH \\ \bm b_2^\UH \end{bmatrix} \cdot  \begin{bmatrix} \gamma_1 & 0 \\ 0 &\gamma_2\end{bmatrix} \cdot \begin{bmatrix} \bm b_1 & \bm b_2 \end{bmatrix}
		\end{align*}
		In the above display $\gamma_1 \geq \gamma_2$ are the ordered eigenvalues of $\bm \Lambda$. We have the following lower bound on $\ln \ZTWis{\lambda}{\phi}{y}$: 
		\begin{align*}
		&\ln \ZTWis{\lambda}{\phi}{y}  = \ln \E \exp(\ip{\bm \Lambda}{\bm g \bm g^\UH}) \gpdf(y-|g_1|^2) \gpdf(y-|g_2|^2) \\
		& = \ln \E \exp(\gamma_1 |g_1|^2 + \gamma_2 |g_2|^2) \gpdf(y-|\bm b_1^\UH \bm g|^2) \gpdf(y-|\bm b_2^\UH \bm g|^2) \\
		& = \ln \E \exp\left(\gamma_1 |g_1|^2 + \gamma_2 |g_2|^2 - \frac{1}{2\sigma^2}(2 y^2 - 2y(|\bm b_1^\UH \bm g|^2+|\bm b_2^\UH \bm g|^2) +|\bm b_1^\UH \bm g|^4+|\bm b_2^\UH \bm g|^4 )\right) \\
		& \explain{(a)}{\geq} \ln \E \exp \left( \gamma_1 |g_1|^2 + \gamma_2 |g_2|^2 - \frac{1}{2\sigma^2}(2 y^2 - 2y(|g_1|^2 + |g_2|^2) + 2|g_1|^4 + 2|g_2|^4 ) \right) \\
		& \geq  \ln \E \exp(\gamma_1 |g_1|^2) \gpdf\left( \frac{y}{\sqrt{2}} - \sqrt{2}|g_1|^2\right) + \ln \E \exp(\gamma_2 |g_2|^2) \gpdf\left( \frac{y}{\sqrt{2}} - \sqrt{2}|g_2|^2\right) - \frac{y^2}{2 \sigma^2} \\
		& = \ln \E_{E \sim \Exp{1}} \exp(\gamma_1 E) \gpdf\left( \frac{y}{\sqrt{2}} - \sqrt{2}E\right) + \ln \E_{E \sim \Exp{1}} \exp(\gamma_2E) \gpdf\left( \frac{y}{\sqrt{2}} - \sqrt{2} E\right) - \frac{y^2}{2 \sigma^2}
		\end{align*}
		In the step marked (a), we used the fact that,
		\begin{align*}
		\|\bm B \bm g\|_2^2 = \|\bm g\|^2, \; \|\bm B \bm g\|_4^4  \leq \|\bm g\|_2^4 \leq 2(|g_1|^4 + |g_2|^4).
		\end{align*}
		Next note that,
		\begin{align*}
		\ip{\bm \Lambda}{ \bm Q} & \leq \gamma_1 \lambda_1(\bm Q) + \gamma_2 \lambda_2(\bm Q),
		\end{align*}
		where $\lambda_1(\bm Q) \geq \lambda_2(\bm Q)$ are the ordered eigenvalues of $\bm Q$. It is easy to check that $\lambda_1(\bm Q) = 1+q$ and $\lambda_2(\bm Q) = 1-q$ which means,
		\begin{align*}
		\ip{\bm \Lambda}{ \bm Q} & \leq \gamma_1(1+q) + \gamma_2(1-q).
		\end{align*} 
		This gives us,
		\begin{align*}
		V_2(\bm \Lambda;q) & \leq  \gamma_1  (1+q) -  \E\ln \E e^{\gamma_1 E} \gpdf\left( \frac{Y}{\sqrt{2}} - \sqrt{2}E\right)   + \gamma_2 (1-q) - \E \ln \E e^{\gamma_2 E} \gpdf\left( \frac{y}{\sqrt{2}} - \sqrt{2}E\right)  + \frac{\E Y^2}{\sigma^2}
		\end{align*}
		Utilizing the coercivity estimates from Proposition \ref{denominator_variational_problem_prop}, we obtain,
		\begin{align*}
		\left( \gamma_1 \cdot (1+q) -  \E_Y\ln \E e^{\gamma_1 E} \gpdf\left( \frac{Y}{\sqrt{2}} - \sqrt{2}E\right)  \right)  \leq - \frac{(1+q) \cdot  |\gamma_1|}{2C}, \\
		\left( \gamma_2 (1-q) -  \E_Y\ln \E e^{\gamma_2 E} \gpdf\left( \frac{Y}{\sqrt{2}} - \sqrt{2}E\right)  \right)  \leq - \frac{(1-q) |\gamma_2|}{2C},
		\end{align*}
		for all:
		\begin{align*}
		    |\gamma_1| &\geq C \left( 1+q + \frac{1}{1+q} \right) \left( \E Y^2 + 1 \right), \\
		    |\gamma_2| &\geq C \left( (1-q) + \frac{1}{1-q} \right) \left( \E Y^2 + 1 \right).
		\end{align*}
		Since,
		\begin{align*}
		\|\bm \Lambda\|^2 & \geq t \implies \max(\gamma_1^2, \gamma^2_2) \geq \frac{t}{2}, \; \forall t,
		\end{align*} 
		we obtain,
		\begin{align*}
		V_2(\bm \Lambda;q) & \leq - \frac{1-q}{2C} \cdot \|\bm \Lambda\|, \; \|\bm \Lambda\| \geq C \left( 1+q + \frac{1}{1-q} \right) (\E Y^2 + 1).
		\end{align*}
		This is equivalent to the estimate: 
		\begin{align*}
		V_2(\lambda, \phi;q) & \leq - \frac{1-q}{2C} \cdot (|\lambda| + |\phi|),  \; |\lambda| + |\phi| \geq C \cdot \left( 1+q + \frac{1}{1-q} \right) (\E Y^2 + 1).
		\end{align*}
		This concludes the proof of item (1) in the statement of the lemma.
		\item The proof is analogous to the proof of item (2) in Proposition \ref{denominator_variational_problem_prop}.
		\item The proof is analogous to the proof of item (3) in Proposition \ref{denominator_variational_problem_prop}.
	\end{enumerate}
\end{proof}
\section{Background on Characteristic Functions}
\label{CF_appendix}
In this section we collect some basic facts about characteristic functions (CF). Most of these results are taken from Chapter XV of \citet{feller2008introduction}.  The characteristic function is simply the Fourier transform of the probability density function. 
\begin{definition}[Characteristic Function] Let $f$ be a probability density function on $\mathbb R$. Then the characteristic function of $f$ is defined as: 
	\begin{align*}
	\psi(t) & = \int_\mathbb{R} e^{\i t x} f(x) \diff x.
	\end{align*}
\end{definition}
If the characteristic function is absolutely integrable, the probability density function can be recovered from it using the Fourier inversion formula.
\begin{theorem}[Fourier Inversion of CFs] \label{fourier_inversion_CF} Let $\psi$ be the CF of density $f$. Then, \begin{align*}
	f(x)  & = \frac{1}{2\pi} \int_{\mathbb{R}}  \psi(t) e^{-\i t x} \diff t.
	\end{align*}
\end{theorem}
The moments of the PDF can be recovered from the Taylors expansion of the CF.
\begin{theorem}[Taylors Series of CF]  \label{taylors_cf} Let $X$ be a random variable with probability density function $f$. Let $\psi$ be the CF of $f$. We have, for any $ t \in \mathbb R$.
	\begin{align*}
	\left| \psi(t) - \left( 1+ \sum_{k=1}^{n-1} \frac{ \E X^k}{k!} \cdot (\i t)^k \right) \right| & \leq \frac{\E |X|^n t^n}{n!}
	\end{align*}
\end{theorem}
The following bound on CFs will be useful in the proofs of the local central limit theorems.
\begin{lemma}[Bounds on CF] Let $\psi$ be a multivariate CF and suppose that, there exists $0<c<1$ and $b>0$ such that,
	\begin{align}
	|\psi(\bm t)| & \leq c \; \forall \;  \|\bm t\| > b.
	\label{cf_bound_requirement}
	\end{align}
	Then, for any $\|\bm t\| \leq b$ we have,
	\begin{align*}
	|\psi(\bm t)| & \leq 1-\frac{1-c^2}{8b^2}\|\bm t\|^2.
	\end{align*}
\end{lemma}
\begin{proof}
	A univariate version is given as Theorem 1 in Chapter 1 of \citet{petrov2012sums}. A multivariate version is given as Theorem 1.8.13 in \citet{ushakov2011selected}.
\end{proof}
Finally we state a Multivariate Berry-Eseen bound due to \citet{bhattacharya1975errors}.
\begin{theorem}[A Multivariate Berry-Eseen Bound, \citep{bhattacharya1975errors}] \label{berry_eseen_theorem} Let $X_1, X_2 \dots X_n$ be independent random vectors in $\mathbb R^k$. Suppose that:
	\begin{align*}
	\E X_i = \bm 0, \; \frac{1}{n} \sum_{i=1}^n \E X_i X_i^\UT  = \bm I_k.
	\end{align*}
	Define: 
	\begin{align*}
	\rho_3 & \Mydef \frac{1}{n} \sum_{i=1}^n \E \|X_i\|^3.
	\end{align*}
	Then, there exists a universal constant $C_k$ depending only on the dimension $k$,  such that for any bounded, lipchitz function $f$ we have, 
	\begin{align*}
	\left| \E f \left( \frac{\sum_{i=1}^n X_i}{\sqrt{n}} \right) - \E_{Z \sim \gauss{\bm 0}{\bm I_k}}  f(Z) \right| & \leq \frac{C_k \cdot \rho_3 \cdot  (\|f\|_\infty + \|f\|_{\text{Lip}})}{\sqrt{n}}.
	\end{align*}
\end{theorem}
\section{Some Miscellaneous Results}\label{misc_appendix}
This appendix collects some miscellaneous facts and results that are useful in our analysis. The first is a classical correlation inequality. 
\begin{fact}[Chebychev Association Inequality, \citep{boucheron2013concentration}] 	\label{chebychev_association}Let $A,B$ be r.v.s and $B \geq 0$. Suppose $f,g$ are two non-decreasing functions. Then, $\E[B]\E[B f(A) g(A)] \geq \E[f(A)B]\E[g(A) B]$.
\end{fact}
The following collects some useful properties of Modified Bessel Function of the first kind. These results can be found in the standard references \citep{abramowitz1965handbook,watson1995treatise}. Item (5) of the following is relatively less known and is due to \citet[Appendix A]{watson1983statistics}.

\begin{fact}[Properties of Modified Bessel Function of the First Kind]\label{bessel_properties}
	For $x \in \mathbb R$, the Modified Bessel Function of the First Kind, denoted by, $I_0(x)$ is defined as: 
	\begin{align*}
	I_0(x) & \Mydef \frac{1}{2\pi} \int_{-\pi}^\pi e^{x \cos(\theta)}.
	\end{align*}
	It satisfies the following properties:
	\begin{enumerate}
		\item $I_0(x)$ is an increasing function on $x \geq 0$ and $I_0(0) = 1$. 
		\item $I_0(x)$ is an even function.
		\item There exists a universal constant $C$ such that,
		\begin{align*}
		I_0(x) & \leq  \frac{C e^x}{\sqrt{x}}, \; \forall x \geq 0.
		\end{align*}
		\item $I_0$ is infinitely differentiable.
		\item The function $\frac{I_0^\prime}{I_0}$ is an increasing concave function with,
		\begin{align*}
		\frac{I_0^\prime(0)}{I_0(0)} = 0, \; \lim_{x \rightarrow \infty} \frac{I_0^\prime(x)}{I_0(x)} = 1,
		\end{align*}
		and,
		\begin{align*}
		\frac{\diff}{\diff z} \left(\frac{I_0^\prime(z)}{I_0(z)} \right) \bigg |_{z =0} & = \frac{1}{2}.
		\end{align*}
	\end{enumerate}
\end{fact}
The following lemma is about a bivariate Gaussian integral.
\begin{lemma} \label{bivariate_gauss_integral} Let $Z_1, Z_2$ be distributed as:
\begin{align*}
    \begin{bmatrix} Z_1 \\ Z_2 \end{bmatrix} & \sim \gauss{\begin{bmatrix} 0 \\ 0\end{bmatrix}}{\begin{bmatrix}1 & \rho \\ \rho & 1 \end{bmatrix}}. 
\end{align*}
Then the integral:
\begin{align*}
    J(a,b) & \Mydef \E_{Z_1,Z_2} \gpdf[1](a - Z_1) \gpdf[1](b-Z_2),
\end{align*}
is given by:
\begin{align*}
    J(a,b) & = \frac{1}{4\pi \sqrt{1-\rho^2/4}} \cdot \exp\left( - \frac{a^2 + b^2  - \rho ab}{4(1-\rho^2/4)} \right).
\end{align*}
\end{lemma}
\begin{proof}
    Note that $J(a,b)$ is the Joint PDF of the random variables $(A,B)$ with distribution:
    \begin{align*}
        A = Z_1 + \epsilon_1, \; B = Z_2 + \epsilon_2, \; \begin{bmatrix} Z_1 \\ Z_2 \end{bmatrix} & \sim \gauss{\begin{bmatrix} 0 \\ 0\end{bmatrix}}{\begin{bmatrix}1 & \rho \\ \rho & 1 \end{bmatrix}}, \; \epsilon_1 \sim \gauss{0}{1}, \; \epsilon_2 \sim \gauss{0}{1}.
    \end{align*}
    We can directly find the distribution of $A,B$ from this description:
    \begin{align*}
        \begin{bmatrix}A \\ B \end{bmatrix} & \sim \gauss{\begin{bmatrix} 0 \\ 0\end{bmatrix}}{\begin{bmatrix}2 & \rho \\ \rho & 2 \end{bmatrix}}.
    \end{align*}
    Hence by the formula for the bivariate Gaussian pdf,
    \begin{align*}
        J(a,b) & = \frac{1}{4\pi \sqrt{1-\rho^2/4}} \cdot \exp\left( - \frac{a^2 + b^2  - \rho ab}{4(1-\rho^2/4)} \right).
    \end{align*}
\end{proof}
We will also find the following bound on truncated Gaussian integrals useful.
\begin{lemma}[Truncated Gaussian Integrals] \label{truncated_gauss_integral} Suppose that $a,A > 0$ and $k \in \mathbb N$. Then, we have,
	\begin{align*}
	\int_a^\infty x^{k} e^{-\frac{x^2}{2A^2}} \diff x & \leq C_k \cdot A \cdot (A^k + a^k) \cdot  e^{-\frac{a^2}{2}}.
	\end{align*}
	In the above display $C_k$ is a universal constant depending only on $k$.
\end{lemma}
\begin{proof}
	Let us first consider the case when $A=1$. Then we have, 
	\begin{align*}
	\int_{a}^\infty x^k e^{-\frac{x^2}{2}} \diff x & \explain{(a)}{=} 2^{\frac{k-1}{2}}\int_{a^2/2}^\infty u^{\frac{k-1}{2}} e^{-u} \diff u \\
	& \explain{(b)}{=} 2^{\frac{k-1}{2}} \cdot e^{-\frac{a^2}{2}} \cdot \int_{0}^\infty \left( x+ \frac{a^2}{2}\right)^{\frac{k-1}{2}} e^{-x} \diff x \\
	&\explain{(c)}{\leq} 2^{k-1} \cdot e^{-\frac{a^2}{2}} \cdot \int_{0}^\infty \left( x^{\frac{k-1}{2}}+ \frac{a^{k-1}}{2^{\frac{k-1}{2}}}\right) e^{-x} \diff x \\
	& \leq 2^{k-1} \cdot e^{-\frac{a^2}{2}} \cdot \left( \sqrt{\int_0^\infty e^{-x} x^{k-1} \diff x} + \frac{a^{k-1}}{2^{\frac{k-1}{2}}} \right) \\
	& \leq 2^{k-1} \cdot e^{-\frac{a^2}{2}} \cdot \left( \sqrt{(k-1)!} + \frac{a^{k-1}}{2^{\frac{k-1}{2}}} \right) \\
	& \leq C_k (1+a^{k}) e^{-\frac{a^2}{2}}
	\end{align*}
	In the step marked (a), we substituted $u = x^2/2$ in the step marked (b) we substituted $u = x+a$. In the step marked (c) we used the inequality $(a+b)^k \leq 2^k(a^k + b^k), \; a,b \geq 0 \; k \geq 0$. Making the substitution $x = Ax$ in the above bound gives us: 
	\begin{align*}
	\int_a^\infty x^{k} e^{-\frac{x^2}{2A^2}} \diff x & \leq C_k \cdot A \cdot (A^k + a^k) \cdot  e^{-\frac{a^2}{2}}.
	\end{align*}
	This concludes the proof. 
\end{proof}
The following lemma contains a useful upper bound on $\E|G|^{-\frac{1}{2}}$ where $G \sim \gauss{0}{1}$. 
\begin{lemma}[Fractional Moments of Gaussian Distribution] \label{fractional_moments} Let $G \sim \gauss{\mu}{\sigma^2}$. Then we have,
	\begin{align*}
	\E \frac{1}{\sqrt{|G|}} & \leq \frac{4}{\sqrt{\mu}}.
	\end{align*}
\end{lemma}
\begin{proof}
	We have,
	\begin{align*}
	\E \frac{1}{\sqrt{|G|}} & \leq \E \frac{1}{\sqrt{|G|}} \mathbf{1}_{|G| \leq 0.5 |\mu|} + \E \frac{1}{\sqrt{|G|}} \mathbf{1}_{|G|> 0.5 |\mu|} \\
	& \leq \int_{-0.5|\mu|}^{0.5 |\mu|} \frac{1}{\sqrt{|x|}} \cdot \frac{1}{\sqrt{2\pi \sigma^2}} e^{-\frac{(x-\mu)^2}{2\sigma^2}} \diff x + \frac{\sqrt{2}}{\sqrt{\mu}} \\
	& \leq \frac{e^{-\frac{\mu^2}{8\sigma^2}}}{\sqrt{2\pi \sigma^2}} \cdot \int_{-0.5 |\mu|}^{0.5 |\mu|} \frac{1}{\sqrt{x}} \diff x + \frac{\sqrt{2}}{\sqrt{\mu}} \\
	& = \frac{2 \sqrt{|\mu|} e^{-\frac{\mu^2}{8 \sigma^2}}}{\sqrt{\pi \sigma^2}} + \frac{\sqrt{2}}{\sqrt{\mu}} \\
	& \leq \frac{4}{\sqrt{\mu}}
	\end{align*}
	In the last step we used the fact that $\max_{x \geq 0} \sqrt{x} e^{-x} \leq \frac{1}{\sqrt{2 e}}$. 
\end{proof}
	
\end{document}